\documentclass[14pt]{article}
\pdfoutput=1
\usepackage[utf8]{inputenc}
\usepackage[letterpaper,top=2cm,bottom=2cm,left=3cm,right=3cm,marginparwidth=1.75cm]{geometry}

\usepackage{amsfonts,amsmath,amssymb,amsthm,enumerate}
\usepackage{authblk}

\newtheorem{define}{Definition}
\newtheorem{lemma}{Lemma}
\newtheorem{corollary}{Corollary}
\newtheorem{remark}{Remark}
\newtheorem{theorem}{Theorem}
\newtheorem{assumption}{Assumption}

\usepackage{graphicx}
\usepackage[colorlinks=true, allcolors=blue]{hyperref}

\title{Wavelet characterization of exponentially weighted Besov space with dominating mixed smoothness and its application to function approximation}

\author[1]{
Yoshihiro Kogure}
\affil[1]{\small Department of Mathematical Informatics, Graduate School of Information Science and Technology, University of
Tokyo, 7-3-1 Hongo, Bunkyo-ku, Tokyo, 113-8656, Japan}

\author[1]{
Ken’ichiro Tanaka }

\date{}
\begin{document}
\newcommand\blfootnote[1]{%
  \begingroup
  \renewcommand\thefootnote{}\footnote{#1}%
  \addtocounter{footnote}{-1}%
  \endgroup
}
\maketitle
\begin{abstract}
Although numerous studies have focused on normal Besov spaces, limited studies have been conducted on exponentially weighted Besov spaces. Therefore, we define exponentially weighted Besov space $VB_{p,q}^{\delta,w}(\mathbb{R}^d)$ whose smoothness includes normal Besov spaces, Besov spaces with dominating mixed smoothness, and their interpolation. Furthermore, we obtain wavelet characterization of $VB_{p,q}^{\delta,w}(\mathbb{R}^d)$. Next, approximation formulas such as sparse grids are derived using the determined formula. The results of this study are expected to provide considerable insight into the application of exponentially weighted Besov spaces with mixed smoothness.
\end{abstract}

\section{Introduction}

\blfootnote{E-mail address: kogure-yoshihiro484@g.ecc.u-tokyo.ac.jp (Y. Kogure), kenichiro@mist.i.u-tokyo.ac.jp (K. Tanaka).}

$\quad \; \; \;$In this study, the coefficients of wavelet expansion were used to characterize exponentially weighted Besov space with dominating mixed smoothness. The characterization of Besov quasinorm by wavelet is a critical topic of research in the Besov space. Numerous studies have been conducted on the normal Besov space \cite{Triebel06,Sawano18} and the Besov space with dominating mixed smoothness \cite{Hansen10,Triebel19}. A decomposition theory on the Besov space with $A_{loc}^p$--weight was proposed \cite{Izuki09,Izuki12}. 
In this study, we focused on the exponentially weighted Besov space with dominating mixed smoothness for obtaining an accurate high-dimensional function approximation formula. First, the background and motivations for high-dimensional approximations were provided. We start discussion with sparse grids typically used in high-dimensional approximation.

Sparse grids, which were first introduced by Smolyak \cite{HM04}, are used to reduce computational cost in high-dimensional numerical methods. In sparse grids for function approximation, a function space is decomposed into hierarchical subspaces, and subsequently finite hierarchical subspaces are selected as approximation spaces. The smoothness of function spaces plays a crucial role in constructing efficient and refined sparse grids \cite{Griebel09,Griebel07}. 

Orthogonal basis in certain Banach spaces $V$ allow the decomposition of a space into infinite subspaces such as $V=\bigoplus_{ \bar{j} \in \mathbb{N}_0^d,\bar{m} \in \mathbb{Z}^d} W_{\bar{j},\bar{m}}$. Many Banach spaces, such as $L^p$ spaces, Sobolev spaces, Besov spaces, and Triebel--Lizorkin spaces, are decomposed by orthogonal basis. Several wavelets form the unconditional basis of those spaces and expression equivalent to their norms is given by using the coefficients of wavelet expansion. Wavelet characterization is the most critical property for constructing sparse grids. 

Let a function $u\in V$ which allows decomposition $ V = \bigoplus_{\bar{j}\in \mathbb{N}_0^d, \bar{m} \in \mathbb{Z}^d} W_{\bar{j},\bar{m}}$ be expressed as $u = \sum_{\bar{j} \in \mathbb{N}_0^d,\bar{m} \in \mathbb{Z}^d} u_{\bar{j},\bar{m}}$, where $u_{\bar{j},\bar{m}} \in W_{\bar{j},\bar{m}}$. Assuming we calculate the approximation $\tilde{u}$ of $u$ with a certain family of finite index sets $\{ G_{\bar{j}} \}_{\bar{j} \in \mathbb{N}_0^d}$ by $\sum_{\bar{j} \in \mathbb{N}_0^d} \sum_{\bar{m} \in G_{\bar{j}}} u_{\bar{j},\bar{m}}$, the approximation error becomes the following:

\begin{equation}
\label{error_intro}
         ||u-\tilde{u}||_V = ||\sum_{\bar{j} \in \mathbb{N}^d_0,\bar{m} \in \mathbb{Z}^d \backslash G_{\bar{j}}}u_{\bar{j},\bar{m}}||_V\leq \sum_{\bar{j} \in \mathbb{N}^d_0,\bar{m} \in \mathbb{Z}^d \backslash G_{\bar{j}}} ||u_{\bar{j},\bar{m}}||_V
\end{equation}

The set $\{ G_{\bar{j}} \}_{\bar{j} \in \mathbb{N}_0^d}$ should be selected so that approximation error (\ref{error_intro}) is minimized. This minimization depends on Banach space $V$ and its decomposition $V= \bigoplus_{\bar{j}\in \mathbb{N}_0^d, \bar{m} \in \mathbb{Z}^d} W_{\bar{j},\bar{m}}$. In this study, we selected Banaha space $V$ as exponentially weighted Besov space $VB^{\delta,w}_{p,q}(\mathbb{R}^d)$ (Section \ref{new_def}) and decomposition $\bigoplus_{\bar{j}\in \mathbb{N}_0^d, \bar{m} \in \mathbb{Z}^d} W_{\bar{j},\bar{m}}$ as subspaces spanned by each of one wavelet in $\{ \psi_{\bar{j},\bar{m}} \}_{\bar{j} \in \mathbb{N}_0^d , \bar{m} \in \mathbb{Z}^d}$ (Section \ref{characwave}).

The spatial decay of a function should be considered for obtaining refined sparse grids. As an example, assume we want to obtain a superior approximation of a function that has an exponential decay at infinity with compact support basis. A basis near the origin is more crucial than a basis far away from the origin. The importance of basis depends on where the support of basis exists on $\mathbb{R}^d$. The motivation to consider smoothness and decay is that the eigenfunctions of Hamilton operator $H$ defined as follows:
\begin{equation}
    H := 
    -\frac{1}{2} \sum_{i=1}^{N} \Delta_i
    -\sum_{i=1}^{N} \sum_{\nu =1}^{K} \frac{Z_{\nu}}{|x_i-a_{\nu}|}
    +\frac{1}{2}
    \sum_{i,j=1 i\neq j}^{N} \frac{1}{|x_i-x_j|}
\end{equation}
is known to belong to $H^{1,0}_{mix} \cap \bigcap_{\vartheta < 3/4} H^{\vartheta,1}_{mix}$ \cite{Kreusler12} and show exponential decay \cite{O'Connor73}. Here, $H^{1,0}_{mix}$ and $H^{\vartheta,1}_{mix}$ is defined by the following:
\begin{equation}
    H^{1,0}_{mix} := 
    \{
    u \in L^2(\mathbb{R}^d) :
    \sum_{\bar{\alpha} \in \mathcal{A}} \|D^{\bar{\alpha}} u \|_2 < \infty
    \}
\end{equation}
\begin{equation}
    H^{\vartheta,1}_{mix} := (H^1,H^{1,1}_{mix})_{\vartheta,2}
\end{equation}
where $\mathcal{A} = \{ (\bar{\alpha}_1,\bar{\alpha}_2,\cdots,\bar{\alpha}_N): \bar{\alpha}_i \in \mathbb{N}^3_0, \: \alpha_{i,1}+\alpha_{i,2}+\alpha_{i,3} \leq 1\}$ and
\begin{equation}
    H^{1}
    :=
    \{
    u \in L^2(\mathbb{R}^d) :
    \sum_{
    \bar{\beta} \in \mathbb{N}_0^{3N},  |\bar{\beta}|_1 \leq 1 } \|D^{\bar{\beta}} u \|_2 < \infty
    \}
\end{equation}
\begin{equation}
    H^{1,1}_{mix}
    :=
    \{
    u \in L^2(\mathbb{R}^d) :
    \sum_{\substack{\bar{\alpha} \in \mathcal{A} \\
    \bar{\beta} \in \mathbb{N}_0^{3N}, \: |\bar{\beta}|_1 \leq 1 }} \|D^{\bar{\beta}} D^{\bar{\alpha}} u \|_2 < \infty
    \}
\end{equation}
The Hamilton operator, $H$, appears in the Schrödinger equation such that multiple electrons interact. The dimension of this system is $3N$, where $N$ denotes the number of electrons. With full grid calculation, considerable curse of dimensionality occurs if $N$ is so large. A remedy for this difficulty is to select basis according to the importance of basis. A study obtained sparse grids in the exponentially weighted Sobolev space with dominating smoothness\cite{And12}. This study extends the results to the Besov space.

First, we define a novel Besov space with dominating mixed smoothness. We used the Besov space because of the availability of various norm-equivalent expressions. Especially, the characterizations of quasinorm by wavelet expansion plays an essential role. However, smoothness like $H^{\vartheta,1}_{mix}$ 
is yet to be formulated in the Besov space. In classical Besov space
$B^s_{p,q}(\mathbb{R}^d)$ theory, parameter $s$ denotes smoothness or differentiability. From the perspective of definition of Besov space based on Fourier transform (\ref{besov}), this smoothness parameter $s$ controls the convergence rate of $L^p$-norm of certain frequency components, that is, $\| f_j \|_p = o(2^{-js})$ where $f_j$ is a function whose frequency domain 
supp$\: \mathcal{F}[f_j] = \{\xi \in \mathbb{R}^d \: | \: 2^j \lesssim |\xi|_2 \lesssim  2^{j+1}\}$. To extend this approach to the Besov space with dominating mixed smoothness, the function should be decomposed into a sequence of functions whose support after Fourier transformation is narrow. In \cite{Hansen10}, to define the Besov space with dominating mixed smoothness, a function is decomposed so that the convergence rate of 
$\| f_{\bar{j}} \|_p \approx o(2^{-s |\bar{j}|_1)})$, where $\bar{j}=(j_1,\cdots,j_d) \in \mathbb{N}_0^d$ and $\text{supp} \; \mathcal{F}[f_{\bar{j}}] = \{\xi=(\xi_1,\cdots,\xi_d) \in \mathbb{R}^d \: | \: 2^{j_i} \lesssim |\xi_i| \lesssim  2^{j_i+1} \text{ for all }i=1,\cdots,d \}$ ($\mathcal{F}[f_j]$ is Fourier transform of $f_j$). We consider a broad framework by generalizing $|\cdot|_1$-norm to norm $\delta(\cdot)$ on $\mathbb{N}_0^d$ (Definition \ref{deltaa} or Definition \ref{defVB}) so that $\| f_{\bar{j}}\|_p \approx o(2^{-\delta(\bar{j})})$

In the other strategy that involves obtaining more refined sparse grids, a weighted Besov space suitable to our settings was introduced. By specifying the decay of a function, a suitable choice of basis was obtained according to the location of the support of basis. When the weight function grows or decreases with a polynomial rate, the classical theory of Besov spaces can be easily applied. However, when the weight function grows or decreases with an exponential rate, such a theory cannot be applied. The Hardy--Littlewood maximal operator is not bounded in $L_p^w$, where the weight function $w$ is exponential. In \cite{Rychkov01}, Rychkov constructed the Besov space with $A^p_{loc}$--weight, which is a wider weight class than exponential weight, by introducing a local reproducing formula and Hardy--Littlewood local maximal operator. Thus, an exponentially weighted Besov space was obtained.

The contributions of this study can be listed as follows:
\begin{enumerate} 
    \item  As a generalization of smoothness, the Besov space, which allows interpolation between the normal Besov space, the Besov space with dominating mixed smoothness, and their interpolation are given. (Section \ref{section3}, Section \ref{new_def} and Section \ref{interpo})

    \item The wavelet characterization of the exponentially weighted Besov space with dominating mixed smoothness is provided. (Section \ref{section5} and Section \ref{characwave})

    \item By using wavelet characterization, both smoothness and exponential decay are considered in the approximation formula. (Section \ref{section8})
\end{enumerate}

\section{Notations}
First, we detail notations used in this paper.
\begin{enumerate}
    \item $\bar{j}$ is av $d$-dimensional vector, that is, $\bar{j}=(j_1,j_2,\cdots,j_d)$. In the case $\bar{1}$, $\bar{1}$ represents $d$-dimensional vector whose components are $1$ ,that is, $\bar{1}:=(1,1,\cdots,1)$.
    \item $\bar{j} \bar{k}$ represents the componentwise multiplications, that is, $\bar{j} \bar{k}=(j_1 k_1,j_2 k_2,\cdots,j_d k_d)$.
    \item Fourier transform $\mathcal{F}$ and its inverse $\mathcal{F}^{-1}$ is are defined by the following equation:
    \begin{equation}
        \mathcal{F} f(\xi)
        :=
        \frac{1}{(2 \pi)^{d/2}} \int_{\mathbb{R}^d}
        f(x) e^{-i \xi \cdot x} dx
        \qquad
        \mathcal{F}^{-1} f(x)
        :=
        \frac{1}{(2 \pi)^{d/2}} \int_{\mathbb{R}^d}
        f(\xi) e^{i x \cdot \xi} d \xi
    \end{equation}
    For simplicity, $\widehat{f}$ is also used to represent $\mathcal{F}f$.

    \item Convolution of two functions $f$ and $g$ (denoted by $f*g(x)$ ) is defined by
    \begin{equation}
        f*g(x) := \int_{\mathbb{R}^d} f(x-y) g(y) dy
    \end{equation}
    where $f$ and $g$ are functions on $\mathbb{R}^d$.
    \item $\langle f, g \rangle$ is $L^2$-product of $f,g \in L^2$.
    \begin{equation}
        \langle f, g \rangle
        :=
        \int_{\mathbb{R}^d} f(x) g(x) dx
    \end{equation}

    \item Dot notation in function argument like $f(\cdot)$ or $e^{\cdot}$ is used to simplify the notations like $f(x)$ or $e^{x}$.
    \item $|\cdot|_1$ ($\ell^1$-norm) and $|\cdot|_{\infty}$ (infinite norm) are norms on $\mathbb{R}^d$.
    \begin{equation*}
        |\bar{j}|_1:=|j_1|+|j_2|+\cdots+|j_d|
        \qquad
        |\bar{j}|_{\infty}:= \max_{i=1\cdots d} |j_i|
    \end{equation*}
    
    $\delta$ denotes an arbitrary norm on $\mathbb{R}^d$.

    \item For $z \in \mathbb{C}$, $\Im z$ represents the imaginary part of $z$.
    
    \item Let $A$ and $B$ be positive real numbers. $A \approx B$ reveals that positive real numbers $c_1$ and $c_2$ exist such that $c_1 A \leq B \leq c_2 A$. Sometimes, $A$ and $B$ are replaced by function spaces. In this case, $A \approx B$ denote $\| \cdot \|_A \approx \| \cdot \|_B$ where $\|\cdot \|_A$ and $\| \cdot \|_B$ are their norms.
    
    \item Let $A$ and $B$ be positive real numbers. $A \lesssim B$ and $A \gtrsim B$ reveals positive real numbers $c_1$ and $c_2$ exist such that $c_1 A \leq B$ and $A \geq c_2 B$ each.
    
    \item Let $w(x)$ be a weight function (\textcolor{black}{that is, positive and locally integrable function}) and $Q$ be a domain. Here, $w(Q)$ is defined by the following expression:
    \begin{equation}
        w(Q) := \int_{Q} w(x) dx
    \end{equation}

     \item $\| \cdot \|_p$ is a normal $L^p$ norm. $\|\cdot \|_{L^p_w}$ is a weighted $L^p$ norm defined by the following expression:
    \begin{equation}
        \|f\|_{L^p_w}
        :=
        \left(
        \int
        |f(x)|^p w(x) dx
        \right)^{1/p}
    \end{equation}

    \item $\mathcal{S}$ denotes the Schwartz space and $\mathcal{S}'$ is its topological dual.

    \item $\mathcal{S}_e$ is a collection of exponentially decreasing $C^{\infty}$ functions whose topology is induced by seminorms defined in (\ref{test_exp}). Here, $\mathcal{S}_e'$ is its topological dual.

    \item $VB^{\delta,w}_{p,q}(\mathbb{R}^d)$ is an exponentially weighted Besov space with dominating smoothness (Definition \ref{defVB} in Section \ref{new_def}).
    
    \item $X_b^{\bar{s}}(\mathbb{R}^d)$ is a test function space (Definition \ref{test_exp} in Section \ref{new_def}).
    
\end{enumerate}
\section{Besov space and extension of its smoothness}
\label{section3}
$\qquad$In this section, a generalized smoothness of the Besov space is considered. We define a new Besov space $VB_{p,q}^{\delta}(\mathbb{R}^d)$, which includes the normal Besov space and the Besov space with mixed smoothness and their interpolation. This generalization exhibits considerable freedom in selecting the smoothness parameter. In the first subsection, we detail this generalization by observing a construction of the Besov space with mixed smoothness. In the second subsection, we detail a definition of the Besov space with generalized smoothness $VB_{p,q}^{\delta}(\mathbb{R}^d)$.

\subsection{Normal Besov space and Besov space with mixed smoothness}

$\qquad$The Besov space is used to decompose a function into a sequence of functions according to the frequency domain and control their convergence rate of $L^p$-norm. In the theory of the Besov space, the dyadic resolution of unity is used to obtain this decomposition. To define the normal Besov space, $B^{s}_{p,q}(\mathbb{R}^d)$ and Besov space with dominating mixed smoothness $MB^{\bar{s}}_{p,q}(\mathbb{R}^d)$, two decomposition methods are used: (a)circular decomposition $\{\phi_j \}_{j \in \mathbb{N}_0}$ (b)rectangular decomposition $\{\Phi_{\bar{j}}\}_{\bar{j}\in \mathbb{N}_0^d}$. 

\begin{enumerate}[(a)]

\item Circular decomposition is expressed as follows. Let $\phi_0(\xi) \in \mathcal{S}$ such that
$\text{supp} \phi_0 \subset \{\xi \in \mathbb{R}^d \; : \; |\xi|<2 \} \qquad$ and $\qquad \phi_0(\xi)=1 \quad if \quad |\xi|\leq 1$
and $\phi_j(\xi) = \phi_0(2^{-j} \xi)- \phi_0(2^{-j+1} \xi)$. Then, $\{\phi_j \}_{j \in \mathbb{N}_0}$ satisfy the properties (i) and (ii).
\begin{enumerate}[(i)]

\item $\text{supp} \phi_j \subset \{ \xi \in \mathbb{R}^d \: : \: 2^{j-1} \leq |\xi| \leq 2^{j+1} \} \quad \text{for} \quad j=1,2,3,\cdots \quad \text{and} \quad \text{supp} \phi_0 \subset \{ \xi \in \mathbb{R}^d \: : \:  |\xi| \leq 2 \}$

\item $\quad \sum_{j=0}^{\infty} \phi_j(\xi) = 1$
\end{enumerate}
Here, (ii) reveals that constant 1 is decomposed by a sequence of $C_0^{\infty}$ functions whose support dyadically increases, and $\{\phi_j \}_{j \in \mathbb{N}_0}$ is the dyadic resolution of unity. 

\item Rectangular decomposition $\{\Phi_{\bar{j}}\}_{\bar{j}\in \mathbb{N}_0^d}$ is expressed using the one-dimensional dyadic resolution of unity $\{\phi_j \}_{j \in \mathbb{N}_0}$. Summation of all $\{\Phi_{\bar{j}}\}_{\bar{j}\in \mathbb{N}_0^d}$ is equal to $1$.
\begin{equation}
\label{Psi}
    \Phi_{\bar{j}}(\xi_1, \cdots, \xi_d) = \prod_{i=1}^d \phi_{j_i} (\xi_i)
\quad \text{and} \quad
\sum_{\bar{j}\in \mathbb{N}^{d}}\Phi_{\bar{j}} (\xi)
=
\prod_{i=1}^{d} \sum_{j_i \in \mathbb{N}_0}\phi_{j_i}(\xi_i)
=1
\end{equation}
\end{enumerate}
The difference between circular decomposition $\{\phi_j \}_{j \in \mathbb{N}_0}$ and rectangular decomposition $\{\Phi_{\bar{j}}\}_{\bar{j}\in \mathbb{N}_0^d}$ is the fineness of decomposition. This result leads to a difference between the normal Besov space $B^{s}_{p,q}(\mathbb{R}^d)$ and Besov space with mixed smoothness $MB^{\bar{s}}_{p,q}(\mathbb{R}^d)$. The definitions of $B^{s}_{p,q}(\mathbb{R}^d)$ and $MB^{\bar{s}}_{p,q}(\mathbb{R}^d)$ are given so that each quasinorms defined below is finite.

\begin{equation}
\label{besov}
    \left\|f|B_{p,q}^{s}(\mathbb{R}^d)\right\|
:=
\left(\sum_{j=0}^{\infty}2^{sjq}
||\phi_{j}(\mathcal{D})f||_p^q\right)^{1/q}
\end{equation}

\begin{equation}
\label{mixedbesov}
    \left\|f|MB_{p,q}^{\bar{s}}(\mathbb{R}^d)\right\|
:=
\left(\sum_{\bar{j}\in \mathbb{N}_0^d}2^{\bar{s} \cdot \bar{j}q}
||\Phi_{\bar{j}}(\mathcal{D})f||_p^q\right)^{1/q}
\end{equation}
where $\phi_j (\mathcal{D})=\mathcal{F}^{-1} \phi_j \mathcal{F}$ and $\Phi_{\bar{j}} (\mathcal{D})=\mathcal{F}^{-1} \Phi_{\bar{j}} \mathcal{F}$.

The volume of supp$\: \phi_i$ reveals the same order of volume of all supp$\: \Phi_{\bar{j}}$ for $|\bar{j}|_{\infty}=i$. Therefore, $B_{p,q}^{s}(\mathbb{R}^d)$ can be defined by using a rectangular decomposition $\{\Phi_{\bar{j}}\}_{\bar{j}\in \mathbb{N}_0^d}$ instead of circular decomposition $\{\phi_j \}_{j \in \mathbb{N}_0}$ with infinite norm $|\cdot|_{\infty}$. This is justified subsequently (Lemma \ref{new_old}). To perform a mathematical treatment of this intuition, a novel Besov space $VB^{\delta}_{p,q}$ is expressed as a generalization of smoothness in the next subsection. This generalization allow us an interpolation between the normal Besov space $B^{s}_{p,q}(\mathbb{R}^d)$ and Besov space with mixed smoothness $MB^{\bar{s}}_{p,q}(\mathbb{R}^d)$. For more details, refer to Lemma \ref{new_old} and Lemma \ref{interpolation} in Section \ref{interpo}

\subsection{Extension of smoothness \texorpdfstring{$VB^{\delta}_{p,q}(\mathbb{R}^d)$}{Lg}
}

\begin{define}
\label{deltaa}
$VB^{\delta}_{p,q}(\mathbb{R}^d)$ is a collection of $f\in \mathcal{S}'$ whose norm 
\begin{equation}
    \left\|f|VB_{p,q}^{\delta}(\mathbb{R}^d)\right\|
:=
\left(\sum_{\bar{j}\in \mathbb{N}_0^d}2^{\delta(\bar{j}) q}
||\Phi_{\bar{j}}(\mathcal{D})f||_p^q\right)^{1/q}
\end{equation}
is finite, where $\delta$ is a norm on $\mathbb{N}_0^d$ and $\Phi_{\bar{j}}(\mathcal{D}) := \mathcal{F}^{-1} \Phi_{\bar{j}} \mathcal{F}$.
\end{define}
Norm $\delta$ provides a flexibility of smoothness. Thus, $\delta$ controls the convergence rate of $||\Phi_{\bar{j}}(\mathcal{D})f||_p^q$ better compared to that of $MB^{\bar{s}}_{p,q}(\mathbb{R}^d)$. As a special case, $VB^{\delta}_{p,q}(\mathbb{R}^d)$ is equal to $MB^{\bar{s}}_{p,q}(\mathbb{R}^d)$ when $\delta(\bar{j})=\bar{s} \cdot \bar{j}$.Here, $VB^{\delta}_{p,q}(\mathbb{R}^d)$ also turns out to be an extension of $B^{s}_{p,q}(\mathbb{R}^d)$ by Theorem \ref{infty_norm}, which is proved after obtaining the wavelet characterization. Although Theorem \ref{infty_norm} referes to weighted case, it also holds for the nonweighted case.

\section{Exponentially weighted Besov space}
\label{new_def}
$\qquad$In addition to generalized smoothness, we consider the exponentially weighted Besov space. Here, the exponential weight indicates that the weight function $w(x)$ increases or decreases exponentially.
\begin{equation}
\label{weight_assumption}
    0<w(x) \lesssim w(x-y) e^{c_w |y|_1} \quad \text{for all} \quad x,y \in \mathbb{R}
\end{equation}
where $c_w$ is a positive constant that depends on $w$. We assume $w$ is a positive and locally integrable function. The theoretical framework of the exponentially weighted Besov space was proposed in \cite{Schott98} by introducing following exponentially decreasing test function space $\mathcal{S}_e$:
\begin{equation}
\label{test_exp}
    \mathcal{S}_e
    :=\{ \phi \in C^{\infty}(\mathbb{R}^d) \; | \;
    q_N(\phi) := \sup_{x \in \mathbb{R}^d}  e^{N|x|_1} \sum_{|\alpha| \leq N} |\partial^{\alpha} \phi(x)| < \infty
    \}
\end{equation}
The topology of $\mathcal{S}_e$ is induced by the following seminorms $\{ q_N \}_{N \in \mathbb{N}}$. The exponentially weighted Besov space is considered to be a subspace of distribution $\mathcal{S}_e'$. However, $\mathcal{S}_e$ is not appropriate for analyzing a space with wavelets whose smoothness is limited (that is, such wavelets are not included in $\mathcal{S}_e$). Therefore, we introduce a test function space $X_{b}^{\bar{s}}(\mathbb{R}^d)$ with limited smoothness. The following definition of $X_{b}^{\bar{s}}(\mathbb{R}^d)$ is inspired by \cite{Hansen10}.

\begin{define}
Let $b \in \mathbb{R}$ and $\bar{s} \in \mathbb{R}^d$.
$X_{b}^{\bar{s}}(\mathbb{R}^d)$ is a subspace of $MB^{\bar{s}}_{2,2}(\mathbb{R}^d)$ such that
\begin{equation}
     X_{b}^{\bar{s}}(\mathbb{R}^d)
    =
    \{
    \varphi \in MB^{\bar{s}}_{2,2}(\mathbb{R}^d) \: : \:
    \sum_{0\leq \bar{\alpha} \leq \bar{s}}
    \sum_{0\leq \bar{\beta} \leq \bar{s}}
    \left(
    \int_{\mathbb{R}^d} \left| e^{b  |x|_1} x^{\bar{\beta}} D^{\bar{\alpha}} \varphi(x) \right|^2 dx \right)^{1/2}<\infty \} 
\end{equation}
\end{define}

Therefore, we define exponentially weighted Besov space $VB_{p,q}^{\delta,w}(\mathbb{R}^d)$. Here, nonweighted version $VB_{p,q}^{\delta}(\mathbb{R}^d)$ as an extension of smoothness of the Besov space has already been defined in the previous section.

\begin{define}
\label{defVB}
Let $w$ be a weight function of (\ref{weight_assumption}) and $\delta$ be a norm on $\mathbb{N}_0^d$. Let $1 \leq p,q \leq \infty$, $c_w/p < b$, $L \in \mathbb{N}_0$ and $\delta(\bar{j}) <  L|\bar{j}|_1$ for all $\bar{j} \in \mathbb{N}_0^d$. Here, $VB_{p,q}^{\delta,w}(\mathbb{R}^d)$ is a space defined by the following quasinorm.
\begin{equation}
\label{vbesov}
    \left\|f|VB_{p,q}^{\delta,w}(\mathbb{R}^d)\right\|
:=
\left(\sum_{j=0}^{\infty} 2^{\delta(\bar{j})q}
\|\boldsymbol{\psi}_{\bar{j}}*f\|_{L_p^w}^q\right)^{1/q}
\qquad
f \in X_{b}^{\bar{s}}(\mathbb{R}^d)'
\end{equation}
where $\boldsymbol{\psi}_{\bar{j}}$ is defined by the equation:
\begin{equation}
\label{bolpsi}
\boldsymbol{\psi}_{\bar{j}}(x)
    := \prod_{i=1}^{d} \psi_{j_i}^{i}(x_i)
    \quad
    \text{and}
    \quad
    \psi_{j_i}^{i}(x_i) := 2^{j_i} \psi^{i}(2^{j_i} x_i)
\end{equation}
where $\psi^{i}(\cdot) := \psi_0^{i}(\cdot)-2^{-1}\psi_0^{i}(2^{-1}\cdot)$ and $\psi_0^{i} \in X_b^{L+1}(\mathbb{R})$ such that $\int_{\mathbb{R}} \psi_0^{i} * \psi_0^{i} (x) dx = 1 $ and $(D^s \widehat{\psi^i})(0)=0$ for $s=0,1,...,L$.
\end{define}
A big difference of (\ref{vbesov}) in terms of the definition.\ref{defVB} from (\ref{besov}) and (\ref{mixedbesov}) is as follows:  the test functions ($\boldsymbol{\psi}_{\bar{j}}$ and $\phi_j(\mathcal{D})= \mathcal{F}^{-1} \phi_j$, $\Phi_{\bar{j}}(\mathcal{D})= \mathcal{F}^{-1} \Phi_{\bar{j}}$)convoluted with $f$. Therefore, we cannot use $\Phi_{\bar{j}}(\mathcal{D})$ as test functions because $\mathcal{F}^{-1} \Phi_{\bar{j}}$ is not in $X^{\bar{s}}_b(\mathbb{R}^d)$

Here, we define the normal exponentially weighted Besov space (weighted version of $B^{s}_{p,q(\mathbb{R}^d)}$) as follows:

\begin{define}
    Let $w$ be a weight function of (\ref{weight_assumption})and $1 \leq p,q \leq \infty$, $c_w/p < b$. $B_{p,q}^{s,w}(\mathbb{R}^d)$ is a space defined by the following quasinorm:
\begin{equation}
\label{wnbesov}
    \left\|f|B_{p,q}^{s,w}(\mathbb{R}^d)\right\|
:=
\left(\sum_{j=0}^{\infty} 2^{sjq}
\|\Psi_{j}*f \|_{L_p^w}^q \right)^{1/q}
\qquad
f \in X_{b}^{\bar{s}}(\mathbb{R}^d)'
\end{equation}
where $\Psi(\cdot) := \Psi_0(\cdot)-2^{-d}\Psi_0(2^{-1}\cdot)$ and $\Psi_0 \in X_b^{\bar{L}+\bar{1}}(\mathbb{R}^d)$ such that $s \leq \min\{ \bar{L} \}$ and $\int_{\mathbb{R}^d} \Psi_0 * \Psi_0(x) dx = 1 $ and $(D^{\bar{s}} \widehat{\psi^i})(0)=0$ for $s=0,1,...,L$.
\end{define}

This study obtained the wavelet characterization of $VB_{p,q}^{\delta,w}(\mathbb{R}^d)$. Some technical results are moved to appendix for readability. Therefore, these technical results are presented in the appendix for simplicity. In the next section, we state the results on the equivalence relations of $VB^{\delta,w}_{p,q}(\mathbb{R}^d)$ quasinorm.

\section{Equivalence relation of \texorpdfstring{$VB^{\delta,w}_{p,q}(\mathbb{R}^d)$}{Lg}}
\label{section5}
$\qquad$ We detailed two equivalent relations (\ref{equi1}) and (\ref{equi2}). Especially, (\ref{equi1}) details the definition of $VB^{\delta,w}_{p,q}(\mathbb{R}^d)$ is independent of the choice of test functions $\{\boldsymbol{\psi}_{\bar{j}}\}_{\bar{j} \in \mathbb{N}_0^d}$ in definition.\ref{defVB}. By contrast, (\ref{equi2}) and subsequent (\ref{large2}) are about the Pettere maximal operator and are completely technical preparations for the proof of wavelet characterization in Section.\ref{characwave}.

\begin{theorem}
\label{equivalence}
Let $f \in VB_{p,q}^{\delta,w}(\mathbb{R}^d)$. 
Let $b > 0$ such that $c_w/p < b$ and $\varphi_0^{i},\psi_0^{i} \in X_b^{L+1}(\mathbb{R}^d)$ and $\varphi^{i}(\cdot) := \varphi_0^{i}(\cdot)-2^{-1}\varphi_0^{i}(2^{-1}\cdot)$, $\psi^{i}(\cdot) := \psi_0^{i}(\cdot)-2^{-1}\psi_0^{i}(2^{-1}\cdot)$ such that $\int_{\mathbb{R}} \varphi_0^{i} * \varphi_0^{i} (x) dx = 1 $, $\int_{\mathbb{R}} \psi_0^{i} * \psi_0^{i} (x) dx = 1 $, $(D^s \widehat{\varphi^i})(0)=0$ and $(D^s \widehat{\psi^i})(0)=0$ for $s=0,1,...,L$ where $L$ is a natural number such that $\delta(\bar{j}) <  L|\bar{j}|_1$ for all $\bar{j} \in \mathbb{N}_0^d$. Define
$\boldsymbol{\varphi}_{\bar{j}}$ and $\boldsymbol{\psi}_{\bar{j}}$ by the expression:
\begin{equation}
    \boldsymbol{\varphi}_{\bar{j}}(x)
    := \prod_{i=1}^{d} \varphi_{j_i}^{i}(x_i)
    \qquad
    \boldsymbol{\psi}_{\bar{j}}(x)
    := \prod_{i=1}^{d} \psi_{j_i}^{i}(x_i)
\end{equation}
and
\begin{equation}
    \phi_{j_i}^{i}(x_i) := 2^{j_i} \phi^{i}(2^{j_i} x_i)
    \quad
    \text{and}
    \quad
    \psi_{j_i}^{i}(x_i) := 2^{j_i} \psi^{i}(2^{j_i} x_i)
\end{equation}
Then,
\begin{equation}
\label{equi1}
    \left(\sum_{\bar{j}\in \mathbb{N}^d}2^{q \delta( \bar{j})}
||\boldsymbol{\varphi}_{\bar{j}}*f||_{L_p^{w}}^q\right)^{1/q}
\approx
\left(\sum_{\bar{j}\in \mathbb{N}^d}2^{q \delta( \bar{j})}
||\boldsymbol{\psi}_{\bar{j}}*f||_{L_p^{w}}^q\right)^{1/q}
\end{equation}
\end{theorem}

\begin{proof}
Therefore, proving the following equations is sufficient:
\begin{equation}
    \left(\sum_{\bar{j}\in \mathbb{N}^d}2^{q \bar{s} \cdot \bar{j}}
||\boldsymbol{\varphi}_{\bar{j}}*f||_{L_p^{w}}^q\right)^{1/q}
\lesssim
\left(\sum_{\bar{j}\in \mathbb{N}^d}2^{q \bar{s} \cdot \bar{j}}
||\boldsymbol{\psi}_{\bar{j}}*f||_{L_p^{w}}^q\right)^{1/q}
\end{equation}
The inverse inequality can be proved using the same procedure. Because of Lemma.\ref{kernel} in Appendix \ref{appendix_2}, $\boldsymbol{\zeta}_{\bar{j}} \in X_{b}^{\bar{L}+\bar{1}}(\mathbb{R}^d)$ exists such that
\begin{equation}
    f = \sum_{\bar{j} \in \mathbb{N}_0^d} \boldsymbol{\zeta}_{\bar{j}}*\boldsymbol{\psi}_{\bar{j}}*f \qquad \text{in} \quad X_{b}^{\bar{L}}(\mathbb{R}^d)'
\end{equation}
Then,
\begin{eqnarray}
    \boldsymbol{\varphi}_{\bar{k}}*f(x)
    &=&
    \sum_{\bar{j} \in \mathbb{N}_0^d} \boldsymbol{\varphi}_{\bar{k}}* \boldsymbol{\zeta}_{\bar{j}}*\boldsymbol{\psi}_{\bar{j}}*f(x) \nonumber \\
    &=&
    \sum_{\bar{j} \in \mathbb{N}_0^d}
    \int_{\mathbb{R}^d} \boldsymbol{\varphi}_{\bar{k}}* \boldsymbol{\zeta}_{\bar{j}}(y) \boldsymbol{\psi}_{\bar{j}}*f(x-y) dy 
\end{eqnarray}
Let $I(\bar{j},\bar{k})$ be an index set defined by $I(\bar{j},\bar{k}):=\{ i \: |\: k_i \leq j_i\}$.
Define $m_{\bar{j},\bar{k},L}(y)$ by the equation:
\begin{equation}
\label{m}
    m_{\bar{j},\bar{k},L}(y) 
    = 
    \prod_{i \in I(\bar{j},\bar{k})} 2^{-k_i}
    (1+|2^{k_i} y_i|)^{L} e^{2^{k_i}b'|y_i|}
    \prod_{i \in I(\bar{j},\bar{k})^c} 2^{-j_i}
    (1+|2^{j_i} y_i|)^{L} e^{2^{j_i}b'|y_i|}
\end{equation}
for some $c_w/p<b'<b$. 
By Lemma.\ref{tech},
\begin{eqnarray}
    |\boldsymbol{\varphi}_{\bar{k}}*f(x)| &\lesssim&
    \sum_{\bar{j} \in \mathbb{N}_0^d} 2^{-L|\bar{j}-\bar{k}|_1}
    \int_{\mathbb{R}^d}
    \frac{|\boldsymbol{\psi}_{\bar{j}}*f(x-y)|}{m_{\bar{j},\bar{k},L}(y)} dy \nonumber \\
    &=&
    \sum_{\bar{j} \in \mathbb{N}_0^d} 2^{-L|\bar{j}-\bar{k}|_1}
    m_{\bar{j},\bar{k},L}^{-1} * (|\boldsymbol{\psi}_{\bar{j}}*f|)(x)
\end{eqnarray}
Recall the assumption on $w$ given in (\ref{weight_assumption})
\begin{eqnarray}
\label{norm1}
    \|\boldsymbol{\varphi}_{\bar{k}}*f(x)\|_{L^p_w} 
    &\lesssim&
    \sum_{\bar{j} \in \mathbb{N}_0^d} 2^{-L|\bar{j}-\bar{k}|_1}
    \| m_{\bar{j},\bar{k},L}^{-1} * (|\boldsymbol{\psi}_{\bar{j}}*f|) \|_{L^p_w} \nonumber \\
    &\lesssim&
    \sum_{\bar{j} \in \mathbb{N}_0^d} 2^{-L|\bar{j}-\bar{k}|_1}
    \| (m_{\bar{j},\bar{k},L}^{-1}e^{c_w|\cdot|/p}) * (|\boldsymbol{\psi}_{\bar{j}}*f w^{1/p}|) \|_{L^p} \nonumber \\
    &\leq&
    \sum_{\bar{j} \in \mathbb{N}_0^d} 2^{-L|\bar{j}-\bar{k}|_1}
    \| m_{\bar{j},\bar{k},L+1}^{-1}e^{c_w|\cdot|/p} \|_{L^1}
    \|\boldsymbol{\psi}_{\bar{j}}*f w^{1/p} \|_{L^p} \nonumber \\
    &\lesssim&
    \sum_{\bar{j} \in \mathbb{N}_0^d} 2^{-L|\bar{j}-\bar{k}|_1}
    \|\boldsymbol{\psi}_{\bar{j}}*f \|_{L^p_w}
\end{eqnarray}
In the third inequality, the special cases of young inequality $\|g*h\|_p \leq \|g\|_1 \|h\|_p$ ($g \in L^1$, $h \in L^p$) is used. In the final inequality, the boundedness of $\| m_{\bar{j},\bar{k},L}^{-1} e^{c_w|\cdot|_1/p} \|_{L^1}$ derived from following inequality is used.
\begin{eqnarray}
\label{mbound}
\| m_{\bar{j},\bar{k},L+1}^{-1} (\cdot) e^{c_w|\cdot|_1/p} \|_{L^1}
&=&
\int_{\mathbb{R}^d} \frac{e^{c_w|x|_1/p}}{m_{\bar{j},\bar{k},L+1}(x)} dx
\leq
\int_{\mathbb{R}^d} \frac{e^{c_w|2^{\min(\bar{j},\bar{k})}x|_1/p}}{m_{\bar{0},\bar{0},L+1}(x)} dx \nonumber \\
&\leq&
\int_{\mathbb{R}^d} \frac{e^{c_w  |x|_1/p}}{m_{\bar{0},\bar{0},L+1}(x)} dx \nonumber \\
&<& \infty
\end{eqnarray}
where $\min(\bar{j},\bar{k}) = (\min(j_1,k_1),\min(j_2,k_2),\cdots ,\min(j_d,k_d))$ and condition $c_w/p < b'<b$ is used.
Choose $\kappa>0$ so that $\delta(\bar{j}) <  (L-\kappa)|\bar{j}|_1$ for all $\bar{j} \in \mathbb{N}_0^d$.
\begin{eqnarray}
\sum_{\bar{k} \in \mathbb{N}_0^d} 2^{q \delta(\bar{k})}
\|\boldsymbol{\varphi}_{\bar{k}}*f(x)\|_{L^p_w}^q
&\lesssim&
\sum_{\bar{k} \in \mathbb{N}_0^d}
2^{q \delta( \bar{k} )}
\left(\sum_{\bar{j} \in \mathbb{N}_0^d} 2^{-L|(\bar{j}-\bar{k})|_1}
    \|\boldsymbol{\psi}_{\bar{j}}*f \|_{L^p_w} \right)^q \nonumber \\
&\lesssim&
\sum_{\bar{k} \in \mathbb{N}_0^d}
2^{q\delta(\bar{k})}
\sum_{\bar{j} \in \mathbb{N}_0^d} 2^{-q(L-\kappa)|\bar{j}-\bar{k}|_1}
    \|\boldsymbol{\psi}_{\bar{j}}*f \|_{L^p_w}^q \nonumber \\
&\leq&
\sum_{\bar{j} \in \mathbb{N}_0^d} 2^{q \delta(\bar{j})} \|\boldsymbol{\psi}_{\bar{j}}*f \|_{L^p_w}^q 
\sum_{\bar{k} \in \mathbb{N}_0^d}
2^{-q(L-\kappa)|\bar{j}-\bar{k}|_1} 2^{q\delta(\bar{k}-\bar{j})} \nonumber \\
&\lesssim&
\sum_{\bar{j} \in \mathbb{N}_0^d} 2^{q\delta( \bar{j})} \|\boldsymbol{\psi}_{\bar{j}}*f \|_{L^p_w}^q
\end{eqnarray}
\end{proof}

The next result is regarding the maximal operator defined by (\ref{supoperator}). An operator of this type is called the Peetre maximal operator and is highly useful in the analysis of the Besov space. Here, (\ref{supoperator}) is a weighted version of the Peetre maximal operator proposed in \cite{Rychkov01}.
\begin{theorem}
\label{supequi}
$f \in VB_{p,q}^{\delta,w}$. 
Let $0<b'<b$ such that $c_w/p <b'< b$ and  $\varphi_0^{i} \in X_b^{L+1}(\mathbb{R}^d)$ and $\varphi^{i}(\cdot) := \varphi_0^{i}(\cdot)-2^{-1}\varphi_0^{i}(2^{-1}\cdot)$ such that $\int_{\mathbb{R}} \varphi_0^{i} * \varphi_0^{i} (x) dx = 1$ and $(D^s \widehat{\varphi})(0)=0$ for $s=0,1,...,L$ where $L$ is a natural number such that $\delta(\bar{j}) < L|\bar{j}|_1$ for $\bar{j} \in \mathbb{N}_0^d$. Define
$\boldsymbol{\varphi}_{\bar{j}}$ by the following:
\begin{equation}
    \boldsymbol{\varphi}_{\bar{j}}(x)
    := \prod_{i=1}^{d} \varphi_{j_i}^{i}(x_i)
    \quad
    \text{and}
    \quad \varphi_{j_i}^{i}(x_i) := 2^{j_i} \varphi^{i}(2^{j_i} x_i) 
\end{equation}

\begin{equation}
\label{supoperator}
    \boldsymbol{\varphi}_{\bar{j},b',L+1}^{*}f(x)
    :=
    \sup_{y \in \mathbb{R}} 
    \frac{|\boldsymbol{\varphi}_{\bar{j}}*f(x-y)|}
    {\prod_{i=1}^d e^{2^{j_i} b' |y_i|_1}(1+|2^{j_i}y_i|)^{L+1}}
\end{equation}
Then,
\begin{equation}
\label{equi2}
\left(\sum_{\bar{j}\in \mathbb{N}^d}2^{q \delta( \bar{j})}
||\boldsymbol{\varphi}_{\bar{j},b',L+1}^{*}f||_{L_p^{w}}^q\right)^{1/q}
\approx
\left(\sum_{\bar{j}\in \mathbb{N}^d}2^{q \delta( \bar{j})}
||\boldsymbol{\varphi}_{\bar{j}}*f||_{L_p^{w}}^q\right)^{1/q}
\end{equation}
\end{theorem}
\begin{proof}
We only prove
\begin{equation}
    \left(\sum_{\bar{j}\in \mathbb{N}^d}2^{q \delta( \bar{j})}
    ||\boldsymbol{\varphi}_{\bar{j},b',L+1}^{*}f||_{L_p^{w}}^q\right)^{1/q}
    \lesssim
    \left(\sum_{\bar{j}\in \mathbb{N}^d}2^{q \delta( \bar{j})}
    ||\boldsymbol{\varphi}_{\bar{j}}*f||_{L_p^{w}}^q\right)^{1/q}
\end{equation}
because inverse inequality can be easily derived from $|\boldsymbol{\varphi}_{\bar{j}}*f(x)| \leq \boldsymbol{\varphi}_{\bar{j},b',L+1}^{*}f(x)$.
Because of Lemma.\ref{kernel}, $\boldsymbol{\zeta}_{\bar{j}} \in X_{b}^{\bar{L}+\bar{1}}(\mathbb{R}^d)$ exists such that
\begin{equation}
    f = \sum_{\bar{j} \in \mathbb{N}_0^d} \boldsymbol{\zeta}_{\bar{j}}*\boldsymbol{\varphi}_{\bar{j}}*f \qquad \text{in} \quad X_{b}^{\bar{L}}(\mathbb{R}^d)'
\end{equation}
Then,

\begin{equation}
    \frac{\boldsymbol{\varphi}_{\bar{k}}*f(x-y)}
    {\prod_{i=1}^d e^{2^{k_i} b' |y_i|_1}(1+|2^{k_i}y_i|)^{L+1}}
    = 
    \sum_{\bar{j} \in \mathbb{N}_0^d}
    \int_{\mathbb{R}^d} dz
    \boldsymbol{\varphi}_{\bar{k}}*\boldsymbol{\zeta}_{\bar{j}}(z)
    \frac{\boldsymbol{\varphi}_{\bar{j}}*f(x-y-z)}
    {
    \prod_{i=1}^d e^{2^{k_i} b' |y_i|}(1+|2^{k_i}y_i|)^{L+1}
    }
\end{equation}
Let $I(\bar{j},\bar{k})$ be an index set defined by $I(\bar{j},\bar{k}):=\{ i :\: k_i < j_i \}$. Let $m_{\bar{j},\bar{k},L}(y)$ by
\begin{equation}
    m_{\bar{j},\bar{k},L}(y) 
    = 
    \prod_{i \in I(\bar{j},\bar{k})} 2^{-k_i}
    (1+|2^{k_i} y_i|)^{L} e^{2^{k_i}b'|y_i|}
    \prod_{i \in I(\bar{j},\bar{k})^c} 2^{-j_i}
    (1+|2^{j_i} y_i|)^{L} e^{2^{j_i}b'|y_i|}
\end{equation}
Here,
\begin{eqnarray}
\label{hoge}
    && m_{\bar{j},\bar{k},L+1}(x-y-z) \prod_{i=1}^d e^{2^{k_i} b' |y_i|}(1+|2^{k_i}y_i|)^{L+1} \nonumber  \\
    &=&
    \prod_{i \in I(\bar{j},\bar{k})} 2^{-k_i}
    (1+|2^{k_i} (x_i-y_i-z_i)|)^{L+1} e^{2^{k_i}b'|x_i-y_i-z_i|} \nonumber \\
    && \qquad \qquad \times
    \prod_{i \in I(\bar{j},\bar{k})^c} 2^{-j_i}
    (1+|2^{j_i} (x_i-y_i-z_i)|)^{L+1} e^{2^{j_i}b'|x_i-y_i-z_i|}
    \prod_{i=1}^d e^{2^{k_i} b' |y_i|}(1+|2^{k_i}y_i|)^{L+1}
    \nonumber \\
    &=&
    \prod_{i \in I(\bar{j},\bar{k})} 2^{-k_i}
    (1+|2^{k_i} (x_i-y_i-z_i)|)^{L+1}  e^{2^{k_i}b'|x_i-y_i-z_i|}
    e^{2^{k_i} b' |y_i|}(1+|2^{k_i}y_i|)^{L+1}
    \nonumber \\
    && \qquad \qquad \times
    \prod_{i \in I(\bar{j},\bar{k})^c} 2^{-j_i}
    (1+|2^{j_i} (x_i-y_i-z_i)|)^{L+1} e^{2^{j_i}b'|x_i-y_i-z_i|}
    e^{2^{j_i} b' |y_i|}(1+|2^{j_i}y_i|)^{L+1}
    \nonumber \\
    &\geq&
    \prod_{i \in I(\bar{j},\bar{k})} 2^{-k_i}
    (1+|2^{k_i} (x_i-z_i)|)^{L+1}  e^{2^{k_i}b'|x_i-z_i|}
    \nonumber \\
    && \qquad \qquad \times
    \prod_{i \in I(\bar{j},\bar{k})^c} 2^{-j_i}
    (1+|2^{j_i} (x_i-z_i)|)^{L+1} e^{2^{j_i}b'|x_i-z_i|}
    \nonumber \\
    &=&
    m_{\bar{j},\bar{k},L+1}(x-z)
\end{eqnarray}
where we use the following relation in the third inequality.
\begin{eqnarray}
    (1+2^{l}|x-y-z|)(1+2^{l}|y|)
    &=&
    1+2^{l}|x-y-z|+2^{l}|y|+2^{2l}|x-y-z||y| \nonumber \\
    &\geq&
    1+2^{l}|x-z|
\end{eqnarray}
By Lemma \ref{tech} in Appendix \ref{appendix_3} and (\ref{hoge}),
\begin{eqnarray}
    \frac{|\boldsymbol{\varphi}_{\bar{k}}*f(x-y)|}
    {\prod_{i=1}^d e^{2^{k_i} b' |y_i|}(1+|2^{k_i}y_i|)^{L+1}}
    &\leq&
    \sum_{\bar{j} \in \mathbb{N}_0^d}
    \int_{\mathbb{R}^d} dz
    |\boldsymbol{\varphi}_{\bar{k}}*\boldsymbol{\zeta}_{\bar{j}}(x-y-z)|
    \frac{|\boldsymbol{\varphi}_{\bar{j}}*f(z)|}
    {
    \prod_{i=1}^d e^{2^{k_i} b' |y_i|}(1+|2^{k_i}y_i|)^{L+1}
    } \nonumber \\
    &\lesssim&
    \sum_{\bar{j} \in \mathbb{N}_0^d}
    \int_{\mathbb{R}^d} dz
    \frac{2^{-L|\bar{j}-\bar{k}|_1}}{m_{\bar{j},\bar{k},L}(x-y-z)}
    \frac{|\boldsymbol{\varphi}_{\bar{j}}*f(z)|}
    {
    \prod_{i=1}^d e^{2^{k_i} b' |y_i|}(1+|2^{k_i}y_i|)^{L+1}
    } \nonumber \\
    &\leq&
    \sum_{\bar{j} \in \mathbb{N}_0^d}
    2^{-L|\bar{j}-\bar{k}|_1}
    \int_{\mathbb{R}^d} dz
    \frac{|\boldsymbol{\varphi}_{\bar{j}}*f(z)|}
    {
    m_{\bar{j},\bar{k},L+1}(x-z)
    } \nonumber \\
    &=&
    \sum_{\bar{j} \in \mathbb{N}_0^d} 2^{-L|\bar{j}-\bar{k}|_1}
    m_{\bar{j},\bar{k},L+1}^{-1} * (|\boldsymbol{\varphi}_{\bar{j}}*f|)(x)
\end{eqnarray}
Thus, we have the following equation:
\begin{equation}
    \boldsymbol{\varphi}_{\bar{j},b',L+1}^{*}f(x)
    \lesssim
    \sum_{\bar{j} \in \mathbb{N}_0^d} 2^{-L|\bar{j}-\bar{k}|_1}
    m_{\bar{j},\bar{k},L+1}^{-1} * (|\boldsymbol{\varphi}_{\bar{j}}*f|)(x)
\end{equation}
The remainder of this proof is similar to that of Theorem.\ref{equivalence} from (\ref{norm1}).
\end{proof}
\begin{remark}
\label{rm}
Polynomial terms appeared in the definitions of $m_{\bar{j} ,\bar{k},L+1}$ in (\ref{m}) and $\boldsymbol{\varphi}_{\bar{j},b',L+1}^{*}$ in (\ref{supoperator}) are not necessary for the proof of theorem.\ref{equivalence} and theorem.\ref{equivalence}. Condition $c_w/p < b$ is essential here. This fact is used in the proof of corollary.\ref{nece}.
\end{remark}

Next corollary is useful to characterize the Besov space by wavelet in Theorem \ref{wav} (ii-b).
\begin{corollary}
\label{nece}
Let $f \in VB_{p,q}^{\delta,w}$, $b,b' > 0$ such that $c_w/p < b' < b$, $\psi_0^{i},\psi^{i}, \varphi_0^{i} \in X_b^{L+1}(\mathbb{R}^d)$ and $\varphi^{i}(\cdot) := \varphi_0^{i}(\cdot)-2^{-1}\varphi_0^{i}(2^{-1}\cdot)$ such that $\int_{\mathbb{R}} \varphi_0^{i} * \varphi_0^{i} (x) dx = 1$, $(D^s \widehat{\psi^i})(0)=0$ and $(D^s \widehat{\varphi^i})(0)=0$ for $s=0,1,...,L$ where $L$ is a natural number such that $\delta(\bar{j}) <  L|\bar{j}|_1$ for all $\bar{j} \in \mathbb{N}_0^d$. Define $\boldsymbol{\psi}_{\bar{j}}$, $\boldsymbol{\varphi}_{\bar{j}}$ by
\begin{equation}
    \boldsymbol{\psi}_{\bar{j}}(x)
    := \prod_{i=1}^{d} \psi_{j_i}^{i}(x_i) 
    \qquad
    \boldsymbol{\varphi}_{\bar{j}}(x)
    := \prod_{i=1}^{d} \varphi_{j_i}^{i}(x_i) 
\end{equation}
\begin{equation}
\label{supoperator}
    \boldsymbol{\psi}_{\bar{j},b'}^{*}f(x)
    :=
    \sup_{y \in \mathbb{R}} 
    \frac{|\boldsymbol{\psi}_{\bar{j}}*f(x-y)|}
    {\prod_{i=1}^d e^{2^{j_i} b' |y_i|}}
\end{equation}

Then, we have the following relation:
\begin{equation}
\label{large2}
\left(\sum_{\bar{j}\in \mathbb{N}^d}2^{q \delta( \bar{j})}
||\boldsymbol{\psi}_{\bar{j},b',L+1}^{*}f||_{L_p^{w}}^q\right)^{1/q}
\lesssim
    \left(\sum_{\bar{j}\in \mathbb{N}^d}2^{q \delta( \bar{j})}
||\boldsymbol{\varphi}_{\bar{j}}*f||_{L_p^{w}}^q\right)^{1/q}
\end{equation}
\end{corollary}
\begin{proof}
Because of Lemma \ref{kernel}, $\boldsymbol{\zeta}_{\bar{j}} \in X_{b}^{\bar{L}+\bar{1}}(\mathbb{R}^d)$ exists such that we have the following equation:
\begin{equation}
    f = \sum_{\bar{j} \in \mathbb{N}_0^d} \boldsymbol{\zeta}_{\bar{j}}*\boldsymbol{\varphi}_{\bar{j}}*f \qquad \text{in} \quad X_{b}^{\bar{L}}(\mathbb{R})'
\end{equation}

By repeating a similar procedure in the proof of Theorem \ref{supequi}, we obtain the following equation:
\begin{equation}
    \frac{|\boldsymbol{\psi}_{\bar{k}}*f(x-y)|}
    {\prod_{i=1}^d e^{2^{k_i} b' |y_i|}}
    \lesssim
    \sum_{\bar{j} \in \mathbb{N}_0^d} 2^{-L|\bar{j}-\bar{k}|_1}
    m_{\bar{j},\bar{k}}^{-1} * (|\boldsymbol{\varphi}_{\bar{j}}*f|)(x)
\end{equation}
where $m_{\bar{j},\bar{k}}(y)$ is defined by
\begin{equation}
    m_{\bar{j},\bar{k}}(z) 
    = 
    \prod_{i \in I(\bar{j},\bar{k})} 2^{-k_i} e^{2^{k_i}b'|z_i|}
    \prod_{i \in I(\bar{j},\bar{k})^c} 2^{-j_i} e^{2^{j_i}b'|z_i|}
\end{equation}
with $I(\bar{j},\bar{k})$.
Thus, we have the following expression:
\begin{equation}
    \boldsymbol{\psi}_{\bar{j},b'}^{*}f(x)
    \lesssim
    \sum_{\bar{j} \in \mathbb{N}_0^d} 2^{-L|\bar{j}-\bar{k}|_1}
    m_{\bar{j},\bar{k}}^{-1} * (|\boldsymbol{\varphi}_{\bar{j}}*f|)(x)
\end{equation}
Finally, considering Remark \ref{rm} and repeating a similar procedure in the proof of Theorem \ref{equivalence}, we obtain the desired inequality.
\end{proof}

\section{Wavelet characterization}
\label{characwave}
$\quad \; \; \;$ The purpose of this section is to characterize the exponentially weighted Besov space with dominating mixed smoothness. In Theorem \ref{wav}, we obtain the characterizations of $\| f | VB_{p,q}^{\delta,w}(\mathbb{R}^d)\|$ using the coefficients of wavelet expansion $f$. We assume wavelet $\{ \boldsymbol{\psi}_{\bar{j},\bar{m}}\}_{\bar{j} \in \mathbb{N}_0^d, \bar{m} \in \mathbb{Z}^d}$ used here satisfies the following properties.
\begin{assumption}
\label{wav_assumption}
$\boldsymbol{\psi}_{\bar{k},\bar{m}}(x)$ is defined by
    \begin{equation}
        \boldsymbol{\psi}_{\bar{k},\bar{m}}(x)
        := 
        \prod_{i=1}^d
        \psi_{k_i,m_i}^{i}(x_i)
    \end{equation}
    where each $\psi_{k_i,m_i}^{i}(x_i)$ is  $C^{L_i}$ wavelet, that is,
\begin{eqnarray}
    \psi_{0,m_i}^{i}(x_i) &=& \phi(x-m_i) \\
    \psi_{k_i,m_i}^{i}(x_i) &=& 2^{k_i/2}\psi(2^{k_i}x-m_i)
\end{eqnarray}
where $\phi$ is a scaling function and $\psi$ is a wavelet function. Furthermore, $\boldsymbol{\psi}_{\bar{k},\bar{m}}(x)$ satisfy the following three properties:
\begin{enumerate}
    \item $|\partial^{s_i}
    \psi_{k_i,m_i}^{i}(x_i)|
    \leq
    c_{\psi^{i}}
    2^{s_i k_i} 2^{k_i/2}
    \qquad
    0
    \leq
    s_i
    \leq
    L_i$
    \item $\text{supp} \: \psi_{k_i,m_i} \subset \gamma Q_{k_i,m_i}$ where $Q_{k_i,m_i} = [2^{-k_i} m_i,2^{-k_i} (m_i+1))$ and $\gamma$ is some positive constant. 
    \item $\{\boldsymbol{\psi}_{\bar{k},\bar{m}}\}_{\bar{k} \in \mathbb{N}_0^d, \bar{m} \in \mathbb{Z}^d}$ is orthogonal basis in $L^2(\mathbb{R}^d)$
\end{enumerate}
\end{assumption}
Here, the d-times product of Daubechies wavelet satisfies the aforementioned assumption. The Daubechies wavelet is compactly supported and the desired smoothness can be obtained\cite{Daubechies92}. Such d-times product of the Daubechies wavelet is included in $X_{b}^{\bar{L}}(\mathbb{R}^d)$ with an arbitrary finite $\bar{L}$ and $b$.

\begin{theorem}
\label{wav}
Let $\bar{L} \in \mathbb{N}_0^d$, $\bar{s} \in \mathbb{R}^d_{\geq 0}$ such that $\delta(\bar{j}) <  L |\bar{j}|_1$ for all $\bar{j} \in \mathbb{N}_0^d$ and $b \in \mathbb{R}_{>0}$ such that $\max(c_w/(p-1),c_w/p)<b$. $\boldsymbol{\psi}_{\bar{k},\bar{m}} \in X_{b}^{(L,\cdots,L)+\bar{1}}(\mathbb{R}^{d})$ satisfies assumption.\ref{wav_assumption}.
\begin{itemize}
    \item [(i)] If a sequence $\{\lambda \}_{\bar{k},\bar{m}}$ satisfies
    \begin{equation}
    \label{sequence}
        \left(\sum_{\bar{k} \in  \mathbb{N}_0^d} 2^{ q \delta(\bar{k})} \left\| \sum_{\bar{m}\in \mathbb{Z}^d} |\lambda_{\bar{k},\bar{m}}|
            2^{|\bar{k}|_1/2}\boldsymbol{\chi}_{\bar{k},\bar{m}}
            \right \|_{L^p_w}^{q} \right)^{1/q} < \infty
    \end{equation}
    then,
        \begin{itemize}
            \item [(i-a)] The series
            \begin{equation}
            \label{series}
                \sum_{\bar{k}\in \mathbb{N}^d_0} \sum_{\bar{m}\in \mathbb{Z}^d} \lambda_{\bar{k},\bar{m}} \boldsymbol{\psi}_{\bar{k},\bar{m}}
            \end{equation}
            converges in $X_{b}^{\bar{L}+\bar{2}}(\mathbb{R}^d)'$ to some distribution $f$.
            \item [(i-b)]
            $f$ belongs to $VB^{\delta,w}_{p,q}(\mathbb{R}^d)$. Furthermore, we have the following expression:
            \begin{equation}
            \label{large}
            \|f|VB_{p,q}^{\delta,w}(\mathbb{R}^d)\|
            \lesssim 
            \left(\sum_{\bar{k} \in \mathbb{N}_0^d} 2^{q\delta(\bar{k})} \left\| \sum_{\bar{m}\in \mathbb{Z}^d} |\lambda_{\bar{k},\bar{m}}|
            2^{|\bar{k}|_1/2} \boldsymbol{\chi}_{\bar{k},\bar{m}}
            \right \|_{L^p_w}^{q} \right)^{1/q}
            \end{equation}
        \end{itemize}
    \item [(ii)] Let $f \in VB_{p,q}^{\delta,w}(\mathbb{R}^d)$ and $\{ \lambda_{\bar{k},\bar{m}} \}_{\bar{k} \in \mathbb{N}_0^d,\bar{m} \in \mathbb{Z}^d}$ be a sequence defined by the following expression:
\begin{equation}
    \lambda_{\bar{k},\bar{m}} = 
    \langle f,\boldsymbol{\psi}_{\bar{k},\bar{m}} \rangle
\end{equation}
    \begin{itemize}
    \item [(ii-a)] Then, we have the following expression:
            \begin{equation}
                \left(\sum_{\bar{k}\in \mathbb{N}_0^d} 2^{q \delta(\bar{k})} \left\| \sum_{\bar{m}\in \mathbb{Z}^d} |\lambda_{\bar{k},\bar{m}}|
            2^{|\bar{k}|_1/2} \boldsymbol{\chi}_{\bar{k},\bar{m}}
            \right \|_{L^p_w}^{q} \right)^{1/q}
            \lesssim
            ||f|VB_{p,q}^{\delta,w}(\mathbb{R}^d)||
            \end{equation}
    \item [(ii-b)]
            The series (\ref{series}) converges to $f$ in $X_b^{\bar{L}+\bar{2}}(\mathbb{R}^d)'$.
        \end{itemize}
\end{itemize}
\end{theorem}
\begin{proof}
\begin{itemize}
    \item [(i-a)]
    Select $\phi \in X^{\bar{L}+\bar{2}}_{b} (\mathbb{R}^d)$.
    Let $K,K' \subset \mathbb{N}_0^d$ be finite s.t. $K \subset K'$
     and $M,M' \subset \mathbb{Z}_0^d$ be finite s.t. $M \subset M'$. Let $\epsilon$ be an arbitrary positive real number. If $\min_{\bar{m} \in M} |\bar{m}|_1$, $\min_{\bar{k} \in K} |\bar{k}|_1$ are sufficiently large
    \begin{eqnarray}
    &&
    \left|
    |\langle
    \sum_{ \bar{k} \in K}
    \sum_{ \bar{m} \in M}
    \lambda_{\bar{k},\bar{m}}
    \boldsymbol{\psi}_{\bar{k},\bar{m}},\phi
    \rangle|
    -
    |\langle
    \sum_{ \bar{k} \in K'}
    \sum_{ \bar{m} \in M'}
    \lambda_{\bar{k},\bar{m}}
    \boldsymbol{\psi}_{\bar{k},\bar{m}},\phi
    \rangle|
    \right| \nonumber \\
    &\leq&
    \sum_{\bar{k} \in K \backslash K'}
    \left\| \sum_{ \bar{m} \in M \backslash M'}
    \lambda_{\bar{k},\bar{m}}
    \boldsymbol{\psi}_{\bar{k},\bar{m}} \right\|_{L^p_w}
    \| \phi \|_{L^{p'}_{w^{-p'/p}}} \nonumber \\
    &\leq&
    \| \phi \|_{L^{p'}_{w^{-{p'}/p}}}
    \left(
    \sum_{\bar{k} \in K \backslash K'}
    2^{q \delta(\bar{k})}
    \left\| \sum_{ \bar{m} \in M \backslash M'}
    \lambda_{\bar{k},\bar{m}}
    \boldsymbol{\psi}_{\bar{k},\bar{m}} \right\|_{L^p_w}^q \right)^{1/q}
    \left(
    \sum_{\bar{k} \in K \backslash K'}
    2^{-q' \delta(\bar{k})}
    \right)^{1/q'} \nonumber \\
    &\lesssim&
    \| \phi \|_{L^{p'}_{w^{-p'/p}}}
    \left(
    \sum_{\bar{k} \in K \backslash K'}
    2^{-q' \delta(\bar{k})}
    \right)^{1/q'}
    \left(
    \sum_{\bar{k} \in K \backslash K'}
    2^{q \delta(\bar{k})}
    \left\| \sum_{ \bar{m} \in M \backslash M'}
    |\lambda_{\bar{k},\bar{m}}|
    2^{|\bar{k}|_1/2}
    \boldsymbol{\chi}_{\bar{k},\bar{m}}(\cdot/\gamma) \right\|_{L^p_w}^q \right)^{1/q} \nonumber \\
    &\lesssim&
    \| \phi \|_{L^{p'}_{w^{-p'/p}}}
    \left(
    \sum_{\bar{k} \in K \backslash K'}
    2^{-q' \delta(\bar{k})}
    \right)^{1/q'}
    \left(
    \sum_{\bar{k} \in K \backslash K'}
    2^{q \delta(\bar{k})}
    \left\| \sum_{ \bar{m} \in M \backslash M'}
    |\lambda_{\bar{k},\bar{m}}|
    2^{|\bar{k}|_1/2}
    \boldsymbol{\chi}_{\bar{k},\bar{m}} \right\|_{L^p_w}^q \right)^{1/q} \nonumber \\
    &\leq&
    \epsilon
    \end{eqnarray}
    where $1/{q'}=1-1/q$. The term $\| \phi \|_{L^{p'}_{w^{-p'/p}}}$ is bounded because $|\phi(x)| \lesssim e^{-b|x|_1}$ by Lemma.\ref{bounded} and $w^{-p/{p'}}(x) \lesssim e^{c_w p/p' |x|}=e^{c_w (p-1) |x|_1}$.
    In the second inequality, the Hölder inequality for the discrete version is used. In the fourth inequality, Lemma 1.4 in \cite{Rychkov01} is used. Thus, because of the completeness of the real number, the convergence of (\ref{series}) in $X^{L+2}_{b}(\mathbb{R}^d)'$ is proved.
    
    \item [(i-b)] Let $\kappa^{L+1,i}_0$ and $\kappa^{L+1,i}$ be functions given in Lemma \ref{kappabound} in Appendix \ref{appendix_4} for each $i=1,\cdots,d$. Define $\boldsymbol{\kappa}^{L+1}_0$ and $\boldsymbol{\kappa}^{L+1}$ by
    \begin{eqnarray}
    \boldsymbol{\kappa}^{L+1}_{\bar{j}}(x) := \prod_{i=1}^d \kappa^{L+1,i}_{j_i}(x_i)
    \end{eqnarray}
    By theorem.\ref{equivalence},
    \begin{equation}
    \label{equi}
        \|f|MB_{p,q}^{\bar{s},w}(\mathbb{R}^d)\|
         \approx
         \left(\sum_{\bar{j}\in \mathbb{N}^d}2^{q \delta(\bar{j})} ||\boldsymbol{\kappa}_{\bar{j}}*f||_{L_p^{w}}^q\right)^{1/q}
    \end{equation}
    The estimate of $||\boldsymbol{\kappa}_{\bar{j}}*f||_{L_p^{w}}$ is given by the following expression:
    \begin{eqnarray}
    ||\boldsymbol{\kappa}_{\bar{j}}*f||_{L_p^{w}}
    &\leq&
    \sum_{\bar{k} \in \mathbb{N}_0^d}
    \left\|\boldsymbol{\kappa}_{\bar{j}}*\left(
    \sum_{\bar{m}\in \mathbb{Z}^d} \lambda_{\bar{k},\bar{m}} \boldsymbol{\psi}_{\bar{k},\bar{m}}
    \right)\right\|_{L_p^{w}}
    \nonumber \\
    &\lesssim&
    \sum_{\bar{k} \in \mathbb{N}_0^d}
    2^{-(L+1)|\bar{j}-\bar{k}|_1}
    \left\| 2^{|\bar{j}|_1} 2^{|\bar{k}|_1/2}
    \int_{c_{L+1}Q_{\bar{j}}} \sum_{\bar{m}\in \mathbb{Z}^d} |\lambda_{\bar{k},\bar{m}}| \chi_{\bar{k},\bar{m}}((\cdot-y)/\gamma)dy
    \right\|_{L_p^{w}} \nonumber \\
    &\lesssim&
    \sum_{\bar{k} \in \mathbb{N}_0^d}
    2^{-(L+1)|\bar{j}-\bar{k}|_1}
    \left\|
    \sum_{\bar{m}\in \mathbb{Z}^d} |\lambda_{\bar{k},\bar{m}}| 
    2^{|\bar{k}|_1/2}
    \boldsymbol{\chi}_{\bar{k},\bar{m}} \right\|_{L_p^{w}}
    \end{eqnarray}
where $\gamma$ is a constant appeared in assumption.\ref{wav_assumption}
In the second inequality Lemma \ref{kappabound} in Appendix \ref{appendix_4} is used. In the third inequality, the boundedness of the local maximal operator in the norm of $L^p_w$ (Lemma 2.11 in \cite{Rychkov01}) is used.
Thus, we have the following equation:
\begin{eqnarray}
\label{larger}
\sum_{\bar{j}\in \mathbb{N}^d}2^{q \delta(\bar{j})} ||\boldsymbol{\kappa}_{\bar{j}}*f||_{L_p^{w}}^q
&\lesssim&
\sum_{\bar{j}\in \mathbb{N}^d}2^{q \delta(\bar{j})} 
\left(
\sum_{\bar{k} \in \mathbb{N}_0^d}
2^{-(L+1)|\bar{j}-\bar{k}|_1}
    \left\|
    \sum_{\bar{m}\in \mathbb{Z}^d} |\lambda_{\bar{k},\bar{m}}| 
    2^{|\bar{k}|_1/2}
    \boldsymbol{\chi}_{\bar{k},\bar{m}} \right\|_{L_p^{w}}  \right)^q \nonumber  \\
&\leq&
\sum_{\bar{j}\in \mathbb{N}^d}2^{q \delta(\bar{j})} 
\sum_{\bar{k} \in \mathbb{N}_0^d}
2^{-q(L+1-\epsilon)|\bar{j}-\bar{k}|_1}
    \left\|
    \sum_{\bar{m}\in \mathbb{Z}^d} |\lambda_{\bar{k},\bar{m}}| 
    2^{|\bar{k}|_1/2}
    \boldsymbol{\chi}_{\bar{k},\bar{m}}\right\|_{L_p^{w}}^q 
    \left(
    \sum_{\bar{k} \in \mathbb{N}_0^d}
    2^{-q' \epsilon |\bar{j}-\bar{k}|_1}
    \right)^{q/q'} \nonumber \\
    &\lesssim&
    \sum_{\bar{k} \in \mathbb{N}_0^d}
    2^{q \delta(\bar{k})} 
    \left\|
    \sum_{\bar{m}\in \mathbb{Z}^d} |\lambda_{\bar{k},\bar{m}}| 
    2^{|\bar{k}|_1/2}
    \boldsymbol{\chi}_{\bar{k},\bar{m}}\right\|_{L_p^{w}}^q
    \sum_{\bar{j} \in \mathbb{N}_0^d}
    2^{-q(L+1-\epsilon)|\bar{j}-\bar{k}|_1 + \delta(\bar{j}-\bar{k})} \nonumber \\
    &\lesssim&
    \sum_{\bar{k} \in \mathbb{N}_0^d}
    2^{q \delta(\bar{k})} 
    \left\|
    \sum_{\bar{m}\in \mathbb{Z}^d} |\lambda_{\bar{k},\bar{m}}| 2^{|\bar{k}|_1/2}\boldsymbol{\chi}_{\bar{k},\bar{m}}\right\|_{L_p^{w}}^q < \infty
\end{eqnarray}
From (\ref{equi}) and (\ref{larger}), $f \in VB^{\delta,w}_{p,q}(\mathbb{R}^d)$ and desired inequality (\ref{large}) are proved.
\item[(ii-a)]
Let $\widetilde{\boldsymbol{\psi}_{\bar{k}}}(x) := \boldsymbol{\psi_{\bar{k}, \bar{0}}}(-x)$. Note that
    \begin{eqnarray}
        \lambda_{\bar{k},\bar{m}} &=&  \int_{\mathbb{R}^d} f(x) 2^{|\bar{k}|_1/2} \psi (2^{\bar{k}}x- \bar{m})dx \nonumber \\
        &=&
         \int_{\mathbb{R}^d} f(x)  \widetilde{\psi}_{\bar{k}} (2^{-\bar{k}} \bar{m}-x)dx \nonumber \\
        &=& f*\widetilde{\psi}_{\bar{k}} (2^{-\bar{k}} \bar{m})
    \end{eqnarray}
    Then,
    \begin{eqnarray}
    &&\sum_{\bar{k}} 2^{q\delta(\bar{k})} \left\| \sum_{\bar{m}\in \mathbb{Z}^d} |\lambda_{\bar{k},\bar{m}}|
            2^{|\bar{k}|_1/2}
            \chi_{\bar{k},\bar{m}}
            \right \|_{L^p_w}^{q} \nonumber \\
    &=&
    \sum_{\bar{k} \in \mathbb{N}_0^d}
    2^{q \delta(\bar{k})}
    \left\|
    \sum_{\bar{m} \in \mathbb{Z}^d}
     |\widetilde{\boldsymbol{\psi}}_{\bar{k}}*f(2^{-\bar{k}}\bar{m})|
    2^{|\bar{k}|_1/2}
    \chi_{\bar{k},\bar{0}}(\cdot - 2^{-\bar{k}}\bar{m}) 
    \right\|_{L^p_w}^q \nonumber \\
    &\lesssim&
    \sum_{\bar{k} \in \mathbb{N}_0^d}
    2^{q \delta(\bar{k})}
    \left\|
    \sum_{\bar{m} \in \mathbb{Z}^d} |\widetilde{\boldsymbol{\psi}}_{\bar{k}}*f(2^{-\bar{k}}\bar{m})|
    2^{ |\bar{k}|_1/2} \chi_{\bar{k},\bar{0}}(\cdot - 2^{-\bar{k}}\bar{m}) e^{-b|2^{\bar{k}} \cdot - \bar{m}|_1}
    \right\|_{L^p_w}^q \nonumber \\
    &\lesssim&
    \sum_{\bar{k} \in \mathbb{N}_0^d}
    2^{q \delta(\bar{k})}
    \left\|
    \sup_{y \in \mathbb{R}^d}
    \left(
    \frac{2^{|\bar{k}|_1/2}|\widetilde{\boldsymbol{\psi}}_{\bar{k}}*f(y)|}{e^{b|2^{\bar{k}}(x-y)|_1}}
    \right)
    \right\|_{L^p_w}^q
    \end{eqnarray}
    
    By Theorem \ref{equivalence} and Collorary \ref{nece}
    \begin{eqnarray}
        \left(\sum_{\bar{k}} 2^{q \delta( \bar{k}) -q|\bar{k}|_1/2} \left\| \sum_{\bar{m}\in \mathbb{Z}^d} |\lambda_{\bar{k},\bar{m}}|
            2^{|\bar{k}|_1/2} \chi_{\bar{k},\bar{m}}
            \right \|_{L^p_w}^{q} \right)^{1/q}
        &\lesssim&
        \left(\sum_{\bar{k} \in \mathbb{N}_0^d}
    2^{q \delta(\bar{k})}
    \left\|
    \boldsymbol{\varphi}_{\bar{k}}*f
    \right\|_{L^p_w}^q\right)^{1/q} \nonumber \\
    &\approx&
    ||f|VB_{p,q}^{\delta,w}(\mathbb{R}^d)||
    \end{eqnarray}
    with some $\varphi_0^{i} \in X_b^{L+1}(\mathbb{R}^d)$ and $\varphi^{i}(\cdot) := \varphi_0^{i}(\cdot)-2^{-1}\varphi_0^{i}(2^{-1}\cdot)$ such that $(D^s \widehat{\varphi})(0)=0$ for $s=0,1,...,L$.
    \item[(ii-b)]
    Let $h$ be a distribution defined by
    \begin{equation}
        h 
        := 
        \sum_{\bar{m} \in \mathbb{N}_0^d}
        \sum_{\bar{k} \in \mathbb{Z}^d}
        \lambda_{\bar{k},\bar{m}}
        \boldsymbol{\psi}_{\bar{k},\bar{m}}
    \end{equation}
    The convergence of $h$ in $X_{b}^{\bar{L}+\bar{2}}(\mathbb{R}^d)$is assured from (i-a) and (ii-a).
    
    Select $\phi \in X_{b}^{\bar{L}+\bar{2}}(\mathbb{R}^d)$ arbitrary, and $\phi$ has the following expansion:
    \begin{equation}
        \phi 
        := 
        \sum_{\bar{k} \in \mathbb{N}_0^d}
        \sum_{\bar{m} \in \mathbb{Z}^d}
        2^{|\bar{k}|_1/2} 
        \langle \phi, \boldsymbol{\psi}_{\bar{k},\bar{m}} \rangle
        \boldsymbol{\psi}_{\bar{k},\bar{m}}
    \end{equation}
    By orthonormality of $\boldsymbol{\psi}_{\bar{k},\bar{m}}$
    \begin{eqnarray}
        \langle
        h,
        \phi
        \rangle 
        &=&
        \sum_{\bar{k} \in \mathbb{N}_0^d}
        \sum_{\bar{m} \in \mathbb{Z}^d}
        \langle \phi, \boldsymbol{\psi}_{\bar{k},\bar{m}} \rangle
        \lambda_{\bar{k},\bar{m}}
        \nonumber \\
        &=&
        \sum_{\bar{k} \in \mathbb{N}_0^d}
        \sum_{\bar{m} \in \mathbb{Z}^d}
        2^{|\bar{k}|_1/2} 
        \langle \phi, \boldsymbol{\psi}_{\bar{k},\bar{m}} \rangle
        \langle f, \boldsymbol{\psi}_{\bar{k},\bar{m}} \rangle
        \nonumber \\
        &=&
        \langle f,
        \sum_{\bar{k} \in \mathbb{N}_0^d}
        \sum_{\bar{m} \in \mathbb{Z}^d}
        2^{|\bar{k}|_1/2} 
        \langle \phi, \boldsymbol{\psi}_{\bar{k},\bar{m}} \rangle
        \boldsymbol{\psi}_{\bar{k},\bar{m}} \rangle \nonumber \\
        &=&
        \langle f, \phi \rangle
    \end{eqnarray}
    Then,
    \begin{equation}
        f = \sum_{\bar{m} \in \mathbb{N}_0^d}
        \sum_{\bar{k} \in \mathbb{Z}^d}
        \lambda_{\bar{m},\bar{k}}
        \boldsymbol{\psi}_{\bar{k},\bar{m}}
        \qquad
        \text{in}
        \quad
        X_{b}^{\bar{L}+\bar{2}}(\mathbb{R}^d)'
    \end{equation}
\end{itemize}
\end{proof}

\section{Interpolation on \texorpdfstring{$VB^{\delta,w}_{p,q}(\mathbb{R}^d)$}{LG}}
\label{interpo}
$\quad \; \; \;$ Here, the generalization of the Besov space with respect to smoothness as $VB^{\delta,w}_{p,q}(\mathbb{R}^d)$ is detailed in this section. First, we evaluate $VB^{\delta,w}_{p,q}(\mathbb{R}^d)$ is equal to $B_{p,q}^{\bar{s},w}(\mathbb{R}^d)$ or $MB_{p,q}^{s,w}(\mathbb{R}^d)$ with some $\delta$. When $\delta(\bar{j}) = s_1 j_1 + \cdots + s_d j_d$, $VB^{\delta,w}_{p,q}(\mathbb{R}^d) \approx MB_{p,q}^{s,w}(\mathbb{R}^d)$ is obvious from definition. However, $\delta$ that leads to $VB^{\delta,w}_{p,q}(\mathbb{R}^d) \approx B_{p,q}^{s,w}(\mathbb{R}^d)$ is not obvious. In the first subsection, we provide a proof of $VB^{sl_{\infty},w}_{p,q}(\mathbb{R}^d) \approx B_{p,q}^{s,w}(\mathbb{R}^d)$. In the next subsection, we prove the interpolation formula for $VB^{\delta, w}_{p,q}(\mathbb{R}^d)$. This interpolation formula allows us to consider the Besov space whose smoothness is between the normal Besov space and Besov space with mixed smoothness.

\subsection{\texorpdfstring{Equivalence of Besov spaces $VB^{sl_{\infty},w}_{p,q}(\mathbb{R}^d) \approx B_{p,q}^{s,w}(\mathbb{R}^d)$}{Lg}
}
\label{infty_norm}

\begin{lemma}
\label{new_old}
If $q \geq 1$, then
\begin{equation}
VB^{s l_{\infty},w}_{p,q}(\mathbb{R}^d)
    \approx
B_{p,q}^{s,w}(\mathbb{R}^d)
\end{equation}
\end{lemma}
\begin{proof}
We have following $\Psi_{k}$ and relation (\ref{minus}). The remaining part of the proof is highly similar to theorem.\ref{equivalence}. 
Define $\Psi_{k}$ as follows:
\begin{equation*}
    \Psi_{k}(x) := \sum_{|\bar{j}|_{\infty}=k} \boldsymbol{\psi}_{\bar{j}}(x)
\end{equation*}
where $\boldsymbol{\psi}_{\bar{j}}$ is given in (\ref{bolpsi}). The following equality holds if $k\geq 1$.
\begin{eqnarray}
\label{minus}
   \Psi_{k}(x) &=& 
   \prod_{i=1}^{d} \left\{ \sum_{j_i=0}^k \psi_{j_i}^{i}(x_i) \right\}
   -
   \prod_{i=1}^{d} \left\{ \sum_{j_i=0}^{k-1} \psi_{j_i}^{i}(x_i) \right\} \nonumber \\
   &=&
   \prod_{i=1}^{d} 2^{k} \psi_0^{i}(2^{k}x_i)
   -
   \prod_{i=1}^{d} 2^{k-1} \psi_0^{i}(2^{k-1}x_i) \nonumber \\
   &=&
   2^{dk} \Psi_{0}(2^k x) - 2^{d(k-1)} \Psi_{0}(2^{k-1} x)
\end{eqnarray}
As for the moment condition, $D^{\bar{s}} \widehat{\Psi}_{k}(0) = 0$ if $k \geq 1$ and $|\bar{s}|_{\infty} \leq L$. To prove $VB^{sl_{\infty}}_{p,q} \approx B^{s,w}_{p,q}$, we only show
\begin{equation}
    \left(\sum_{\bar{j} \in \mathbb{N}_0^d}2^{qsl_{\infty}(\bar{j})}
    \| \psi_{\bar{j}}*f \|_{L^p_w}^q \right)^{1/q}
    \approx
    \left( \sum_{k =0}^{\infty} 2^{qsk} \|  \Psi_{k}*f \|_{L^p_w}^q \right)^{1/q}
\end{equation}
First, we prove the following expression:
\begin{equation}
    \left(\sum_{\bar{j} \in \mathbb{N}_0^d}2^{qsl_{\infty}(\bar{j})}
    \| \psi_{\bar{j}}*f \|_{L^p_w}^q \right)^{1/q}
    \lesssim
    \left( \sum_{k =0}^{\infty} 2^{qsk} \|  \Psi_{k}*f \|_{L^p_w}^q \right)^{1/q}
\end{equation}
By lemma.\ref{kernel}, $\{\Phi_k \}_{k=0}^{\infty}$ exists such that:
\begin{equation}
    f = \sum_{k=0}^{\infty} \Phi_k * \Psi_k*f
\end{equation}
Let $\bar{k}_0 = (k,k,\cdots,k) \in \mathbb{N}_0^d$. Repeat the same discussion in the proof of theorem.\ref{equivalence} to obtain
\begin{equation}
    |\psi_{\bar{j}}*f(x)|
    \lesssim
    \sum_{k =0}^{\infty} 2^{-L|\bar{j}-\bar{k}_0|_1} m_{\bar{j},\bar{k}_0,L+1}^{-1}*(|\Psi_k * f|)(x)
\end{equation}
The remainder of this proof is the same as that of theorem.\ref{equivalence} wherein we put $\delta = s l_{\infty}$. We obtain
\begin{equation}
    \sum_{\bar{j} \in \mathbb{N}_0^d}2^{qsl_{\infty}(\bar{j})}
    \| \psi_{\bar{j}}*f \|_{L^p_w}^q
    \lesssim
    \sum_{k =0}^{\infty} 2^{qsk} \|  \Psi_{k}*f \|_{L^p_w}^q
\end{equation}
This is the desired inequality.

On the other hand, the following inverse inequality is proved in the same procedure.
\begin{equation}
    \left( \sum_{k =0}^{\infty} 2^{qsk} \|  \Psi_{k}*f \|_{L^p_w}^q \right)^{1/q}
    \lesssim
    \left(\sum_{\bar{j} \in \mathbb{N}_0^d}2^{qsl_{\infty}(\bar{j})}
    \| \psi_{\bar{j}}*f \|_{L^p_w}^q \right)^{1/q}
\end{equation}
\end{proof}

\subsection{Real interpolation}
$\quad \; \; \;$ To obtain the intermediate space between normal Besov space and mixed Besov space, real interpolation is used. Here, $K$-method typically are used to obtain simple expressions of that intermediate space. We begin with the definition of $K$-method and some remarks necessary for the proof of Lemma \ref{interpolation}.

\begin{define} \cite{Bergh76}
    \label{realInterpolation}
Let $(X_0,X_1)$ be a compatible couple of quasi-Banach spaces, and let $\theta$,$p$ satisfy $0<\theta <1$, $0<p\leq \infty$
\begin{enumerate}
\item
For $x=x_0+x_1\in X_0+X_1$

\begin{equation*}
    \|x\|_{(X_0,X_1)_{\theta,p}}:=\left(\int_0^\infty
    (t^{-\theta}K(t,x,X_1,X_2))^p\frac{dt}{t})\right)^{\frac{1}{p}}
\end{equation*}
where

\begin{equation*}
    K(t,x,X_1,X_2):=\inf\{\|x_0\|_{X_0}+t\|x_1\|_{X_1}:x_0\in X_0,x_1\in X_1, x=x_1+x_0\}
\end{equation*}
Inferior takes all combinations of $x=x_1+x_0$($x_0\in X_0$,$x_1\in X_1$)
\item
The real interpolation quasi-Banach spaces $(X_0,X_1)_{\theta,p}$ is the subspace of $X_0+X_1$ defined by the following expression:
\begin{equation}
(X_0,X_1)_{\theta,p}:=\{x\in X_0+X_1:\|x\|_{(X_0,X_1)_{\theta,p}}<\infty\}
\end{equation}
\end{enumerate}
\end{define}

\begin{remark}
\label{mono}
\begin{enumerate}
\item[1.]
Function $K(t,x,X_1,X_2)$ is a monotone increasing function with respect to $t$. Assume $t_0 < t$. By the definition of $\inf$, for arbitrary $\epsilon>0$, $x_1 ^{(t)} \in X_1$, $x_2^{(t)} \in X_2$ exists such that $x = x_1^{(t)} + x_2^{(t)}$ and 
\begin{equation}
    \| x_1  \|_{X_1} + t \| x_2  \|_{X_2}
    \leq
    K(t,x,X_1,X_2) + \epsilon
\end{equation}
Then, we have the following expression:
\begin{eqnarray}
    K(t_0,x,X_1,X_2) 
    &\leq& \| x_1^{(t)} \|_{X_1}  +t_0 \| x_2^{(t)} \|_{X_2}\nonumber \\
    &\leq& \| x_1^{(t)} \|_{X_1}  +t \| x_2^{(t)} \|_{X_2}  \nonumber \\
    &\leq&
     K(t,x,X_1,X_2) + \epsilon
\end{eqnarray}
Because $\epsilon>0$ is arbitrary
\begin{equation}
    K(t_0,x,X_1,X_2) \leq K(t,x,X_1,X_2)
\end{equation}
\item[2.] If $0 < K(t,f,X_1,X_2)$, for $epsilon>0$ such that $\epsilon < K(t,f,X_1,X_2)$, $x_1^{(t)} \in X_1$ and $x_2^{(t)} \in X_2$ exist such that $x = x_1^{(t)} + x_2^{(t)}$ and
\begin{eqnarray}
    \| x_1^{(t)} \|_{X_1}  +t \| x_2^{(t)} \|_{X_2} 
    &\leq& K(t,x,X_1,X_2) + \epsilon \nonumber \\
    &\leq& 2K(t,x,X_1,X_2)
\end{eqnarray}
\end{enumerate}
\end{remark}

The following lemma is the interpolation result between $VB^{\delta_1,w}_{p,q}(\mathbb{R}^d)$ and $VB^{\delta_2,w}_{p,q}(\mathbb{R}^d)$($\delta_1 < \delta_2$). 
\begin{lemma}
\label{interpolation}
Let $\delta_1$ and $\delta_2$ be norms such that $\delta_1 \geq \delta_2$ For $0<\theta<1$, we have the following expression:
\begin{equation}
    \left( VB_{p,q}^{\delta_1}(\mathbb{R}^d), VB_{p,q}^{\delta_2}(\mathbb{R}^d)\right)_{q,\theta} \approx VB_{p,q}^{\delta}(\mathbb{R}^d)
\end{equation}
where $\delta = (1-\theta) \delta_1+ \theta\delta_2$.
\end{lemma}
\begin{proof}
Proof of this lemma is an extension of Theorem 4.25 in \cite{Sawano18}. \\
\noindent First, we prove $VB_{p,q}^{\delta,w}(\mathbb{R}^d) \subset \left( VB_{p,q}^{\delta_1,w}(\mathbb{R}^d), VB_{p,q}^{\delta_2,w}(\mathbb{R}^d)\right)_{q,\theta}$. Split the domain of integral of $K-functional$ into $[0,1]$ and $\mathbb{R} \backslash [0,1]$.
Because $\| f | VB^{\delta_2,w}_{p,q} \| \leq \| f | VB^{\delta,w}_{p,q} \|$,
\begin{eqnarray}
\label{two}
    \left( \int_{0}^1 (t^{-\theta} K(t,f,VB^{\delta_1,w}_{p,q},VB^{\delta_2,w}_{p,q}) )^{q}\frac{dt}{t} \right)^{1/q}
    &\leq&
    \left( \int_{0}^1 (t^{-\theta} t \|f |VB^{\delta_2,w}_{p,q}\| )^{q}\frac{dt}{t} \right)^{1/q} \nonumber \\
    &\lesssim&
    \|f |VB^{\delta_2,w}_{p,q}\| \nonumber \\
    &\leq&
    \|f |VB^{\delta,w}_{p,q}\|
\end{eqnarray}
By Theorem \ref{wav}, $f$ has the following wavelet expansion:
\begin{equation}
    f = \sum_{\bar{k} \in \mathbb{N}_0^d} \sum_{\bar{m} \in \mathbb{Z}_0^d} \lambda_{\bar{k},\bar{m}} \boldsymbol{\psi}_{\bar{k},\bar{m}}
\end{equation}
and decompose $f$ into $f^{(j)}_1$ and $f^{(j)}_2$, where
\begin{equation}
    f^{(j)}_1
    :=
    \sum_{
    \substack{\bar{k} \in \mathbb{N}_0^d
    \\
    \delta_1(\bar{k}) - \delta_2(\bar{k}) \leq j}
    } 
    \sum_{\bar{m} \in \mathbb{Z}_0^d} \lambda_{\bar{k},\bar{m}} \boldsymbol{\psi}_{\bar{k},\bar{m}}
    \qquad
    f^{(j)}_2
    :=
    \sum_{
    \substack{\bar{k} \in \mathbb{N}_0^d
    \\
    \delta_1(\bar{k}) - \delta_2(\bar{k}) > j}
    } 
    \sum_{\bar{m} \in \mathbb{Z}_0^d} \lambda_{\bar{k},\bar{m}} \boldsymbol{\psi}_{\bar{k},\bar{m}}
\end{equation}
We then have the following expression:
\begin{eqnarray}
    &&\left( \int_1^{\infty} (t^{-\theta}  K(t,f,VB^{\delta_1,w}_{p,q},VB^{\delta_2,w}_{p,q}))^q \frac{dt}{t} \right)^{1/q} \nonumber \\
    &=&
    \left(
    \sum_{j=0}^{\infty}
    \int_{2^j}^{2^{j+1}} (t^{-\theta} K(t,f,VB^{\delta_1,w}_{p,q},VB^{\delta_2,w}_{p,q}))^q \frac{dt}{t} \right)^{1/q} \nonumber \\
    &\lesssim&
    \left(
    \sum_{j=1}^{\infty}
    2^{-\theta j q} K(2^{j},f,VB^{\delta_1,w}_{p,q},VB^{\delta_2,w}_{p,q})^q \right)^{1/q} \nonumber \\
    &\lesssim&
     \left(
    \sum_{j=1}^{\infty}
    2^{-\theta j q}
    (\|f^{(j)}_1 | VB^{\delta_1,w}_{p,q} \|
    +
    2^{j}
    \|f^{(j)}_2 | VB^{\delta_2,w}_{p,q} \|)^q \right)^{1/q} \nonumber \\
    &\lesssim&
    \left (
    \sum_{j=1}^{\infty}
    2^{-\theta j q}
    \|f^{(j)}_1 | VB^{\delta_1,w}_{p,q}\|^q
    \right) ^{1/q}
    +
    \left (
    \sum_{j=1}^{\infty}
    2^{(1-\theta) j q}
    \|f^{(j)}_2 | VB^{\delta_2,w}_{p,q} \|^q
    \right)^{1/q}
\end{eqnarray}
In the third inequality, the monotonicity of function $K$ with respect to $t$ is used. For more details, please refer to Remark \ref{mono}-1. 
By (i-b) of Theorem \ref{wav},
\begin{eqnarray}
    &&\left( \int_1^{\infty} (t^{-\theta}  K(t,f,VB^{\delta_1,w}_{p,q},VB^{\delta_2,w}_{p,q}))^q \frac{dt}{t} \right)^{1/q} \nonumber \\
    &\lesssim&
    \left (
    \sum_{j=1}^{\infty}
    2^{-\theta j q}
    \sum_{
    \substack{\bar{k} \in \mathbb{N}_0^d
    \\
    \delta_1(\bar{k}) - \delta_2(\bar{k}) \leq j}
    }
    2^{\delta_1(\bar{k})q} \left\|\sum_{\bar{m} \in \mathbb{Z}^d}  \lambda_{\bar{k},\bar{m}} 
    2^{ |\bar{k}|_1/2}
    \boldsymbol{\chi}_{\bar{k},\bar{m}} \right\|_{L^p_w}^q
    \right)^{1/q} \nonumber \\
    && \qquad \qquad  \qquad \qquad \qquad \qquad +
    \left (
    \sum_{j=1}^{\infty}
    2^{-(1-\theta) j q}
    \sum_{
    \substack{\bar{k} \in \mathbb{N}_0^d
    \\
    \delta_1(\bar{k}) - \delta_2(\bar{k}) > j}
    }
    2^{\delta_2(\bar{k})q} \left\|\sum_{\bar{m} \in \mathbb{Z}^d}  \lambda_{\bar{k},\bar{m}} 
    2^{ |\bar{k}|_1/2}
    \boldsymbol{\chi}_{\bar{k},\bar{m}} \right\|_{L^p_w}^q
    \right)^{1/q} \nonumber \\
    &=&
    \left (
    \sum_{\bar{k} \in \mathbb{N}_0^d}
    2^{\delta(\bar{k})q} \left\|\sum_{\bar{m} \in \mathbb{Z}^d}  \lambda_{\bar{k},\bar{m}} 
    2^{ |\bar{k}|_1/2}
    \boldsymbol{\chi}_{\bar{k},\bar{m}} \right\|_{L^p_w}^q
    \sum_{\substack{j \in \mathbb{N}_0
    \\
    \delta_1(\bar{k}) - \delta_2(\bar{k}) \leq j}}
    2^{-\theta j q - \delta(\bar{k})q + \delta_1(\bar{k})q}
    \right)^{1/q} \nonumber \\
    && \qquad \qquad \qquad +
    \left (
    \sum_{\bar{k} \in \mathbb{N}_0^d}
    2^{\delta(\bar{k})q} \left\|\sum_{\bar{m} \in \mathbb{Z}^d}  \lambda_{\bar{k},\bar{m}}
    2^{ |\bar{k}|_1/2}
    \boldsymbol{\chi}_{\bar{k},\bar{m}} \right\|_{L^p_w}^q
    \sum_{\substack{j \in \mathbb{N}_0
    \\
    \delta_1(\bar{k}) - \delta_2(\bar{k}) > j}}
    2^{-(1-\theta) j q - \delta(\bar{k})q + \delta_2(\bar{k})q}
    \right)^{1/q} \nonumber \\
    &=&
    \left (
    \sum_{\bar{k} \in \mathbb{N}_0^d}
    2^{\delta(\bar{k})q} \left\|\sum_{\bar{m} \in \mathbb{Z}^d}  \lambda_{\bar{k},\bar{m}} 
    2^{ |\bar{k}|_1/2}
    \boldsymbol{\chi}_{\bar{k},\bar{m}} \right\|_{L^p_w}^q
    \sum_{\substack{j \in \mathbb{N}_0
    \\
    \delta_1(\bar{k}) - \delta_2(\bar{k}) \leq j}}
    2^{-\theta (j - \delta_1(\bar{k}) + \delta_2(\bar{k}) )q}
    \right)^{1/q} \nonumber \\
    && \qquad \qquad \qquad +
    \left (
    \sum_{\bar{k} \in \mathbb{N}_0^d}
    2^{\delta(\bar{k})q} \left\|\sum_{\bar{m} \in \mathbb{Z}^d}  \lambda_{\bar{k},\bar{m}}
    2^{ |\bar{k}|_1/2}
    \boldsymbol{\chi}_{\bar{k},\bar{m}} \right\|_{L^p_w}^q
    \sum_{\substack{j \in \mathbb{N}_0
    \\
    \delta_1(\bar{k}) - \delta_2(\bar{k}) > j}}
    2^{-(1-\theta) (j + \delta_1(\bar{k}) - \delta_2(\bar{k}))q}
    \right)^{1/q} \nonumber \\
    &&\approx
    \left (
    \sum_{\bar{k} \in \mathbb{N}_0^d}
    2^{\delta(\bar{k})q} \left\|\sum_{\bar{m} \in \mathbb{Z}^d}  \lambda_{\bar{k},\bar{m}}
    2^{ |\bar{k}|_1/2}
    \boldsymbol{\chi}_{\bar{k},\bar{m}} \right\|_{L^p_w}^q \right)^{1/q}
\end{eqnarray}
By (ii-a) of Theorem \ref{wav},
\begin{equation}
\label{one}
    \left( \int_1^{\infty} (t^{-\theta}  K(t,f,VB^{\delta_1,w}_{p,q},VB^{\delta_2,w}_{p,q}))^q \frac{dt}{t} \right)^{1/q}
    \lesssim
    \|f |VB^{\delta,w}_{p,q}\|
\end{equation}
From (\ref{two}) and (\ref{one}), we have the following expression:
\begin{equation}
    \|f |\left( VB_{p,q}^{\delta_1,w}, VB_{p,q}^{\delta_2,w}\right)_{q,\theta}\|
    \lesssim
    \|f |VB^{\delta,w}_{p,q}\|
\end{equation}

Second, we prove $\left( VB_{p,q}^{\delta_1,w}(\mathbb{R}^d), VB_{p,q}^{\delta_2,w}(\mathbb{R}^d)\right)_{q,\theta} \subset VB_{p,q}^{\delta,w}(\mathbb{R}^d)$. By Remark \ref{mono}-2, a decomposition $f=f_1^{(t)} + f_2^{(t)}$ exists for each $t>0$ such that
\begin{equation}
    \|f_1^{(t)} | VB^{\delta_1,w}_{p,q}\|
    +
    t \|f_2^{(t)} | VB^{\delta_2,w}_{p,q}\|
    \leq
    2 K(t,f,VB^{\delta_1,w}_{p,q},VB^{\delta_2,w}_{p,q})
\end{equation}
Then, by triangle inequality and the definition of norm $\| \cdot | VB^{\delta_1,w}_{p,q}\|$ and $\| \cdot | VB^{\delta_2,w}_{p,q}\|$,
\begin{eqnarray}
    &&\|f | VB_{p,q}^{\delta,w} \|
    \leq
    \|f_1^{(2^j)} | VB_{p,q}^{\delta,w} \|
    +
    \|f_2^{(2^j)}  | VB_{p,q}^{\delta,w} \|
    \nonumber \\
    &=&
    \left(
    \sum_{\bar{k} \in \mathbb{N}_0^d}
    2^{\delta(\bar{k})q} \| \boldsymbol{\varphi}_{\bar{k}} * f_1^{(2^j)} \|^q_{L^p_w}
    \right)^{1/q}
    +
    \left(
    \sum_{\bar{k} \in \mathbb{N}_0^d}
    2^{\delta(\bar{k})q} \| \boldsymbol{\varphi}_{\bar{k}} * f_2^{(2^j)} \|^q_{L^p_w}
    \right)^{1/q} \nonumber \\
    &\lesssim&
    \left(
    \sum_{\bar{k} \in \mathbb{N}_0^d}
    2^{\delta(\bar{k})q} \| \boldsymbol{\varphi}_{\bar{k}} * f_1^{(2^j)} \|^q_{L^p_w}
    +
    \sum_{\bar{k} \in \mathbb{N}_0^d}
    2^{\delta(\bar{k})q} \| \boldsymbol{\varphi}_{\bar{k}} * f_2^{(2^j)} \|^q_{L^p_w}
    \right)^{1/q} \nonumber \\
    &=&
    \left(
    \sum_{\bar{k} \in \mathbb{N}_0^d}
    2^{\delta_1(\bar{k})q} 2^{- \theta q (\delta_1(\bar{k}) - \delta_2(\bar{k}))} \| \boldsymbol{\varphi}_{\bar{k}} * f_1^{(2^j)} \|^q_{L^p_w}
    +
    \sum_{\bar{k} \in \mathbb{N}_0^d}
    2^{\delta_2(\bar{k})q} 2^{(1-\theta) q (\delta_1(\bar{k}) - \delta_2(\bar{k}))}
    \| \boldsymbol{\varphi}_{\bar{k}} * f_2^{(2^j)} \|^q_{L^p_w}
    \right)^{1/q}
\end{eqnarray}
For all $\bar{k} \in \mathbb{N}_0^d$, $j \in \mathbb{N}_0$ exist such that $j \leq \delta_1(\bar{k})- \delta_2(\bar{k}) \leq j+1$.
\begin{eqnarray}
    &&\|f | VB_{p,q}^{\delta,w} \|
    \nonumber \\
    &\lesssim&
    \left(
    \sum_{j=0}^{\infty}
    2^{-\theta j q}
    \left\{
    \sum_{\substack{\bar{k} \in \mathbb{N}_0^d \\
    j \leq \delta_1(\bar{k}) - \delta_2(\bar{k}) \leq j+1}}
    2^{q \delta_1(\bar{k})} \| \boldsymbol{\varphi}_{\bar{k}} * f_1^{(2^j)} \|^q_{L^p_w}
    +
    2^{qj}
    \sum_{\substack{\bar{k} \in \mathbb{N}_0^d \\ j \leq \delta_1(\bar{k}) - \delta_2(\bar{k}) \leq j+1}}
    2^{q \delta_2(\bar{k})} \|\boldsymbol{\varphi}_{\bar{k}} * f_2^{(2^j)} \|^q_{L^p_w}
    \right\}
    \right)^{1/q} \nonumber \\
    &\leq&
    \left(
    \sum_{j=0}^{\infty}
    2^{-\theta j q}
    \left\{
    \|f_1^{(2^j)} |VB^{\delta_1,w}_{p,q}\|^q
    +
    2^{jq}
    \|f_2^{(2^j)} |VB^{\delta_1,w}_{p,q}\|^q
    \right\}
    \right)^{1/q} \nonumber \\
    &\lesssim&
    \left(
    \sum_{j=0}^{\infty}
    2^{-\theta j q}
    \left\{
    \|f_1^{(2^j)} |VB^{\delta_1,w}_{p,q}\|
    +
    2^{j}
    \|f_2^{(2^j)} |VB^{\delta_1,w}_{p,q}\|
    \right\}^q
    \right)^{1/q} \nonumber \\
    &\lesssim&
    \left(
    \sum_{j=0}^{\infty}
    2^{-\theta j q}
    K(2^j,f,VB^{\delta_1,w}_{p,q},VB^{\delta_2,w}_{p,q})^q
    \right)^{1/q} \nonumber \\
    &\lesssim&
    \left(
    \sum_{j=0}^{\infty} \int_{2^j}^{2^{}j+1} 
    t^{-\theta q} K(t,f,VB^{\delta_1,w}_{p,q},VB^{\delta_2,w}_{p,q})^q
    \frac{dt}{t}
    \right)^{1/q} \nonumber \\
    &\leq&
    \| f |  \left( VB_{p,q}^{\delta_1}, VB_{p,q}^{\delta_2} \right)_{q,\theta} \|
\end{eqnarray}
The fourth inequality comes from definition of $K$-functional and monotonicity of $K$-functional in Remark \ref{mono}-2 is used in the fifth inequality.
Thus,
\begin{equation}
    \|f | VB_{p,q}^{\delta,w} \|
    \lesssim
    \| f |  \left( VB_{p,q}^{\delta_1}, VB_{p,q}^{\delta_2} \right)_{q,\theta} \|
\end{equation}

\end{proof}

\section{Sparse Grids on \texorpdfstring{$VB^{\delta,w}_{p,q}(\mathbb{R}^d)$}{Lg}
}
\label{section8}
$\quad \; \; \;$ Based on the results from previous section, we detail a formula of sparse grids on $VB^{\delta,w}_{p,q}(\mathbb{R}^d)$. As explained in the introduction, sparse grids are obtained by considering the approximation error. Let $\tilde{f}$ be an approximation of $f$ as follows:
\begin{equation}
    \tilde{f} = \sum_{\bar{j} \in \mathbb{N}_0^d} \sum_{\bar{m} \in G_{\bar{j}}}
    \lambda_{\bar{j},\bar{m}} \psi_{\bar{j},\bar{m}}
\end{equation}
where $G_{\bar{j}}$ is a finite subset of $\mathbb{Z}^d$  and $G_{\bar{j}} \neq \varnothing$ with finite $\bar{j}$. Next, an error estimate becomes the following:
\begin{eqnarray*}
&&
\|f -\tilde{f} | VB_{p,q}^{\delta_1} \|
=
\|\sum_{\bar{j} \in \mathbb{N}_0^d} \sum_{\bar{m} \in \mathbb{Z}^d \backslash G_{\bar{j}}} \lambda_{\bar{j},\bar{m}} \psi_{\bar{j},\bar{m}} | VB_{p,q}^{\delta_1} \| \nonumber \\
&\approx&
\left( \sum_{\bar{j}\in \mathbb{N}_0^d} 2^{\delta_1(\bar{j})q}
    \left\|\sum_{\bar{m} \in \mathbb{Z}^d \backslash G_{\bar{j}}} \lambda_{\bar{j},\bar{m}} \boldsymbol{\chi}_{\bar{j},\bar{m}} \right\|_p^q \right)^{1/q}
=
\left( \sum_{\bar{j}\in \mathbb{N}_0^d} 2^{\delta_1(\bar{j})q}
    \left( \int_{\mathbb{R}^d} |\sum_{\bar{m} \in \mathbb{Z}^d \backslash G_{\bar{j}}} \lambda_{\bar{j},\bar{m}}
    2^{ |\bar{j}|_1/2}
    \boldsymbol{\chi}_{\bar{j},\bar{m}} |^p dx \right)^{q/p} \right)^{1/q} \nonumber \\
&=&
\left( \sum_{\bar{j}\in \mathbb{N}_0^d} 2^{\delta_1(\bar{j})q}
    \left(\sum_{\bar{n} \in \mathbb{Z}^d} \int_{Q_{\bar{j},\bar{n}}} |\sum_{\bar{m} \in \mathbb{Z}^d \backslash G_{\bar{j}}} \lambda_{\bar{j},\bar{m}}
    2^{ |\bar{j}|_1/2}
    \boldsymbol{\chi}_{\bar{j},\bar{m}} |^p dx \right)^{q/p} \right)^{1/q} \nonumber \\
&=&
\left( \sum_{\bar{j}\in \mathbb{N}_0^d} 2^{\delta_1(\bar{j})q}
    \left(\sum_{\bar{n} \in \mathbb{Z}^d}
    |\lambda_{\bar{j},\bar{n}}|^p |Q_{\bar{j},\bar{n}}| \right)^{q/p} \right)^{1/q}
=
\left( \sum_{\bar{j}\in \mathbb{N}_0^d} 2^{\delta_1(\bar{j})q}
    \left(\sum_{\bar{m} \in \mathbb{Z}^d}
    |\lambda_{\bar{j},\bar{m}}|^p  
    w(Q_{\bar{j},\bar{m}})
    \frac{|Q_{\bar{j},\bar{m}}|}{w(Q_{\bar{j},\bar{m}})}
    \right)^{q/p}
    \right)^{1/q} \nonumber \\
&\leq&
\left( \sum_{\bar{j}\in \mathbb{N}_0^d} 2^{\delta_1(\bar{j})q}
\max_{\bar{m} \in \mathbb{Z}^d \backslash G_{\bar{j}}}
\left(\frac{|Q_{\bar{j},\bar{m}}|}{w(Q_{\bar{j},\bar{m}})}\right)^{q/p}
\left(\sum_{\bar{m} \in \mathbb{Z}^d}
    |\lambda_{\bar{j},\bar{m}}|^p w(Q_{\bar{j},\bar{m}}) \right)^{q/p} \right)^{1/q} \nonumber \\
&\leq&
\max_{\bar{j} \in \mathbb{N}_0^d} \left(2^{-(\delta_2(\bar{j})-\delta_1(\bar{j}))} 
\max_{\bar{m} \in \mathbb{Z}^d \backslash G_{\bar{j}}}
\left(\frac{|Q_{\bar{j},\bar{m}}|}{w(Q_{\bar{j},\bar{m}})}\right)^{1/p}
\right)
\left( \sum_{\bar{j}\in \mathbb{N}_0^d} 2^{\delta_2(\bar{j})q}
\left(\frac{|Q_{\bar{j},\bar{m}}|}{w(Q_{\bar{j},\bar{m}})}\right)^{q/p}
\left(\sum_{\bar{m} \in \mathbb{Z}^d}
    |\lambda_{\bar{j},\bar{m}}|^p w(Q_{\bar{j},\bar{m}}) \right)^{q/p} \right)^{1/q} \nonumber \\
&=&
\max_{\bar{j} \in \mathbb{N}_0^d} \left(2^{-(\delta_2(\bar{j})-\delta_1(\bar{j}))} 
\max_{\bar{m} \in \mathbb{Z}^d \backslash G_{\bar{j}}}
\left(\frac{|Q_{\bar{j},\bar{m}}|}{w(Q_{\bar{j},\bar{m}})}\right)^{1/p}
\right)
\left( \sum_{\bar{j}\in \mathbb{N}_0^d} 2^{\delta_1(\bar{j})q} \left\|\sum_{\bar{m} \in \mathbb{Z}^d \backslash G_{\bar{j}}} \lambda_{\bar{j},\bar{m}} 
2^{ |\bar{j}|_1/2}
\boldsymbol{\chi}_{\bar{j},\bar{m}} \right\|_{L^p_w}^q \right)^{1/q} \nonumber \\
&\approx&
\max_{\bar{j} \in \mathbb{N}_0^d} \left(2^{-(\delta_2(\bar{j})-\delta_1(\bar{j}))} 
\max_{\bar{m} \in \mathbb{Z}^d \backslash G_{\bar{j}}}
\left(\frac{|Q_{\bar{j},\bar{m}}|}{w(Q_{\bar{j},\bar{m}})}\right)^{1/p}
\right)
\|f| VB_{p,q}^{\delta_1} \|
\end{eqnarray*}

In case $w(x) = e^{b|x|_1}$,
\begin{equation}
    \|f - \sum_{\bar{j}\in \mathbb{N}_0^d} \sum_{\bar{m} \in G_{\bar{j}}} \lambda_{\bar{j},\bar{m}} \boldsymbol{\psi}_{\bar{j},\bar{m}} | VB_{p,q}^{\delta_1} \|
    \lesssim
    \max_{\bar{j}\in \mathbb{N}_0^d } \left(2^{-(\delta_2(\bar{j})-\delta_1(\bar{j}))}
    \max_{\bar{m} \in \mathbb{Z}^d \backslash G_{\bar{j}}} 
    \exp\left(-b\left|\frac{\bar{m}}{2^{\bar{j}}}\right| \right)
    \right)
    \|f|VB_{p,p}^{\delta,w}(\mathbb{R}^d)\|
\end{equation}
Minimum $G_{\bar{j}}$ is computed as follows:
\begin{equation}
    2^{-(\delta_2(\bar{j})-\delta_1(\bar{j}))}
    \max_{\bar{m} \in \mathbb{Z}^d \backslash G_{\bar{j}}} 
    \exp \left(-b\left|\frac{\bar{m}}{2^{\bar{j}}}\right| \right)
    <
    \epsilon
\end{equation}
for all $\bar{j} \in \mathbb{N}_0^d$ when a positive real number $\epsilon$ is given. Here, $G_{\bar{j}}$ is the null set when
$2^{-(\delta_2(\bar{j})-\delta_1(\bar{j}))} \epsilon^{-1} < 1$. In case $2^{-(\delta_2(\bar{j})-\delta_1(\bar{j}))} \epsilon^{-1} \geq 1$, $\bar{m}$ is obtained by the following expression:
\begin{equation}
    \log(2^{-(\delta_2(\bar{j})-\delta_1(\bar{j}))} \epsilon^{-1} ) <  \min_{\bar{m} \in \mathbb{Z}^d \backslash G_{\bar{j}}}
    b\left|\frac{\bar{m}}{2^{\bar{j}}}\right|_1
\end{equation}
Here, $G_{\bar{j}}$ is obtained by calculating all $\bar{m}$ such that:
\begin{equation}
    b\left|\frac{\bar{m}}{2^{\bar{j}}}\right|_1
    \leq
    \log(2^{-(\delta_2(\bar{j})-\delta_1(\bar{j}))} \epsilon^{-1} )
\end{equation}
for each $\bar{j}$ such that $2^{-(\delta_2(\bar{j})-\delta_1(\bar{j}))} \epsilon^{-1} \geq 1$. 

Figure \ref{fig:weight_figures} and Figure \ref{fig:smooth_fig} detail the centers of the wavelet when dimension is $2$ by x marks. Figure \ref{fig:weight_figures} contains figures when weight parameter $b$ changes. Figure \ref{fig:smooth_fig} contains figures when norm $\delta$, which controls smoothness changes.

\begin{figure}[htbp]
 \begin{minipage}[b]{0.49\linewidth}
  \centering
  \includegraphics[keepaspectratio, scale=0.49,bb= 0 0 400 180]
  {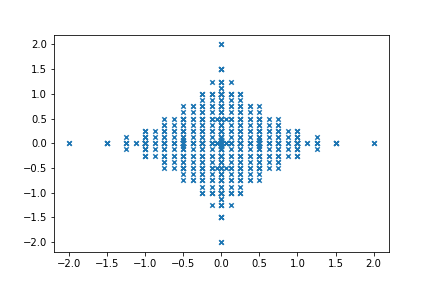}
 \end{minipage}
 \begin{minipage}[b]{0.49\linewidth}
  \centering
  \includegraphics[keepaspectratio, scale=0.49,bb= 0 0 400 180]
  {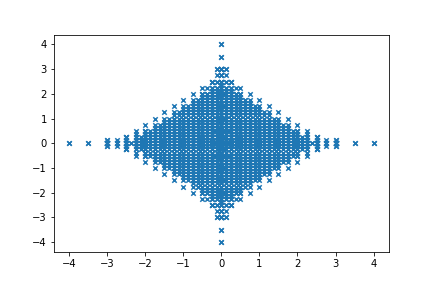}
 \end{minipage}
 \begin{minipage}[b]{0.49\linewidth}
  \centering
  \includegraphics[keepaspectratio, scale=0.49,bb= 0 0 400 300]
  {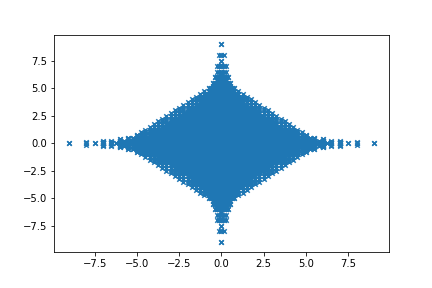}
 \end{minipage}
 \begin{minipage}[b]{0.49\linewidth}
  \centering
  \includegraphics[keepaspectratio, scale=0.49,bb= 0 0 400 300]
  {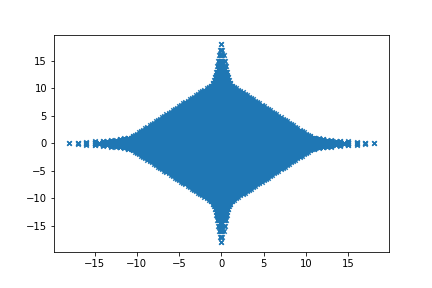}
 \end{minipage}
 \caption{Centers of the wavelet in case $b=4$(upper left), $b=2$(upper left), $b=1$(lower right), $b=0.5$(lower right) when $\delta_2(\cdot)= 0.25 |\cdot|_1 + 0.75|\cdot|_{\infty}$, $\delta_1(\cdot)=0$ and $\epsilon=0.03$.
 }\label{fig:weight_figures}
\end{figure}

\begin{figure}[htbp]
\label{smooth_fig}
 \begin{minipage}[b]{0.49\linewidth}
  \centering
  \includegraphics[keepaspectratio, scale=0.49,bb= 0 0 400 180]
  {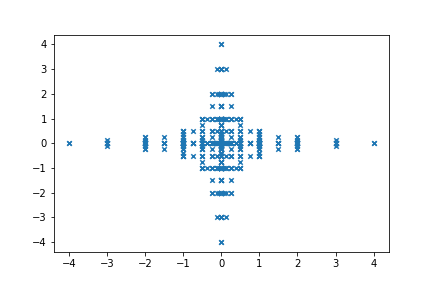}
 \end{minipage}
 \begin{minipage}[b]{0.49\linewidth}
  \centering
  \includegraphics[keepaspectratio, scale=0.49,bb= 0 0 400 180]
  {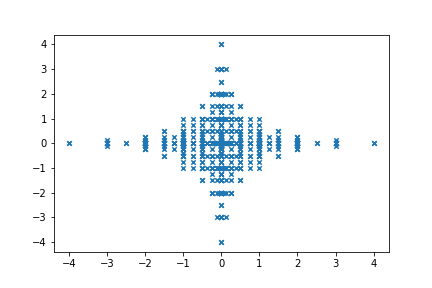}
 \end{minipage}
 \begin{minipage}[b]{0.49\linewidth}
  \centering
  \includegraphics[keepaspectratio, scale=0.49,bb= 0 0 400 300]
  {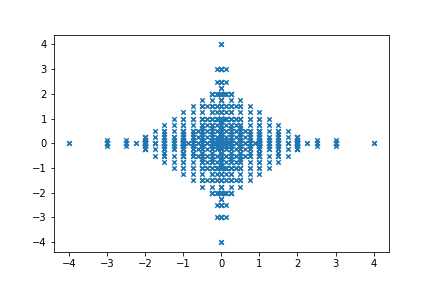}
 \end{minipage}
 \begin{minipage}[b]{0.49\linewidth}
  \centering
  \includegraphics[keepaspectratio, scale=0.49,bb= 0 0 400 300]
  {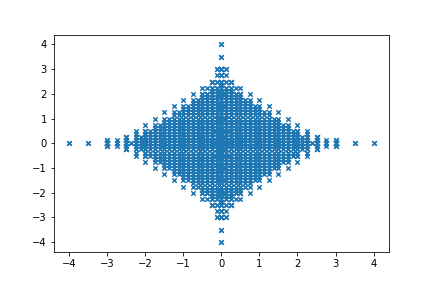}
 \end{minipage}
 \caption{Centers of the wavelet in the case $\delta_2(\cdot)= (1-\theta)|\cdot|_1 + \theta|\cdot|_{\infty}$, $\delta_1(\cdot)=0$ and  $\theta=0$(upper left), $\theta=0.25$(upper right), $\theta=0.5$(lower left), $\theta=0.75$(lower right) when $b=2$ and $\epsilon=0.03$.}\label{fig:smooth_fig}
\end{figure}

\newpage

\appendix
\section{Appendix}
$\quad \; \; \;$ Technical results used in this paper are detailed in this appendix. 
\subsection{Some results on \texorpdfstring{$X_b^{\bar{s}}(\mathbb{R}^d)$}{Lg}}
$\quad \; \; \;$ Every lemma in this subsection is intended to prove reproducing formula on $X_b^{\bar{s}}(\mathbb{R}^d)$ in the next subsection. Only Lemma \ref{repro} is directly dedicated to the proof of the reproducing formula on $X_b^{\bar{s}}(\mathbb{R}^d)$ in the next subsection. Following lemmas from Lemma \ref{bounded} to Lemma \ref{fril} are preparations for Lemma \ref{repro}.
\begin{lemma}

\label{bounded}
Let $\bar{s} \geq \bar{2}$ and $\varphi \in X_{b}^{\bar{s}}(\mathbb{R}^d)$, then we have the following expression:
\begin{equation}
    |e^{b'|x|_1} \varphi(x)| \leq C 
\end{equation}
for $b' < b$
where $C$ is a some positive constant.
\end{lemma}
\begin{proof}
Let $\bar{b'} = (b',b',\cdots,b')$. 
We give a proof when $\bar{x} \geq \bar{0}$ because proof can be done in the similar manner in other cases.
\begin{eqnarray}
    |e^{b' |x|_1} \varphi(x)| &\leq& (2 \pi)^{-d/2}\|\mathcal{F}[e^{b' \sum_{j=1}^d (\cdot_j)}\varphi] \|_{1} \nonumber \\
    &=& (2 \pi)^{-d/2}\|\prod_{j=1}^{d} \{ b^2+(\cdot_j +ib')^2 \}^{-1} \prod_{j=1}^{d} \{ b^2 +(\cdot_j +ib')^2 \} \mathcal{F}[\varphi](\cdot + i \bar{b'}) \|_{L^1} \nonumber \\
    &\leq&
    (2 \pi)^{-d/2}
    \|\prod_{j=1}^{d} \{ b^2+(\cdot_j +ib')^{2} \}^{-1}\|_{L^2}
    \|\prod_{j=1}^{d} \{ b^2+(\cdot_j +ib')^2 \}
    \mathcal{F}[\varphi](\cdot + i\bar{b'}) \|_{L^2}
    \nonumber \\
    &=&
    (2 \pi)^{-d/2} (2b)^{-d}
    \|e^{-b |\cdot|_1}
    e^{b' \sum_{j=1}^d (\cdot_j)}\|_{L^2}
    \|e^{b \sum_{j=1}^d (\cdot_j)} \prod_{j=1}^d(b^2+D_j^2) \varphi \|_{L^2}
\end{eqnarray}
In the last inequality, the fact $\mathcal{F}[e^{-\alpha|\cdot|}](\xi) = \frac{2 \alpha}{\alpha^2 + \xi^2}$ is used. The term $\|e^{-b |\cdot|_1}
    e^{b' \sum_{j=1}^d (\cdot_j)}\|_{L^2}$ is bounded because ${b'} < {b}$. $\|e^{b \sum_{j=1}^d (\cdot_j)} \prod_{j=1}^d(b+D_j^2) \varphi \|_{L^2}$ is also bounded due to the definition of $X_{b}^{\bar{s}}(\mathbb{R}^d)$ ($\bar{s}\geq \bar{2}$).

\end{proof}

\begin{lemma}
\label{conv}
    Let $s \geq 1$ and $\varphi,\psi \in X_b^s(\mathbb{R})$, then $\varphi*\psi \in X_b^s(\mathbb{R})$.
\end{lemma}
\begin{proof}
For all $0\leq \alpha, \beta \leq s$ proving the following is sufficient:
\begin{equation}
    \| e^{b|\cdot|} (\cdot)^{\alpha} \partial^{\beta} (\varphi*\psi) \|_{L^2} < \infty
\end{equation}
From $(x)^{2\alpha} \lesssim (x-y)^{2\alpha}+(y)^{2\alpha}$ and $2^{b|x|} \leq e^{b|x-y|} 2^{b|y|}$
\begin{eqnarray}
    \| e^{b|\cdot|} (\cdot)^{\alpha} \partial^{\beta} (\varphi*\psi) \|_{L^2} 
    &\leq&
    \left(
    \int_{\mathbb{R}}
    \left(
    e^{b|x|}
    |x|^{\alpha}
    \int_{\mathbb{R}}
    |\varphi(y)| |\partial^{\beta} \psi(x-y)|dy
    \right)^2
    dx
    \right)^{1/2}
    \nonumber\\
    &\lesssim&
    \left(
    \int_{\mathbb{R}}
    \left(
    \int_{\mathbb{R}}
    |x-y|^{\alpha}
    e^{b|y|}
    e^{b|x-y|}
    |\varphi(y)| |\partial^{\beta} \psi(x-y)|dy
    \right)^2
    dx
    \right)^{1/2} \nonumber \\
    && \qquad \qquad \qquad \qquad
    +
    \left(
    \int_{\mathbb{R}}
    \left(
    \int_{\mathbb{R}}
    |y|^{\alpha}
    e^{b|y|}
    e^{b|x-y|}
    |\varphi(y)| |\partial^{\beta} \psi(x-y)|dy
    \right)^2
    dx
    \right)^{1/2}
    \nonumber\\
    &\lesssim&
    \| (|e^{b|\cdot|}  \varphi|)*(|e^{b|\cdot|} (\cdot)^{\alpha} \partial^{\beta} \psi|) \|_{L^2}
    +
    \| (|e^{b|\cdot|} (\cdot)^{\alpha}  \varphi|)*(|e^{b|\cdot|}  \partial^{\beta} \psi|) \|_{L^2}
    \nonumber \\
    &\leq&
     \|e^{b|\cdot|} \varphi \|_{L^1}
     \|e^{b|\cdot|}(\cdot)^{\alpha} \partial^{\beta} \psi \|_{L^2} 
     +
     \|e^{b|\cdot|}  \partial^{\beta} \psi \|_{L^1}
    \|e^{b|\cdot|} (\cdot)^{\alpha} \varphi \|_{L^2}
     \nonumber \\
     &\leq&
     \|(1 + |\cdot|)^{-1} \|_{L^2}
     \|e^{b|\cdot|} (1 + |\cdot|)\varphi \|_{L^2}
     \|e^{b|\cdot|}(\cdot)^{\alpha} \partial^{\beta} \psi \|_{L^2} \nonumber \\
     && \qquad +
     \|(1 + |\cdot|)^{-1} \|_{L^2}
     \|e^{b|\cdot|} (1 + |\cdot|) \partial^{\beta}\psi \|_{L^2}
    \|e^{b|\cdot|} (\cdot)^{\alpha} \partial^{\beta} \varphi\|_{L^2} \nonumber \\
     &<& \infty
\end{eqnarray}
In the second inequality, young inequality is used and in the fourth inequality, the Hölder inequality is used.
\end{proof}

\begin{lemma}
\label{fril}
Let $\varphi \in X_{b}^{\bar{s}+{1}}(\mathbb{R}^d)$ and
$
    G_{\epsilon}(x) := \frac{1}{(\sqrt{2 \pi \epsilon})^d} \exp(-\frac{|x|^2}{2 \epsilon})
$. Then,
\begin{enumerate}[(i)]
    \item$G_{\epsilon} * (\partial^{\beta} \varphi) \xrightarrow{\epsilon \rightarrow 0} (\partial^{\beta} \varphi)$ uniformly if $0 \leq \beta \leq s$
    \item $G_{\epsilon} * \varphi \xrightarrow{\epsilon \rightarrow 0} \varphi
    \qquad
    \text{in}
    \quad
    X_{b'}^{\bar{s}}(\mathbb{R}^d)$ for all $0<b'<b$
    \item
    When $0 \leq \beta \leq s$
    \begin{enumerate}
        \item $\sup_{x \in \mathbb{R}} |e^{b|x|} \partial^{\beta} (G_{\epsilon} * \varphi)(x)| \leq C_1 $
        \item $\sup_{x \in \mathbb{R}} |e^{b|x|} \partial^{\beta+1} (G_{\epsilon} * \varphi)(x)| \leq C_2 \epsilon^{-3/4}$
    \end{enumerate}
    where $C_1$ and $C_2$ are positive real numbers that do not depend on $\epsilon$.
\end{enumerate}
Furthermore, $G_{\epsilon} * \varphi$ is a $C^{\infty}$-function.
\end{lemma}
\begin{proof}
(i)
\begin{eqnarray}
    |G_{\epsilon}*(\partial^{\beta} \varphi)(x)-(\partial^{\beta} \varphi)(x)|
    &\leq&
    (2 \pi)^{-d/2}
    \|\mathcal{F}[G_{\epsilon}*(\partial^{\beta} \varphi)(\cdot)-(\partial^{\beta} \varphi)(\cdot)] \|_{L^1} \nonumber \\
    &=&
    (2 \pi)^{-d/2}
    \|e^{-\epsilon |\cdot|^2/2}\widehat{(\partial^{\beta} \varphi)}(\cdot)-\widehat{(\partial^{\beta} \varphi)}(\cdot) \|_{L^1}
\end{eqnarray}
where we use the fact $\mathcal{F}[e^{-a|\cdot|^2/2}](\xi) = a^{-1/2} e^{-|\xi|^2/2a}$. $|e^{-\epsilon |\xi|^2/2}\widehat{(\partial^{\beta} \varphi)}(\xi)-\widehat{(\partial^{\beta} \varphi)}(\xi)|$ is dominated by $2|\widehat{(\partial^{\beta} \varphi)}(\xi)|$ and $|\widehat{(\partial^{\beta} \varphi)}| \in L^1$ because $0 \leq \beta \leq s$. Thus, by Lebesgue's dominated convergence theorem, we have the following expression:
\begin{equation}
    \lim_{\epsilon \rightarrow 0} |G_{\epsilon}*\varphi(x)-\varphi(x)|
    \leq
    \lim_{\epsilon \rightarrow 0}
    (2 \pi)^{-d/2}
    \|e^{-\epsilon |\cdot|^2/2}\widehat{\partial^{\beta}\varphi}(\cdot)-\widehat{\partial^{\beta}\varphi}(\cdot) \|
    =0
\end{equation}

(ii) We may assume $\epsilon < 1$ because the limit $\epsilon \rightarrow 0$ is considered. Therefore, the following equation is sufficient for $0 \leq \bar{\alpha},\bar{\beta} \leq \bar{s}$.
\begin{equation}
\label{&L2}
    \lim_{\epsilon \rightarrow 0}
    \left(
    \int_{\mathbb{R}} \left| \partial^{\bar{\beta}} (G_{\epsilon} *\varphi(x))- \partial^{\bar{\beta}} \varphi(x)\right|^2 |x^{2 \bar{\alpha}}| e^{2b'|x|} dx
    \right)^{1/2}
    =0
\end{equation}
If the exchange between integral and $\lim_{\epsilon \rightarrow 0}$ is justified, (\ref{&L2}) holds from the following equation: 
\begin{equation}
    \lim_{\epsilon \rightarrow 0}
    \partial^{\bar{\beta}} (G_{\epsilon} *\varphi(x))
    =
    \lim_{\epsilon \rightarrow 0}
     G_{\epsilon} * (\partial^{\bar{\beta}}\varphi(x))
    = \partial^{\bar{\beta}} \varphi(x)
    \qquad
    \text{uniformly}
\end{equation}
which is expressed in (i).
Thus, we obtain a function that dominates $G_{\epsilon} * (\partial^{\beta}\varphi(x))$ and justify (\ref{&L2}) by  Lebesgue's dominated convergence theorem. The function dominating $|G_{\epsilon} * (\partial^{\beta}\varphi(x))||x|^{\alpha} e^{b'|x|}$ is expressed as follows. If $\bar{x} \geq \bar{0}$
\begin{eqnarray}
     |e^{\bar{b} \cdot \bar{x}} G_{\epsilon} * (\partial^{\beta}\varphi)(x)|
    &\leq&
    (2 \pi)^{-d/2}
    \|\mathcal{F}[G_{\epsilon} * (\partial^{\bar{\beta}}\varphi)](\cdot+ib)\|_{L^1} \nonumber \\
    &=&
    \|e^{-\epsilon |\cdot+ib|^2/2} \mathcal{F}[\partial^{\bar{\beta}}\varphi](\cdot+ib) \|_{L^1}
    \nonumber \\
    &\leq&
    e^{db^2}
    \|\mathcal{F}[\partial^{\bar{\beta}} \varphi](\cdot+ib)\|_{L^1}
    \nonumber \\
    &<&
    \infty
\end{eqnarray}
where 
$\bar{b}=(b,b,\cdots,b) \in \mathbb{R}^d$ 
and we use 
$\mathcal{F}[e^{-a|\cdot|^2/2}](\xi) = a^{-1/2} e^{-|\xi|^2/2a}$. 
In the last inequality, 
$\| \mathcal{F}[\partial^{\beta} \varphi](\cdot + i b) \|_{L^1} \leq \| \prod_i^d (1+\cdot_i + i b)^{-1} \|_{L^2} \| \prod_i^d (1+ \cdot_i + i b ) \mathcal{F}[\partial^{\bar{\beta}} \varphi](\cdot + i b)\|_{L^2} < \infty$ 
because 
$\partial^{\bar{\beta}} \varphi \in X_b^{\bar{1}}(\mathbb{R}^d)$ . 
In the other case, similar inequalities hold. Thus, 
$|G_{\epsilon} * (\partial^{\beta}\varphi(x))||x|^{\alpha} e^{b'|x|}$ 
is dominated by 
$C|x|^{\alpha} e^{-(b-b')|x|_1}$ ($0<b'<b$) 
where 
$C$ 
is a positive constant independent of $\epsilon$. Then, by using the Lebesgue's dominated convergence theorem, (\ref{&L2}) is justified.

(iii) We may assume $\epsilon < 1$ because we take the limit $\epsilon \rightarrow 0$. We give a prove of the estimate of $|e^{b|x|_1} \partial^{\bar{\beta}+\bar{1}} (G_{\epsilon} * \varphi)(x)|$. The upper bound of $|e^{b|x|_1} \partial^{\bar{\beta}} (G_{\epsilon} * \varphi)(x)|$ is given in the same procedure.
If $\bar{x} \geq \bar{0}$,
\begin{eqnarray}
\label{sup3}
     |e^{\bar{b} \cdot \bar{x}} \partial^{\bar{\beta}+\bar{1}} (G_{\epsilon} * \varphi)(x)|
    &\leq&
    (2 \pi)^{-d/2}
    \| \mathcal{F}[ (\partial^{\bar{1}} G_{\epsilon})*(\partial^{\bar{\beta}} \varphi)](\cdot + ib) \|_{L^1}
    \nonumber \\
    &=&
    \| \mathcal{F}[ \partial^{\bar{1}} G_{\epsilon}](\cdot + ib) \mathcal{F}[ \partial^{\bar{\beta}} \varphi](\cdot + ib) \|_{L^1} \nonumber \\
    &\leq&
    \| \mathcal{F}[\partial^{\bar{1}} G_{\epsilon}](\cdot + ib) \|_{L^2}
    \| \mathcal{F}[\partial^{\bar{\beta}} \varphi](\cdot + ib)\|_{L^2} \nonumber \\
    &=&
    \| e^{\bar{b} \cdot (\cdot)} \partial^{\bar{1}} G_{\epsilon}\|_{L^2}
    \| e^{\bar{b} \cdot (\cdot)} \partial^{\bar{\beta}} \varphi\|_{L^2} \nonumber \\
    &\leq&
    \| e^{b|\cdot|_1} \partial G_{\epsilon}\|_{L^2}
    \| e^{b|\cdot|_1} \partial^{\bar{\beta}} \varphi\|_{L^2} 
\end{eqnarray}
where 
$\bar{b}=(b,b,\cdots,b) \in \mathbb{R}^d$.
In the third inequality, the Hölder inequality is used. In case other than $\bar{x} \geq \bar{0}$, similar upper bounds are obtained.

Thus, $\| e^{b|\cdot|_1} \partial^{\bar{1}} G_{\epsilon} \|_{L^2}$ can be estimated as follows:
\begin{eqnarray}
\label{comp}
    \| e^{b|\cdot|} \partial^{\bar{1}} G_{\epsilon} \|_{L^2}
    &=&
    \left(
    \int_{\mathbb{R}}
    e^{2b|x|} \left| \frac{x}{\epsilon} G_{\epsilon}(x) \right|^2
    dx
    \right)^{1/2} \nonumber \\ 
    &=&
    \left(
    \int_{\mathbb{R}}
    e^{2b\sqrt{\epsilon} |x|} \left| \frac{\sqrt{\epsilon}^dx}{\epsilon^d} G_{\epsilon}(\sqrt{\epsilon}x) \right|^2
    \sqrt{\epsilon}^d
    dx
    \right)^{1/2}
    \nonumber \\
    &=&
    \epsilon^{-3d/4}
    \left(
    \int_{\mathbb{R}}
    e^{2 \sqrt{\epsilon} b|x|} \left| x G_{1}(x) \right|^2
    dx
    \right)^{1/2} \nonumber \\
    &\leq&
    \epsilon^{-3d/4}
    \left(
    \int_{\mathbb{R}}
    e^{2b|x|} \left| x G_{1}(x) \right|^2
    dx
    \right)^{1/2}
\end{eqnarray}
In the second inequality, the change of variables $x \leftarrow x/\sqrt{\epsilon}$ is used.
From (\ref{sup3}) and (\ref{comp}), we obtain the following:
\begin{equation}
    \sup_{x \in \mathbb{R}} |e^{b|x|} \partial^{\beta+1} G_{\epsilon} * \varphi(x)|
    \leq
    C_2 \epsilon^{-3d/4}
\end{equation}
where $C_2$ is a positive constant defined by the following:
\begin{equation}
    C_2 := \| e^{b|\cdot|} \partial^{\beta} \varphi\|_{L^2} \left(
    \int_{\mathbb{R}}
    e^{2b|x|} \left| x G_{1}(x) \right|^2
    dx
    \right)^{1/2} < \infty
\end{equation}
$C_2$ does not depend on $\epsilon$.
\end{proof}

\begin{lemma}
\label{repro}
Let $0<b'<b$, and $\psi \in X_b^{\bar{s}+\bar{1}}(\mathbb{R}^d)$ such that $\int_{\mathbb{R}} \psi(x) = 1 $ and $\psi^t(x) = t \psi(tx)$. Then,
\begin{enumerate}[(i)]
    \item For all $\varphi \in X_b^{\bar{s}+\bar{1}}(\mathbb{R}^d)$, $ \psi^t*\varphi  \xrightarrow{t \rightarrow \infty} \varphi $ in $X_{b'}^{\bar{s}}(\mathbb{R}^d)$
    \item For all $f \in X_{b'}^{\bar{s}}(\mathbb{R}^d)'$, $ \psi^t*f \xrightarrow{t \rightarrow \infty} f $ in $X_{b'}^{\bar{s}}(\mathbb{R}^d)'$
\end{enumerate}
\end{lemma}
\begin{proof}
(i) From Lemma \ref{fril}, $\varphi^{(\epsilon)} \in C^{\infty}$ such that $\|\varphi^{(\epsilon)} - \varphi | X_{b'}^s(\mathbb{R}^d)\| \xrightarrow{\epsilon \rightarrow 0}0$ exists. By triangle inequality, we have the following expression:
\begin{eqnarray}
\label{triangle}
&&\|\psi^t*\varphi - \varphi | X_{b'}^{\bar{s}}(\mathbb{R})\| \nonumber \\
&\leq&
\|\psi^t*\varphi - \psi^t*\varphi^{(\epsilon)}| X_{b'}^{\bar{s}}(\mathbb{R}^d)\|
+
\|\varphi - \varphi^{(\epsilon)} | X_{b'}^{\bar{s}}(\mathbb{R}^d)\|
+
\|\psi^t*\varphi^{(\epsilon)} - \varphi^{(\epsilon)} | X_{b'}^{\bar{s}}(\mathbb{R}^d)\|
\end{eqnarray}
We estimate three terms in (\ref{triangle}) one by one.

(a) The first term $\|\psi^t*\varphi - \psi^t*\varphi^{(\epsilon)}| X_{b'}^{\bar{s}}(\mathbb{R}^d)\|$ in (\ref{triangle})
\begin{eqnarray}
\label{estimation}
    &&\|\psi^t*\varphi - \psi^t*\varphi^{(\epsilon)}| X_{b'}^s(\mathbb{R}^d)\| \nonumber \\
    &=&
    \sum_{0 \leq \bar{\alpha}, \bar{\beta} \leq \bar{s}}
    \left( \int_{\mathbb{R}^d} |x|^{2 \bar{\alpha}} |\partial^{\beta} (\psi^t*\varphi) - \partial^{\beta}(\psi^t*\varphi^{(\epsilon)})|^2 e^{2b'|x|_1} dx\right)^{1/2} \nonumber \\
    &=&
    \sum_{0 \leq \bar{\alpha}, \bar{\beta} \leq \bar{s}}
    \left( \int_{\mathbb{R}^d} |x|^{2 \alpha} | \psi^t*(\partial^{\bar{\beta}}\varphi) - \psi^t*(\partial^{\bar{\beta}}\varphi^{(\epsilon)})|^2 e^{2b'|x|_1} dx\right)^{1/2} \nonumber \\
    &\lesssim&
    \sum_{0 \leq \bar{\alpha}, \bar{\beta} \leq \bar{s}}
    \left( \int_{\mathbb{R}^d}   \left\{ \left| \psi^t e^{b'|\cdot|_1}\right|*\left|(\cdot)^{\bar{\alpha}}(\partial^{\bar{\beta}}\varphi - \partial^{\bar{\beta}}\varphi^{(\epsilon)}) e^{b'|\cdot|_1} \right| \right\}^2  dx\right)^{1/2} \nonumber \\
    &\leq&
    \sum_{0 \leq \alpha, \beta \leq s}
    \left\|\psi^t e^{b'|\cdot|_1}\right\|_{L^1}
    \left\|(\cdot)^{\bar{\alpha}}(\partial^{\bar{\beta}}\varphi - \partial^{\bar{\beta}}\varphi^{(\epsilon)}) e^{b'|\cdot|_1} \right\|_{L^2} \nonumber \\
    &\leq&
    \left\|\prod_{i=1}^d (1+|\cdot_i|)^{-1}\right\|_{L^2}
     \left\|\prod_{i=1}^d (1+|\cdot_i|) \psi^t e^{b'|\cdot|}\right\|_{L^2}
    \|\varphi^{(\epsilon)} - \varphi | X_{b'}^s(\mathbb{R}^d)\|
\end{eqnarray}
In the third inequality, $|x|^{2\alpha} \lesssim |x-y|^{2\alpha} +|y|^{2\alpha}$ and $e^{b'|x|} \leq e^{b'|x-y|}e^{b'|y|}$ are used. In the fourth inequality, young inequality for convolution is used. In the last inequality, the Hölder inequality is used. The last term in (\ref{estimation}) goes to $0$ as $\epsilon \rightarrow 0$ due to $\|\varphi^{(\epsilon)} - \varphi | X_{b'}^s(\mathbb{R}^d)\|$.

(b) The second term $\|\varphi - \varphi^{(\epsilon)} | X_{b'}^s(\mathbb{R}^d)\|$ in (\ref{triangle}) becomes arbitrary small as $\epsilon$ goes to $0$ due to the choice of $\varphi^{(\epsilon)}$.

(c) The third term $\|\psi^t*\varphi^{(\epsilon)} - \varphi^{(\epsilon)} | X_{b'}^s(\mathbb{R})\|$ in (\ref{triangle})
We obtain a similar procedure as in Proposition 2.14. given in \cite{Schott98}. Start from the estimation of $|\psi^t*(\partial^{\beta}\varphi^{(\epsilon)})(x) - (\partial^{\beta}\varphi^{(\epsilon)})(x)|$. Remember that $\int_{\mathbb{R}} \psi(x) = 1 $, then we have the following expression:
\begin{eqnarray}
\label{over}
|\psi^t*(\partial^{\beta}\varphi^{(\epsilon)})(x) - (\partial^{\beta}\varphi^{(\epsilon)})(x)|
\leq
\int_{\mathbb{R}^d} 
\left|
(\partial^{\beta}\varphi^{(\epsilon)})(x) -
(\partial^{\beta}\varphi^{(\epsilon)})(x-\frac{y}{t})\right|
\left| \psi(y) \right| dy
\end{eqnarray}
We split the integral in (\ref{over}) over $|y|_{2} \leq t^{1/4}$ and over $|y|_{2} > t^{1/4}$. On $|y|_{2} \leq t^{1/4}$, by the mean value theorem
\begin{eqnarray}
    \left|(\partial^{\bar{\beta}}\varphi^{(\epsilon)})(x) - (\partial^{\bar{\beta}}\varphi^{(\epsilon)})(x-\frac{y}{t})\right|
    &\leq&
    \frac{|y|}{t^d} \sup_{\bar{0} \leq \bar{\tau} \leq \bar{1}} \left|\partial^{\bar{\beta}+\bar{1}}\varphi^{(\epsilon)}(x-\frac{\bar{\tau} y}{t})\right|
    \nonumber \\
    &\leq&
    \frac{|y|}{t^d}
    \sup_{\bar{0} \leq \bar{\tau} \leq \bar{1}} \left|e^{-b|x-\frac{\bar{\tau} y}{t}|_1}  e^{b|x-\frac{\bar{\tau} y}{t}|}  \partial^{\bar{\beta}+\bar{1}}\varphi^{(\epsilon)}(x-\frac{\bar{\tau} y}{t})\right|
    \nonumber \\
    &\leq&
    \frac{|y|}{t^d}
    e^{-b|x|_1} 
    e^{b|\frac{ y}{t}|_1}
    \sup_{\bar{0} \leq \bar{\tau} \leq \bar{1}} \left|e^{b|x-\frac{\bar{\tau} y}{t}|}  \partial^{\beta+1}\varphi^{(\epsilon)}(x-\frac{\bar{\tau} y}{t})\right|
\end{eqnarray}
The upper bound of $ \sup_{\bar{0} \leq \bar{\tau} \leq \bar{1}} \left|e^{b|x-\frac{y}{t}|_1}  \partial^{\bar{\beta}+\bar{1}}\varphi^{(\epsilon)}(x-\frac{\bar{\tau} y}{t})\right| \leq C_2 \epsilon^{-3d/4}$ is given in lemma.\ref{fril} of (ii-b). Thus, the integral of (\ref{over}) over $|y|_{2} \leq t^{1/4}$ bounded by the following expression:
\begin{eqnarray}
\label{first}
    \int_{|y|_2 \leq t^{1/4}} \left| (\partial^{\beta}\varphi^{(\epsilon)})(x) - (\partial^{\beta}\varphi^{(\epsilon)})(x-\frac{y}{t})\right| \left| \psi(y) \right| dy 
    &\leq&
    C_2 e^{-b|x|_1}  \epsilon^{-3d/4} \int_{|y|_2 \leq t^{1/4}} \frac{|y|}{t^d} e^{b|\frac{y}{t}|_1} dy
    \nonumber \\
    &\leq&
    C_2' e^{-b|x|} \epsilon^{-3d/4} \times
    t^{d/4} \frac{t^{d/4}}{t^d}
    e^{b|\frac{t^{1/4}}{t}|_1}
    \nonumber \\
    &\leq&
    C_2' e^{b} e^{-b|x|_1} t^{-d/2} \epsilon^{-3d/4}
\end{eqnarray}
We assume $t > 2$ to claim $e^{b|t^{1/4}/t|} \leq e^{b}$ because we consider the limit $t \rightarrow \infty$.
Furthermore, let $\kappa$ be a real number such that $0 < \kappa < 1$. In this case, $1/t < 1 - \kappa/2$.
On $|y|_{2} > t^{1/4}$
\begin{eqnarray}
&& \int_{|y|_{2} > t^{1/4}} 
\left|
(\partial^{\bar{\beta}}\varphi^{(\epsilon)})(x) -
(\partial^{\bar{\beta}}\varphi^{(\epsilon)})(x-\frac{y}{t})\right|
\left| \psi(y) \right| dy \nonumber \\
&=&
\int_{|y|_{2} > t^{1/4}} 
\left| e^{-b|x|_1} e^{b|x|_1}
(\partial^{\bar{\beta}}\varphi^{(\epsilon)})(x) -
e^{-b|x-\frac{y}{t}|} e^{b|x-\frac{y}{t}|}
(\partial^{\bar{\beta}}\varphi^{(\epsilon)})(x-\frac{y}{t})\right|
\left| \psi(y) \right| dy \nonumber \\
&\leq&
e^{-b|x|_1} \cdot 2 \sup_{x \in \mathbb{R}} \left| e^{b|x|_1}
(\partial^{\bar{\beta}}\varphi^{(\epsilon)})(x) \right|
\int_{|y|_{2} > t^{1/4}} \left| e^{b|y|_1/t} \psi(y) \right| dy \nonumber \\
\label{second}
&\leq&
e^{-b|x|_1} \cdot 2 \sup_{x \in \mathbb{R}^d} \left| e^{b|x|_1}
(\partial^{\bar{\beta}}\varphi^{(\epsilon)})(x) \right|
\int_{|y|_{2} > t^{1/4}} \left| 2^{b|y|(1-\kappa/2)} \psi(y) \right| dy \nonumber \\
&\leq&
e^{-b|x|_1}
\cdot 2
\sup_{x \in \mathbb{R}^d} \left| e^{b|x|_1}
(\partial^{\bar{\beta}}\varphi^{(\epsilon)})(x) \right|
\left(
\int_{|y|_{2} > t^{1/4}} \left| \psi(y) \right|^2 e^{2b|y|_1} dy \right)^{1/2}
\left(
\int_{|y|_{2} > t^{1/4}} e^{-b\kappa|y|} dy \right)^{1/2} \nonumber \\
&\leq&
e^{-b|x|_1}
\cdot 2\left\| \psi | X_b^{\bar{0}}(\mathbb{R}^d) \right\|
\sup_{x \in \mathbb{R}^d} \left| e^{b|x|_1}
(\partial^{\bar{\beta}}\varphi^{(\epsilon)})(x) \right|
\left(
\int_{|y|_{2} > t^{1/4}} e^{-b\kappa|y|_1} dy \right)^{1/2}
\end{eqnarray}
In the last inequality, the Hölder inequality is used. $\sup_{x \in \mathbb{R}} \left| e^{b|x|_1}
(\partial^{\bar{\beta}}\varphi^{(\epsilon)})(x) \right|$ is bounded by some positive constant that does not depend on $\epsilon$ by lemma.\ref{fril} of (ii-a) when $0\leq \beta \leq s$. Thus, from (\ref{first}) and (\ref{second}), the following expression is obtained:
\begin{eqnarray}
    &&\int_{\mathbb{R}}
    |x|^{2 \alpha} |\psi^t*(\partial^{\beta}\varphi^{(\epsilon)})(x) - (\partial^{\beta}\varphi^{(\epsilon)})(x)|^2 e^{2b'|x|_1} dx \nonumber \\
    & \lesssim&
    \max
    \left\{t^{-d/2} \epsilon^{-3d/4}, \left(\int_{|y|_{2} > t^{d/4}} e^{-\kappa|y|_1} dy \right)^{1/2} \right\}
    \int_{\mathbb{R}} |x|^{2 \alpha} e^{-2(b-b')|x|} dx
\end{eqnarray}

Considering (a),(b), and (c) and let $t= \epsilon^{-2}$, we make $\|\psi^t*\varphi - \varphi | X_{b'}^s(\mathbb{R})\|$ arbitrary small for some $t$.

(ii) Let $\varphi \in X_b^{\bar{s}+\bar{1}}(\mathbb{R}^d)$, $f \in X_{b'}^{\bar{s}}(\mathbb{R}^d)'$ and $\tilde{\psi^t}(x) = \psi^t(-x)$
\begin{eqnarray}
|\langle f-\psi^t*f, \varphi \rangle|
&=&
|\langle f, \varphi - \tilde{\psi^t}*\varphi \rangle| \nonumber \\
&\leq&
\| f | X_{b'}^{\bar{s}}(\mathbb{R}^d)' \| \|\varphi - \tilde{\psi^t}*\varphi |X_{b'}^{\bar{s}} (\mathbb{R}^d) \| \nonumber \\
&\xrightarrow{t \rightarrow \infty}& 0
\end{eqnarray}
\end{proof}

\subsection{Reproducing formula on \texorpdfstring{$X_b^{\bar{s}}(\mathbb{R}^d)$}{Lg}}

\label{appendix_2}
$\quad \; \; \;$ Following result is a modification of Theorem 1.6 in \cite{Rychkov01} so that reproducing formula fits to our setting. 
\begin{lemma}
\label{kernel}
\begin{enumerate}[(1)]
    \item Let $s \geq 1$, $\varphi_0 \in X_b^{s+1}(\mathbb{R})$ and $\varphi(x) := \varphi_0(x)- 2^{-1} \varphi_0(2^{-1}x)$. Then, there exist $\psi_0, \psi \in X_b^{s+1}(\mathbb{R})$
satisfying
\begin{enumerate}
    \item $(D^{\alpha} \widehat{\psi})(0)=0$ for $\alpha = 0,\cdots,L$ where $L$ is an arbitrary natural number.
    \item
    $\left(\sum_{j=0}^{N} \varphi_j * \psi_j \right)*f \xrightarrow{N \rightarrow \infty} f$ in $X_{b'}^{s}(\mathbb{R})'$ for all $f \in X_{b'}^s(\mathbb{R})'$
\end{enumerate}
for all $b \in \mathbb{R}$ such that $0<b'<b$.
\item Let $s \geq 1$, $\varphi_0^i \in X_b^{s+1}(\mathbb{R})$ and $\varphi^i(x_i) := \varphi_0^i(x_i)- 2^{-1} \varphi_0^i(2^{-1}x_i)$ and $\psi_0^i, \psi^i \in X_b^{s+1}(\mathbb{R})$ are functions constructed from $\varphi_0^i$ and $\varphi^i$ in \textit{(1)}. Let $\boldsymbol{\varphi}_{\bar{j}}$ and $\boldsymbol{\psi}_{\bar{j}}$ be the following:
\begin{equation}
    \boldsymbol{\varphi}_{\bar{j}}(x)
    := \prod_{i=1}^{d} \varphi_{j_i}^{i}(x_i)
    \qquad
    \boldsymbol{\psi}_{\bar{j}}(x)
    := \prod_{i=1}^{d} \psi_{j_i}^{i}(x_i)
\end{equation}
The following expression is then obtained:
\begin{equation}
\label{dtimes}
    \left(\sum_{|\bar{j}|_{\infty} \leq N, \bar{j} \in \mathbb{N}_0^d} \boldsymbol{\varphi}_{\bar{j}} * \boldsymbol{\psi}_{\bar{j}} \right)*f \xrightarrow{N \rightarrow \infty} f \quad \text{in} \quad X_{b'}^{\bar{s}}(\mathbb{R}^d)' \quad \text{for all} \quad f \in X_{b'}^{\bar{s}}(\mathbb{R}^d)'
\end{equation}

\end{enumerate}
\end{lemma}
\begin{proof}
(proof of (1))
This proof is almost identical to the proof of theorem 1.6 in \cite{Rychkov01}.
We may assume $\int_{\mathbb{R}} \varphi_0*\varphi_0(x) dx = 1$.
Let
\begin{eqnarray}
    && g_0(x) := \varphi_0 * \varphi_0(x) \\
    \label{g}
    && g(x) := g_0(x) - 2^{-1} g_0(2^{-1} x)
    = \varphi *( \varphi_0(\cdot) + 2^{-1} \varphi_0(2^{-1}\cdot)) \\
    && g_j(x)=2^j g(2^j x) \\
    &&g_j(x) = \varphi_j *( 2^j  \varphi_0(2^j \cdot) + 2^{j-1} \varphi_0(2^{j-1}\cdot))
\end{eqnarray}
Here, $g(x)=\varphi *( \varphi_0(\cdot) + 2^{-1} \varphi_0(2^{-1}\cdot))(x)$ can be calculated as follows:
\begin{eqnarray}
    g(x) 
        &=& g_0(x) - 2^{-1}g(2^{-1}x) \nonumber \\
        &=& \varphi_0*\varphi_0(x) - 2^{-1}\varphi_0*\varphi_0(2^{-1}x) \nonumber \\
        &=&
        \int_{\mathbb{R}} \varphi_0(y)\varphi_0(x-y)dy
        -
        2^{-1}
        \int_{\mathbb{R}} \varphi_0(y)\varphi_0(x/2-y)dy \nonumber \\
        &=&
        \int_{\mathbb{R}} \varphi_0(y)\varphi_0(x-y)dy
        -
        2^{-2}
        \int_{\mathbb{R}} \varphi_0(2^{-1}y)\varphi_0(2^{-1}x-2^{-1}y)dy \nonumber \\
        && \quad
        + 2^{-1} \int_{\mathbb{R}} \varphi_0(2^{-1}y) \varphi_0(x-y)dy
        - 2^{-1} \int_{\mathbb{R}} \varphi_0(2^{-1}y) \varphi_0(x-y)dy \nonumber \\
        &=&
        \int_{\mathbb{R}} (\varphi_0(y)-2^{-1}\varphi_0(2^{-1}y) )\varphi_0(x-y)dy \nonumber \\
        && \quad
        +\int_{\mathbb{R}}
        2^{-1}\varphi_0(2^{-1}y)
        (\varphi_0(x-y)
        -2^{-1}\varphi_0(2^{-1}x-2^{-1}y))dy \nonumber \\
        &=&
        \varphi*\varphi_0 + \varphi*(2^{-1}\varphi_0(2^{-1}\cdot))
        =
        \varphi *( \varphi_0(\cdot) + \varphi_0(2^{-1}\cdot))(x)
\end{eqnarray}

From Lemma \ref{repro} of (ii), we have the following expression:
\begin{equation}
    \left(\sum_{j=0}^{N} g_j(\cdot) \right) * f = 2^N \varphi_0 * \varphi_0(2^N \cdot) *f
    \xrightarrow{N \rightarrow \infty} f \qquad \text{in} \quad X_{b'}^s(\mathbb{R})'
\end{equation}
Convoluting $\sum_{j=0}^{\infty} g_j(\cdot)$ $K:=L+2$ times reproduces $f$.
\begin{equation}
    \left(\sum_{j=0}^{\infty} g_j(\cdot) \right)^{*K} * f = f
\end{equation}
where $\zeta^{*K}$ indicates $K$ times convolution of $\zeta$.

Define $G^j$, $\psi_0$ and $\psi$ by the following expression:
\begin{eqnarray}
    \widehat{G^j} &:=& 
    \sum_{m=1}^K {}_K C_m \widehat{g_j}^{m} (1 - \sum_{k=0}^{j} \widehat{g_k})^{K-m}  \\
    \widehat{\psi_0} &:=& \widehat{\varphi_0}
    \sum_{m=1}^K {}_K C_m
    \widehat{g_0}^{m-1}
    (1 - \widehat{g_0})^{K-m} \\
    \widehat{\psi_j} &:=& (\widehat{\varphi_0}(2^{-j} \cdot)+\widehat{\varphi_0}(2^{1-j} \cdot))
    \sum_{m=1}^K {}_K C_m
    \widehat{g_j}^{m-1}
    (1 - \sum_{k=0}^{j} \widehat{g_k})^{K-m}
\end{eqnarray}
Note that $G^j$, $\psi_0$ and $\psi$ satisfy
\begin{eqnarray}
    \label{G}
    &&
    G^j = \sum_{m=1}^{K} {}_K C_m g_j^{*m} * \left( \sum_{k=j+1}^{\infty} g_k \right)^{*(K-m)}  \\
    \label{psi_0}
    &&
    \psi_0 = \varphi_0 * \sum_{m=1}^{K} {}_K C_m g_0^{*(m-1)} * \left(\sum_{k=1}^{\infty} g_k \right)^{*(K-m)}
     \\
    \label{psi_j}
    &&
    \psi_j = (2^j \varphi_0(2^j \cdot) + 2^{j-1} \varphi_0(2^{j-1}\cdot))* \sum_{m=1}^{K}
    {}_K C_m 
    g_j^{*(m-1)} * (\sum_{k=j+1}^{\infty} g_k)^{*(K-m)}
\end{eqnarray}
Note $G^{0} = \psi_0*\varphi_0$ and $G^{j} = \psi_j*\varphi_j$ by
\begin{eqnarray*}
     \psi_0*\varphi_0 &=& \varphi_0 * \varphi_0 * \sum_{m=1}^{K} {}_K C_m g_0^{*(m-1)} * \left(\sum_{k=1}^{\infty} g_k \right)^{*(K-m)} \nonumber \\
          &=& 
          g_0 * \sum_{m=1}^{K} {}_K C_m g_0^{*(m-1)} * \left(\sum_{k=1}^{\infty} g_k \right)^{*(K-m)} \nonumber \\
          &=&
          \sum_{m=1}^{K} {}_K C_m g_0^{*m} * \left(\sum_{k=1}^{\infty} g_k \right)^{*(K-m)}
          \nonumber \\
          &=&G^0
\end{eqnarray*}
and
\begin{eqnarray*}
    \psi_j*\varphi_j 
    &=&
    \varphi_j*(2^j \varphi_0(2^j \cdot) + 2^{j-1} \varphi_0(2^{j-1}\cdot))* \sum_{m=1}^{K}
    {}_K C_m 
    g_j^{*(m-1)} * (\sum_{k=j+1}^{\infty} g_k)^{*(K-m)} \nonumber \\
    &=&
    \sum_{m=1}^{K}
    {}_K C_m 
    g_j^{*m} * (\sum_{k=j+1}^{\infty} g_k)^{*(K-m)} \nonumber \\
    &=&
    G^j
\end{eqnarray*}
Considering (\ref{g}), (\ref{G}), (\ref{psi_0}) and (\ref{psi_j}) into considerations, we have the following expression:
\begin{equation}
\label{reproducing_}
    f
    =
    \left(\sum_{j=0}^{\infty} g_j(\cdot) \right)^{*K} *f
    =
    \left( \sum_{j=0}^{\infty} G^j \right)*f
    =
    \left(
    \sum_{j=0}^{\infty} \psi_j * \varphi_j
    \right)*f
\end{equation}
Second equality in (\ref{reproducing_}) can be obtained from
\begin{eqnarray*}
    &&
    \left( \sum_{j=0}^{\infty} g_j \right)^{*K} \nonumber \\
    &=&
    \sum_{m=1}^K
    {}_K C_m
    g_0^{*m} *\left(\sum_{k=1}^{\infty} g_k \right)^{*(K-m)}
    +
    \sum_{m=1}^K {}_K C_m g_1^{*m}*\left(\sum_{k=2}^{\infty} g_k \right)^{*(K-m)}
    \nonumber \\
    && \qquad \qquad \qquad \qquad \qquad \qquad \qquad \qquad
    +
    \cdots
    +
    \sum_{m=1}^K {}_K C_m g_j^{*m}*\left(\sum_{k=j+1}^{\infty} g_k \right)^{*(K-m)}
    +
    \cdots \nonumber \\
    &=&
    G^{0} 
    +
    G^{1}   
    +
    \cdots
    +
    G^{j}
    +\cdots \nonumber \\
    &=&
    \sum_{j=0}^{\infty} G^j
\end{eqnarray*}
The first equality is calculated so that $j$-th term contains $\{ g_k \}_{k \geq j}$ ($g_k$s, whose index is larger than $j$.) 
For the vanishing moment of $\psi_j$ $j \geq 1$
\begin{equation}
    D^{t}\mathcal{F}[\psi_j](0)
    =
    \sum_{\alpha=0}^{t} \sum_{\beta = 0}^{\alpha} (D^{t-\alpha}\widehat{\varphi_0}(0) + 2^{2(t-\alpha)}D^{t-\alpha}\widehat{\varphi_0}(0)) \sum_{m=1}^{K} D^{\alpha-\beta}\widehat{g}_j^{m-1}(0) D^{\beta}(\sum_{k=j+1}^{\infty} \widehat{g}_k)^{(K-m)}(0)
\end{equation}
becomes 0 when $t \leq K-2(=L)$.  From Lemma \ref{conv}, convolutions of functions from $X_b^{s+1}(\mathbb{R}^d)$ reproduce a function in $X_b^{s+1}(\mathbb{R})$. Thus, $\psi_0, \psi_j \in X_b^{s+1}(\mathbb{R}^d)$. $\psi_0$ and $\psi_j$ are desired functions.

(proof of (2)) (\ref{dtimes}) is straightforward from the following calculation:
\begin{eqnarray*}
    \left(\sum_{|\bar{j}|_{\infty} \leq N, \bar{j} \in \mathbb{N}_0^d} \boldsymbol{\varphi}_{\bar{j}} * \boldsymbol{\psi}_{\bar{j}} \right)*f
    &=&
    \left(\sum_{k=0}^{N} \sum_{\substack{\bar{j} \in \mathbb{N}_0^d \\ |\bar{j}|_\infty =k}}
    \boldsymbol{\varphi}_{\bar{j}} * \boldsymbol{\psi}_{\bar{j}} \right)*f \nonumber \\
    &=&
    \left(\sum_{k=0}^{N} \sum_{\substack{\bar{j} \in \mathbb{N}_0^d \\ |\bar{j}|_\infty =k}}
    \boldsymbol{\varphi}_{\bar{j}} * \boldsymbol{\psi}_{\bar{j}} \right)*f \nonumber \\
    &=&
    2^N \boldsymbol{\varphi}_{\bar{0}} *\boldsymbol{\varphi}_{0}(2^N \cdot) *f \nonumber \\
    &\xrightarrow{N \rightarrow \infty}& f \quad \text{in} \quad X_{b'}^{\bar{s}}(\mathbb{R}^d)'
\end{eqnarray*}
In the third equality, a calculation similar to (\ref{minus}) is performed. Limit $N \rightarrow \infty$ is justified by Lemma \ref{repro}.
\end{proof}

\subsection{Technical lemmas for theorem \ref{equivalence} and theorem \ref{supequi}}
\label{appendix_3}
$\quad \; \; \;$ The topic in this subsection is Lemma \ref{tech}, which is essential for proving Theorem \ref{equivalence} and Theorem \ref{supequi}. To prove Lemma \ref{tech}, we use some results on the analytic function on the strip domain. Theorem \ref{L2} and Theorem \ref{analy} are from Theorem IX.13 and Theorem IX.14 of \cite{Michael75}. For self-containment, we detail their proofs in Appendix \ref{last_appendix} because their proofs are omitted in \cite{Michael75}. Next Lemma \ref{expdecay} is a preparation for Lemma \ref{tech} and proved by using Theorem \ref{L2} and Theorem \ref{analy}.

\begin{theorem}
\label{L2}
(Theorem IX.13 in \cite{Michael75})
Let f be in $L^2(\mathbb{R})$. Then $e^{b|x|}f \in L^2(\mathbb{R})$ for all $b < a$ if and only if $\widehat{f}$ has an analytic continuation to the set $\{\zeta | |\Im \zeta|<a\}$ with the property that for each $\eta \in \mathbb{R}$ with $|\eta| < a$,  $\widehat{f}(\cdot + i \eta) \in L^2(\mathbb{R})$ and for any $b < a$
\begin{equation}
\label{supsup}
    \sup_{|\eta| \leq b} \|\widehat{f}(\cdot + i \eta) \|_{L^2} < \infty
\end{equation}

\end{theorem}

\begin{theorem} (Theorem IX.14 in \cite{Michael75})

\label{analy}
$\widehat{f}(\xi)$ is a function that satisfies condition (i) and (ii).

(i) $\widehat{f}(\xi)$ has an analytic continuation to the set $\{\xi  \: : |\Im \xi| < a \}$ for some $a>0$.

(ii) Let $y \in \mathbb{R}$.
\begin{equation*}
\sup_{|y| < b}\| \widehat{f}(\cdot + i y)\|_1 < \infty
\end{equation*}
for any $0 < b <a$

Then, a constant $C_b$ should exists so that,
\begin{equation}
    |f(x)| \leq C_b e^{- b |x|}
\end{equation}
where $f(x)$ is a Fourier inverse of $\widehat{f}(\xi)$.
\end{theorem}

\begin{lemma}
\label{expdecay}
Let $\psi_0 \in  X_b^{L+1}(\mathbb{R})$ that satisfies $D^{s}\widehat{\psi}_0(0)=0$ for $s=0,1,...,L+1$. Next, $\Psi^{\alpha - \beta}_k( \xi)$ ($0 \leq \beta \leq \alpha \leq L+1$) is defined by the following:
\begin{equation}
    \widehat{\Psi}^{\alpha - \beta}_k( \xi)
    :=
    \frac{(D^{\alpha - \beta} \widehat{\psi_0})(\xi/2^k)}{(\xi/2^k)^{L+1-(\alpha - \beta)}}
\end{equation}
satisfies the following inequality.
\begin{equation}
\label{boundexp}
    |\mathcal{F}^{-1}[\widehat{\Psi}_k^{\alpha - \beta}](x)| < C_{b'} 2^k e^{-b'2^k|x|}
\end{equation}
for all $0<b'<b$.
\end{lemma}
\begin{proof}
Proving $\widehat{\Psi}_0^{\alpha - \beta}$ satisfies conditions (i) and (ii) in Theorem \ref{analy} is sufficient for the theorem:

(i) We observe the analyticity of $\widehat{\Psi}_0^{\alpha - \beta}( \xi)$ around $\xi = 0$ and on $\xi \neq 0$. By Theorem \ref{L2} and $\psi_0 \in X_b^{L+1}(\mathbb{R})$, $D^{\alpha - \beta} \widehat{\psi_0} (\xi)$ is analytic on strip domain $|\Im \xi| < b$, thus has Taylor expansion on $\xi=0$ with radius of convergence $b$. By the moment condition of $D^{\alpha - \beta} \widehat{\psi_0}$ ($D^{s}\widehat{\psi}_0(0)=0$ for $s=0,1,...,L+1$), $\widehat{\Psi}_0^{\alpha - \beta}( \xi)$ also has Taylor expansion around $\xi = 0$ with radius of convergence $b$ i.e.
\begin{eqnarray}
\label{series_}
    \widehat{\Psi}_0^{\alpha - \beta}( \xi)
    =
    \frac{1}{\xi^{L+1-(\alpha - \beta)}}
    \sum_{m=0}^{\infty}
    \frac{D^m (D^{\alpha - \beta} \widehat{\psi_0})(0)}{m!}\xi^{m}
    = \sum_{m=0}^{\infty} \frac{D^{m+L+1-(\alpha-\beta) }(D^{\alpha - \beta}\widehat{\psi_0})(0)}{(m+L+1-(\alpha-\beta))!} \xi^{m}
\end{eqnarray}
 When $\xi \neq 0$, $\widehat{\Psi}_0^{\alpha - \beta}( \xi)$ is obviously analytic $|$Im$\xi| < b$ because $\xi^{-L-1+(\alpha - \beta)}$ is analytic on $\xi \neq 0$ and $D^{\alpha - \beta} \widehat{\psi_k}(\xi)$ is analytic on $|$Im$ \xi| < b$ by Theorem \ref{L2}.

(ii) Let $x,y$ be a real number. If $|y| < b$, $\widehat{\Psi_0}^{\alpha - \beta}(x+iy)$ has Taylor expansion around $x+iy =0$ when $|x| < \sqrt{b^2-y^2}$. Let $\tau$ be a real number such that $0<\tau< \sqrt{b^2-y^2}$.
\begin{eqnarray*}
&&\left\| \widehat{\Psi}_0^{\alpha - \beta}(x + iy) \right\|_1 \nonumber \\
&=&
\int_{|x| < \tau} \left| \frac{(D^{\alpha - \beta} \widehat{\psi_0})(x + i y)}{(x + i y)^{L+1-(\alpha - \beta)}} \right| dx
+
\int_{|x| \geq \tau} \left| \frac{(D^{\alpha - \beta} \widehat{\psi_0})(x + i y)}{(x + i y)^{L+1-(\alpha - \beta)}} \right| dx \nonumber \\
&\leq&
2 \tau
\sum_{m=0}^{\infty}  \frac{|D^{m+L+1-(\alpha-\beta) }(D^{\alpha - \beta}\widehat{\psi_0})(0)|}{(m+L+1-(\alpha - \beta))!} |\tau + i y|^m \nonumber \\
&& \qquad \qquad \qquad \qquad
+
|\tau + i y|^{-L-1+(\alpha - \beta)}
\int_{|x| \geq \tau}
\left| (D^{\alpha - \beta} \widehat{\psi_0})(x + i y)
\right| dx \nonumber \\
\end{eqnarray*}
In the last equality, the Taylor expansion of $\widehat{\Psi_0}^{\alpha - \beta}( \xi)$ is used in the first term. The first term is bounded because the Taylor expansion of $\widehat{\psi_0}(\xi)$ converges absolutely when $|\tau + iy| < b$. The second term is also bounded because 
\begin{eqnarray}
    \int_{|x| \geq \tau} \left| (D^{\alpha - \beta} \widehat{\psi_0})(x + i y)\right|dx
    &\leq&
    \int_{\mathbb{R}}
    \left| (D^{\alpha - \beta} \widehat{\psi_0})(x + i y)\right|dx \nonumber \\
    &\leq&
    \left(
    \int_{\mathbb{R}}
    (1+|x+iy|)^{-2}dx 
    \right)^{1/2}\nonumber \\
    && \qquad \qquad \quad
    \times \left(
    \int_{\mathbb{R}}
    \left|(1+|x+iy|)
     (D^{\alpha - \beta} \widehat{\psi_0})(x + i y)\right|^2dx \right)^{1/2} \nonumber \\
     &=&
     \left(
    \int_{\mathbb{R}}
    \left|(1+|x+iy|)\right|^{-2}dx 
    \right)^{1/2}\nonumber \\
    && \qquad \qquad \quad
    \times \left(
    \int_{\mathbb{R}}
    \left|
    e^{yz}
    (1+D)
     z^{\alpha - \beta} \psi_0(z)\right|^2dz \right)^{1/2}
     \nonumber \\
     &<& \infty
\end{eqnarray}
In the last inequality, the second term is bounded because $|y|<b$. Thus, we have the following expression:
\begin{equation}
    \sup_{|y| < b'}
    \left\| \widehat{\Psi_0}^{\alpha - \beta}(\cdot + iy) \right\|_1 < \infty
\end{equation}
for all $0<b'<b$.
From (i), (ii) and Theorem \ref{analy}, the desired inequality holds.
\end{proof}

\begin{lemma}
\label{tech}
Let $L$ be a positive number and $\psi_0(x) \in X_{b'}^{L+1}(\mathbb{R})$ be a function satisfying the moment condition 
\begin{equation}
    D^{\alpha} \widehat{\psi_0} (0) = 0
    \qquad \text{for all} \quad \alpha = 0,1,\cdots,L+1
\end{equation}
Let $\zeta_0(x)\in X_{b'}^{s}(\mathbb{R})$ be a function satisfying the same moment condition above, where $\psi_0$ is replaced by $\zeta_0$.
$\{\zeta_j\}_{j\in \mathbb{N}_0}$ and $\{ \psi_j \}_{j\in \mathbb{N}_0}$ are the sequences of functions defined by $\zeta_j(x)=2^{j} \zeta_0(2^{j}x)$, $\psi_j(x)=2^{j} \psi_0(2^{j}x)$. For all $b \in \mathbb{R}$ s.t. $0< b < b'$, the following inequalities hold.
\begin{eqnarray}
 \sup_{z \in \mathbb{R}} e^{2^{k} b |z|} |\zeta_{k}*\psi_{j}(z)| (1+|2^{k}z|)^{L+1} \lesssim 2^{-(j-k)L} 2^{k} \qquad if \quad k \leq j \\
 \sup_{z \in \mathbb{R}}
 e^{2^{j} b |z|}| \zeta_{k}*\psi_{j}(z)|(1+|2^{j}z|)^{L+1} \lesssim 2^{-(k-j)L} 2^{j} \qquad if \quad k > j
\end{eqnarray}
\end{lemma}
\begin{proof}
Although the proof is given only when $k > j$, the proof for $k \leq j$ is follows a similar procedure.
\begin{eqnarray}
\label{scale}
&&\zeta_j*\psi_{k}(z) = \int_{\mathbb{R}^d} 2^{j+k} \zeta(2^j (z-x)) \psi(2^k x)dx \nonumber \\
&=& \int_{\mathbb{R}^d}  \zeta(2^{j} z - x) 2^{k}\psi(2^{k-j} x)dx
=
2^{j}\int_{\mathbb{R}^d} \zeta(2^{j} z - x) 2^{k-j}\psi(2^{k-j} x) dx \nonumber \\
&=& 2^{j} \zeta * \psi_{k-j}(2^j z)
\end{eqnarray}
Here, we consider the case only $\psi$ has subscript. The upper bound of $|2^{b|z|} |\zeta * \psi_k(z)| (1+|z|^{L+1})$ is estimated by the $L^1$ norm of its Fourier transform. Here, we estimate $|2^{b|z|}|\zeta * \psi_k(z)| (1+|z|)^{L+1}|$ when $z \leq 0$. When $z > 0$, the estimate is obtained in a similar manner.
 \begin{eqnarray}
\label{sup}
|e^{bz}|\zeta * \psi_k(z)| (1+|z|)^{L+1}|
&\leq& (2 \pi)^{-1/2}  \|\mathcal{F}[e^{b(\cdot)} \zeta * \psi_k(\cdot) (1+|\cdot|)^{L+1}] \|_{L^1} \nonumber \\
&\leq&
(2 \pi)^{1/2} {}\|(1+D)^{{L+1}}  \mathcal{F}[   e^{b(\cdot)} \zeta * \psi_k(\cdot)] \|_{L^1} \nonumber \\
&=& \| (1+D)^{L+1}[\widehat{\zeta}(\cdot + ib) \widehat{\psi_k}(\cdot + ib)] \|_{L^1} 
\end{eqnarray}
The Leibniz rule is used to expand $(1+D)^{L+1}[\widehat{\zeta}(\xi + ib) \widehat{\psi_k}(\xi + ib)]$:
\begin{eqnarray*}
\label{Leibniz}
    (1+D)^{L+1}[\widehat{\zeta}(\xi + ib) \widehat{\psi_k}(\xi + ib)] 
    &=&
    \sum_{\alpha = 0}^{L+1}
    {}_{L+1}C_{\alpha}
    \sum_{\beta=0}^{\alpha} {}_{\alpha}C_{\beta} D^{\beta} \widehat{\zeta}(\xi+ ib) D^{\alpha - \beta}\widehat{\psi_k}(\xi+ ib) \nonumber \\
    &=&
    \sum_{\alpha = 0}^{L+1}
    {}_{L+1}C_{\alpha}
    \sum_{\beta=0}^{\alpha} {}_{\alpha}C_{\beta} 
    2^{-(\alpha-\beta)k}
    D^{\beta} \widehat{\zeta}(\xi+ ib) (D^{\alpha - \beta}\widehat{\psi_0})(2^{-k} (\xi+ib))
\end{eqnarray*}
Let 
\begin{equation}
\label{expo}
    (D^{\alpha - \beta}\widehat{\psi_0})(2^{-k}\xi) =2^{-k(L+1)+k(\alpha - \beta)} \xi^{L+1-(\alpha - \beta)} \widehat{\Psi}^{\alpha - \beta}_k( \xi)
\end{equation}
where
\begin{equation}
\label{boundexp}
    |\mathcal{F}^{-1}[\widehat{\Psi}_0^{\alpha - \beta}](x)| < C_{b} e^{-b''|x|}
\end{equation}
by Lemma \ref{expdecay} for some $b''$ such that $b<b''<b'$. From (\ref{sup})
\begin{eqnarray}
\label{sup2}
&& |e^{bz}|\zeta * \psi_k(z)| (1+|z|)^{L+1}| \nonumber \\
&\lesssim&
 \sum_{\alpha = 0}^{L+1}{}_{L+1}
C_{\alpha}
\sum_{\beta=0}^{\alpha} {}_{\alpha}C_{\beta}
2^{-(\alpha-\beta)k}
\| D^{\beta} \widehat{\zeta}(\cdot+ ib) (D^{\alpha - \beta}\widehat{\psi_0})(2^{-k} (\cdot+ib))  \|_{L^1} 
\nonumber \\
&=&
 \sum_{\alpha = 0}^{L+1}{}_{L+1}
C_{\alpha}
\sum_{\beta=0}^{\alpha} {}_{\alpha}C_{\beta}
2^{-k(L+1)}
\| D^{\beta} \widehat{\zeta}(\cdot+ ib) (\cdot+ib)^{L+1-(\alpha - \beta)} \widehat{\Psi}^{\alpha - \beta}_k(\cdot+ib)  \|_{L^1} \nonumber \\
&\leq&
 \sum_{\alpha = 0}^{L+1}{}_{L+1}
C_{\alpha}
\sum_{\beta=0}^{\alpha} {}_{\alpha}C_{\beta}
2^{-k(L+1)}
\| D^{\beta} \widehat{\zeta}(\cdot+ ib) (\cdot+ib)^{L+1-(\alpha - \beta)}\|_{L^2} \|\widehat{\Psi}^{\alpha - \beta}_k(\cdot+ib)  \|_{L^2} \nonumber \\
&=&
 \sum_{\alpha = 0}^{L+1}{}_{L+1}
C_{\alpha}
\sum_{\beta=0}^{\alpha} {}_{\alpha}C_{\beta}
2^{-k(L+1)}
\| e^{b(\cdot)} D^{L+1-(\alpha - \beta)}((\cdot)^{\beta}\zeta)\|_{L^2} \|e^{b|\cdot|} \mathcal{F}^{-1}[\widehat{\Psi}^{\alpha - \beta}_k]\|_{L^2} \nonumber \\
&\leq&
2^{-kL}
 \sum_{\alpha = 0}^{L+1}{}_{L+1}
C_{\alpha}
\sum_{\beta=0}^{\alpha} {}_{\alpha}C_{\beta}
\| e^{b(\cdot)} D^{L+1-(\alpha - \beta)}((\cdot)^{\beta}\zeta)\|_{L^2} \|e^{b|\cdot|} \mathcal{F}^{-1}[\widehat{\Psi}^{\alpha - \beta}_0]\|_{L^2}
\end{eqnarray}
In the fourth equality, Parseval's identity is used. The last term in (\ref{sup2}) is bounded because of $\zeta \in X_b^{L+1}(\mathbb{R})$ and (\ref{boundexp}). In case $z > 0$, we obtain a similar inequality.
Thus, we obtain the following expression:
\begin{equation}
\label{result2}
     e^{b|z|}|\zeta * \psi_k(z)| (1+|z|)^{L+1}
     \lesssim
     2^{-kL}
\end{equation}
Here, $\lesssim$ reveals that the approximation is independent of $k$.

From (\ref{scale})
\begin{equation}
    e^{b2^{j}|z|}|\zeta_j * \psi_k(z)| (1+|2^j z|)^{L+1}
     \lesssim
     2^{-(k-j)L} 2^{j}
\end{equation}
for $j \leq k$. This is the desired inequality.
\end{proof}

\subsection{Technical lemmas for Theorem \ref{wav}}
\label{appendix_4}
$\quad \; \; \;$ In this appendix, we detail the supplemental results necessary to prove the Theorem \ref{wav}. Especially, functions constructed in Lemma \ref{delta} and the estimates in Lemma \ref{kappabound} are provided. Lemma \ref{delta} is from \cite{Schott98}.
\begin{lemma} \cite{Schott98}
\label{delta}
Let $L\in\mathbb{N}_0$. There exist functions $\Phi_L$, $\Psi_L \in C_0^{\infty}$ such that
\begin{equation}
\label{un}
    \int_{\mathbb{R}^d} \Phi_L(x) dx = 1
\end{equation}
and
\begin{equation}
\label{del}
    \Delta^{L} \Psi_L(x)
    =
    \Phi_L(x) - 2^{-d}\Phi_L(x/2)
\end{equation}
\end{lemma}

\begin{lemma}
\label{support}
Assume $\psi$ satisfies the vanishing moment, that is $D^{s}\widehat{\psi}(0) = 0$ for $s = 0,\cdots,L+1$. If supp$\: \psi \subset \{0 \leq x \leq R \}$,
the support of $\mathcal{F}^{-1}[\widehat{\psi}/(\cdot)^{L+1}])$ is also included in $\{0 \leq x \leq R \}$.
\end{lemma}
\begin{proof}
Here, $\psi$, $D^{L}\psi\in L^1$ because $\psi \in C^{L}_0$ and supp$\: \psi \subset \{0 \leq x \leq R \}$. Thus,
\begin{equation}
    |\widehat{\psi}(\xi)| \leq |\int_{0}^R e^{-i \xi x} \psi(x) dx|
    \leq
    \left\{
    \begin{array}{ll}
    e^{R Im \xi} \|\psi\|_1 & Im \xi \geq 0 \\
    \|\psi\|_1 &  Im \xi < 0
    \end{array}
    \right.
\end{equation}
and
\begin{equation}
    |\xi^{L} \widehat{\psi}(\xi)| \leq |\int_{0}^R e^{-i \xi x} D^{L} \psi(x) dx|
    \leq
    \left\{
    \begin{array}{ll}
    e^{R Im \xi} \| D^{L}\psi\|_1 & Im \xi \geq 0 \\
    \| D^{L}\psi\|_1 &  Im \xi < 0
    \end{array}
    \right.
\end{equation}
Then,
\begin{equation}
\label{polydecay}
    |\widehat{\psi}(\xi)|
    \leq
    \left\{
    \begin{array}{ll}
    \frac{C e^{R Im \xi}}{(1+|\xi|^L)}  & Im \xi \geq 0 \\
     \frac{C}{(1+|\xi|^L)} &  Im \xi < 0
    \end{array}
    \right.
    \quad \text{for all} \quad \xi \in \mathbb{C}
\end{equation}
with some constant $C>0$. $\widehat{\psi}(\xi)$ is entirely (analytic on $\mathbb{C}$) because $\psi$ has a compact support. Next, $\widehat{\psi}(\xi)/\xi^{L+1}$ is also the similar because of the vanishing moment of $\widehat{\psi}$ (similar to the proof of Lemma \ref{expdecay}). Thus, based on Cauchy's integral theorem, we have the following expression:
\begin{eqnarray}
    \mathcal{F}^{-1}[\widehat{\psi}(\cdot)/(\cdot)^{L+1}](x)
    &=&
    \int_{\mathbb{R}} e^{i \xi x} \frac{\widehat{\psi}(\xi)}{\xi^{L+1}}  d \xi \nonumber \\
    &=&
    \int_{\mathbb{R}} e^{i (\xi+i \eta) x} \frac{\widehat{\psi}(\xi+i \eta)}{(\xi+i \eta)^{L+1}}  d \xi 
\end{eqnarray}
for all $\eta \in \mathbb{R}$

If $x \geq 0$,  (\ref{polydecay}) is considered and $\eta > 0$ is assumed
\begin{eqnarray}
    |\mathcal{F}^{-1}[\widehat{\psi}(\cdot)/(\cdot)^{L+1}](x)|
    &\leq&
    \int_{\mathbb{R}} e^{- \eta x} \frac{|\widehat{\psi}(\xi+i \eta)|}{|\xi+i \eta|^{L+1}}
     d \xi 
    \nonumber \\
    &\leq&
    e^{\eta R- \eta x}
    \int_{\mathbb{R}}
    \frac{C}{(1+|\xi+i \eta|^{L})|\xi+i \eta|^{L+1}}
     d \xi 
\end{eqnarray}
To take the limit $\eta \rightarrow \infty$, let $\eta_0$ be a positive real number s.t. $0< \eta_0 < \eta$ when $\eta$ is sufficiently large. Next, if $x > R$, we have the following:
\begin{eqnarray}
    |\mathcal{F}^{-1}[\widehat{\psi}(\cdot)/(\cdot)^{L+1}](x)|
    &\leq&
    e^{\eta R- \eta x}
    \int_{\mathbb{R}}
    \frac{C}{(1+|\xi+i \eta_0|^L)|\xi+i \eta_0|^{L+1}}
     d \xi \nonumber \\
     &\xrightarrow{\eta \rightarrow \infty} 0&
\end{eqnarray}
We conclude that 
\begin{equation}
\label{x>0}
    \mathcal{F}^{-1}[\widehat{\psi}(\cdot)/(\cdot)^{L+1}](x) = 0
    \quad \text{when} \quad x > R
\end{equation}

If $x<0$, consider (\ref{polydecay}) and assume $\eta < 0$. A similar calculation gives the following result:
\begin{eqnarray}
     |\mathcal{F}^{-1}[\widehat{\psi}(\cdot)/(\cdot)^{L+1}](x)|
    &\leq&
    e^{- \eta x}
    \int_{\mathbb{R}}
    \frac{C}{(1+|\xi+i \eta_0|^L)|\xi+i \eta_0|^{L+1}}
     d \xi \nonumber \\
     &\xrightarrow{\eta \rightarrow -\infty} 0&
\end{eqnarray}
with some $\eta_0 < 0$. Thus, we have the following expression:
\begin{equation}
\label{x<0}
    \mathcal{F}^{-1}[\widehat{\psi}(\cdot)/(\cdot)^{L+1}](x) = 0
    \quad \text{for} \: \text{all} \quad x < 0
\end{equation}
From (\ref{x>0}) and (\ref{x<0}), we obtain supp$\: \mathcal{F}^{-1}[\widehat{\psi}(\cdot)/(\cdot)^{L+1}] \subset [0,R]$.
\end{proof}

\begin{lemma}
\label{support2}
Assume $\psi$ satisfies the vanishing moment, that is $D^{s}\widehat{\psi}(0) = 0$ for $s = 0,\cdots,L+1$. If supp$\: \psi \subset \{|x| \leq R \}$,
the support of $\mathcal{F}^{-1}[\widehat{\psi}/(\cdot)^{L+1}])$ is also included in $\{|x| \leq R \}$.
\end{lemma}

\begin{lemma}
\label{kappabound}
Assume $\psi_{k,m}$ satisfies the assumption \ref{wav_assumption}. $\psi$ and vanishing moment, that is $D^{s}\widehat{\psi}(0) = 0$ for $s = 0,\cdots, L+1$ 
Let $\kappa_j$ be a $C^{\infty}(\mathbb{R})$ function defined by the following expression:
\begin{equation}
    \kappa_0^L(x) = \Phi_L(x)
    \qquad
    \kappa_j^L(x) = 2^j \Delta^{L} \Psi_L(2^j x) \quad j\geq 1
\end{equation}
where $\Phi_L$ and $\Delta^{L} \Phi_L$ are functions in Lemma \ref{delta}. Assume support of $\Phi_L$ and $\Delta^{L} \Psi_L(2^j x)$ are included in cube $c_{L} Q_0$ with some positive real number $c_L >0$ and $|Q_0|=1$. Next, we have the following expression:
\begin{eqnarray}
|\kappa^{L+1}_{j}*\psi_{k,m}(x)| \lesssim
\left\{
\begin{array}{ll}
    2^{-(L+1)(j-k)}
    2^j 2^{k/2}
    \int_{c_{L} Q_{j}} \chi_{k,m}((x-y)/\gamma) dy &  j \geq k \\
    2^{-(L+1)(k-j)}
    2^j 2^{k/2}
    \int_{c_{L} Q_j} 
    \chi_{k,m}((x-y)/\gamma) dy &  j < k
\end{array}
\right.
\end{eqnarray}
where $Q_j=2^{-j} Q_0$ and $\gamma$ is a positive constant that appeared in Assumption.\ref{wav_assumption}.
\end{lemma}
\begin{proof}
(i) If $j \geq k$,
\begin{eqnarray}
    \mathcal{F}[\kappa^{L+1}_{j}*\psi_{k,m}](\xi)
    &=&
    (\xi/2^j)^{2(L+1)} \widehat{\Psi}_{L+1}(\xi/2^j) 2^{-k/2} \widehat{\psi}(\xi/2^k) e^{-i m/2^k} \nonumber \\
    &=&
    2^{-(L+1)(j-k)} (\xi/2^j)^{L+1} \widehat{\Psi}_{L+1}(\xi/2^j) 2^{-k/2} (\xi/2^k)^{L+1} \widehat{\psi}(\xi/2^k) e^{-i m/2^k}
\end{eqnarray}
Thus, we obtain the following expression:
\begin{equation}
    \kappa^{L+1}_{j}*\psi_{k,m}(x)
    =
    2^{-(L+1)(j-k)} 2^j 2^{-k/2} (\partial^{L+1}\Psi_{L+1})(2^j \cdot)*(\partial^{L+1}\psi)_{k,m}(x)
\end{equation}
where $(\partial^{L+1}\psi)_{k,m}(x)=2^k \partial^{L+1}\psi(2^k x - m)$. Therefore, the estimate of $|\kappa^{L+1}_{j}*\psi_{k,m}(x)|$ becomes
\begin{eqnarray}
    |\kappa^{L+1}_{j}*\psi_{k,m}(x)|
    &\leq&
    \sup_{z \in \mathbb{R}}
    |\partial^{L+1}\Psi_{L+1}(z)|
    2^{-(L+1)(j-k)}
    2^j 2^{-k/2}
    \int_{c_{L} Q_{j}} |(\partial^{L+1}\psi)_{k,m}(x-y)| dy \nonumber \\
    &\leq&
    c_{\psi} \sup_{z \in \mathbb{R}} |\partial^{L+1}\Psi_{L+1}(z)|
    2^{-(L+1)(j-k)}
    2^j 2^{k/2}
    \int_{c_{L} Q_{j}} \chi_{k,m}((x-y)/\gamma) dy
\end{eqnarray}

(ii)If $j < k$,

\begin{eqnarray}
    \mathcal{F}[\kappa^{L+1}_{j}*\psi_{k,m}](\xi)
    &=&
    (\xi/2^j)^{2(L+1)} \widehat{\Psi}_{L+1}(\xi/2^j)
    2^{-k/2} \widehat{\psi}(\xi/2^k) e^{-i m/2^k} \nonumber \\
    &=&
    2^{-(L+1)(k-j)} (\xi/2^j)^{3(L+1)} \widehat{\Psi}_{L+1}(\xi/2^j) 2^{-k/2} (\xi/2^k)^{-(L+1)} \widehat{\psi}(\xi/2^k) e^{-i m/2^k}
\end{eqnarray}
Thus, we obtain the following expression:
\begin{equation}
    \kappa^{L+1}_{j}*\psi_{k,m}(x)
    =
    2^{-(L+1)(k-j)}
    2^j 2^{-k/2} (\partial^{3(L+1)} \Psi_L)(2^j \cdot)*
    (\mathcal{F}^{-1}[\widehat{\psi}/(\cdot)^{L+1}])_{k,m}(x)
\end{equation}
where $(\mathcal{F}^{-1}[\widehat{\psi}/(\cdot)^{L+1}])_{k,m}(x) =2^k (\mathcal{F}^{-1}[\widehat{\psi}/(\cdot)^{L+1}])(2^k x - m) $. By Lemma \ref{support}, the estimate of $|\kappa^{L+1}_{j}*\psi_{k,m}(x)|$ becomes the following:
\begin{eqnarray}
    |\kappa^{L+1}_{j}*\psi_{k,m}(x)|
    &\leq&
    \sup_{z \in \mathbb{R}} |\partial^{3(L+1)} \Psi_{L+1}(z)|
    2^{-(L+1)(j-k)}
    2^j 2^{-k/2}
    \int_{c_{L} Q_j} 
    |(\mathcal{F}^{-1}[\widehat{\psi}/(\cdot)^{L+1}])_{k,m}(x-y)|dy \nonumber \\
    &\leq&
    \sup_{z \in \mathbb{R}} |\mathcal{F}^{-1}[\widehat{\psi}/(\cdot)^{L+1}](z)|
    \sup_{z \in \mathbb{R}} |\partial^{3(L+1)}\Psi_{L+1}(z)|
    2^{-(L+1)(j-k)}
    2^j 2^{k/2}
    \int_{c_{L} Q_j} 
    \chi_{k,m}((x-y)/\gamma) dy \nonumber 
\end{eqnarray}
\end{proof}

\subsection{Proof of  Theorem \ref{L2} and Theorem \ref{analy}}
\label{last_appendix}
\begin{proof}[Proof of Theorem \ref{L2}]

Assume $e^{b|x|}f \in L^2(\mathbb{R})$ for all $b<a$. 
Let $\xi,\eta \in \mathbb{R}$ and $|\eta|<a$.
Choose $b>0$ and $\epsilon>0$ such that $|\eta| < b-\epsilon < a$ and $b < a$. For all $N \in \mathbb{N}_0$, we have the following expression:
\begin{eqnarray}
\label{dominate}
    \int_{\mathbb{R}} |x|^N |e^{i(\xi+i \eta)x}f(x)|dx
    &=&
    \int_{\mathbb{R}} |x|^N |e^{-\eta x}f(x)|dx \nonumber \\
    &\leq&
    \int_{\mathbb{R}} |x|^N e^{(b -\epsilon) |x|} |f(x)|dx \nonumber \\
    &\leq&
    \left( \int_{\mathbb{R}} |x|^{2N} e^{-2\epsilon |x|} dx \right)^{1/2}
    \left( \int_{\mathbb{R}} e^{2b |x|} |f(x)|^2 dx\right)^{1/2} \nonumber \\
    &<& \infty
\end{eqnarray}
 In the third inequality, the Hölder inequality is used. Thus, if $\Im|\zeta| < a $, by the Lebesgue's dominated convergence theorem
\begin{equation}
    \partial^N \widehat{f}(\zeta) 
    =
    (2 \pi)^{-1/2}
    \int_{\mathbb{R}} (-ix)^N e^{-i \zeta x }f(x)dx
\end{equation}
for all $N \in \mathbb{N}_0$ and integral of the right side converges because of (\ref{dominate}).
This shows analyticity of $\widehat{f}(\zeta)$ on $Im|\zeta| < a$.  If $|\eta| \leq b$ 
\begin{eqnarray}
\|\widehat{f}(\cdot + i \eta) \|_{L^2}^2
&=&
\int_{\mathbb{R}} \left| \int_{\mathbb{R}} e^{-i(\xi+i \eta)x}f(x) dx \right|^2 d\xi  \nonumber \\
&=&
\int_{\mathbb{R}} \left| \int_{\mathbb{R}} e^{-i\xi x} e^{\eta x}f(x) dx \right|^2 d\xi \nonumber \\
&=&
\|\mathcal{F}[e^{\eta \cdot}f(\cdot)]\|_{L^2}^2
\nonumber \\
&=&
\|e^{\eta \cdot}f(\cdot)\|_{L^2}^2
\end{eqnarray}
Thus,
\begin{equation}
\label{inverse}
    \sup_{|\eta| \leq b} \|\widehat{f}(\cdot + i \eta) \|_{L^2}
    = \sup_{|\eta| \leq b} \|e^{\eta \cdot}f(\cdot)\|_{L^2}
    < \infty
\end{equation}
The boundedness of $\sup_{|\eta| \leq b} \|e^{\eta \cdot}f(\cdot)\|_{L^2}$ is justified by assumption $e^{b|x|}f \in L^2(\mathbb{R})$ for all $b<a$.
This result shows (\ref{supsup}).

Next, the inverse assertion is proved. We claim, for $|\eta|<a$ and $g \in C_0^{\infty}(\mathbb{R})$.
\begin{equation}
    \int_{\mathbb{R}} \overline{\widehat{g}(\xi)} \widehat{f}(\xi) d \xi
    =
    \int_{\mathbb{R}} \overline{\widehat{g}(\xi-i\eta)} \widehat{f}(\xi-i\eta) d \xi
\end{equation}
Assume supp $\: g \subset [-R,R]$. Then, $\widehat{g}(\xi)$ is analytic on $\mathbb{C}$ and, for all $N \in \mathbb{N}_0$, $C_N$ exists such that the following expression is satisfied:
\begin{equation}
\label{payley}
    |\widehat{g}(\xi)|
    \leq
    \frac{C_N e^{R |Im \xi|}}{(1+|\xi|)^{N+1}}
\end{equation}
Based on the Cauchy–Riemann equations, $\overline{\widehat{g}(\overline{\xi})}$ is also analytic and satisfies (\ref{payley}). Therefore, we apply Cauchy's integral formula to $\overline{\widehat{g}(\overline{\xi})} \widehat{f}(\xi))$ and obtain the following expression:
\begin{eqnarray}
\label{cauchy}
    && \int_{-L}^L \overline{\widehat{g}(\xi)} \widehat{f}(\xi) d \xi \nonumber \\
    &=&
    \int_{-L}^L
    \overline{\widehat{g}(\xi-i\eta)} \widehat{f}(\xi+i\eta)d \xi
    -
    \int_{0}^{\eta} \overline{\widehat{g}(L-i\eta)}\widehat{f}(L+i\eta)  i d \eta
    +
     \int_{0}^{\eta} \overline{\widehat{g}(-L-i\eta)} \widehat{f}(L+i\eta)i d \eta
\end{eqnarray}
with $a > \eta>0$.
Consider the estimate of the second and third terms in (\ref{cauchy}). Here, $\int_{0}^{\eta} |\overline{\widehat{g}(L-i\eta)}\widehat{f}(L+i\eta)  | d \eta$ is integrable with respect to $L$ by Fubini--Tonelli theorem and 
\begin{eqnarray*}
     \int_{0}^{\eta} \int_{\mathbb{R}} 
    \left|\overline{\widehat{g}(L-i\eta)}\widehat{f}(L+i\eta)\right| dL \; d\eta 
    &\leq&
    \int_{0}^{\eta} \int_{\mathbb{R}} 
    \frac{C_N e^{R |\eta|}}{(1+|L|)^{N+1}}
    \left|\widehat{f}(L+i\eta)\right| dL \; d\eta \nonumber \\
    &\leq&
    \int_{0}^{\eta} C_N e^{R |\eta|}
    \left(\int_{\mathbb{R}} 
    \frac{dL}{(1+|L|)^{2N+2}}  \right)^{1/2}
    \sup_{|y| \leq b} \|\widehat{f}(\cdot+i y) \|_{L^2}  \; d\eta  \nonumber \\
    &<& \infty
\end{eqnarray*}
The $\int_{0}^{\eta} |\overline{\widehat{g}(-L-i\eta)} \widehat{f}(L+i\eta)| d \eta$ is also integrable with respect to $L$ by the same discussion. Thus, two sequences $\{ L_n^{+} \}_{n \in \mathbb{N}_0}$ and $\{ L_n^{-} \}_{n \in \mathbb{N}_0}$ exists such that
\begin{equation}
    L_n^{+} \xrightarrow[]{n \rightarrow \infty} +\infty
    \quad \text{and} \quad
    L_n^{-} \xrightarrow[]{n \rightarrow \infty} +\infty
\end{equation}
and
\begin{equation}
    \int_{0}^{\eta} |\overline{\widehat{g}(L_n^{+}-i\eta)}\widehat{f}(L_n^{+}+i\eta)| d \eta
    \xrightarrow[]{n \rightarrow \infty} 0
    \quad \text{and} \quad
    \int_{0}^{\eta} |\overline{\widehat{g}(-L_n^{-}-i\eta)} \widehat{f}(L_n^{-}+i\eta)| d \eta
    \xrightarrow[]{n \rightarrow \infty} 0
\end{equation}
Thus, we obtain the desired equality.
\begin{equation}
\label{desired_1}
    \int_{\mathbb{R}} \overline{\widehat{g}(\xi)} \widehat{f}(\xi) d \xi
    =
    \int_{\mathbb{R}} \overline{\widehat{g}(\xi-i\eta)} \widehat{f}(\xi+i\eta) d \xi
\end{equation}
From the equality (\ref{desired_1}) and $\mathcal{F}[\overline{g}(x) e^{\eta x}](\xi-i \eta) = \widehat{g}(\xi)$, we obtain the following:
\begin{eqnarray}
\label{c_0}
    \int_{\mathbb{R}} \overline{g}(x) e^{\eta x}f(x) dx
    &=&
    \int_{\mathbb{R}} \overline{\widehat{g}(\xi)}\widehat{f}(\xi+i\eta)d\xi
    \nonumber \\
    &=&
    \int_{\mathbb{R}} \overline{g(x)} \mathcal{F}^{-1}[\widehat{f}(\cdot+i\eta)](x)dx
\end{eqnarray}
(\ref{c_0}) holds for every $g \in C_0^{\infty}(\mathbb{R})$, then $e^{\eta x}f(x)=\mathcal{F}^{-1}[\widehat{f}(\cdot+i\eta)](x)$. This leads to $e^{\eta x}f \in L^2(\mathbb{R})$ when $a > \eta > 0$. In the case $-a < \eta <0$, $e^{\eta x}f \in L^2(\mathbb{R})$. also holds. Thus, for $|\eta|<a$
\begin{equation}
    \int_{\mathbb{R}} |e^{\eta|x|}f(x)|^p dx
    \leq
    \int_{x <0} |e^{-\eta x}f(x)|^p dx
    +
    \int_{x \leq 0} |e^{\eta x}f(x)|^p dx < \infty
\end{equation}
Finally, we obtain $e^{\eta |x|}f \in L^2(\mathbb{R})$.
\end{proof}

\begin{proof}[Proof of Theorem \ref{analy}]
Assume $x\geq 0$. By the same procedure in the proof of theorem\ref{L2}, we have the following expression:
\begin{equation}
\label{desired10}
    f(x) = 
    \int_{\mathbb{R}} e^{-i \xi x} \widehat{f}(\xi) d\xi
    =\int_{\mathbb{R}} e^{-i(\xi-ib)x} \widehat{f}(\xi-ib) d\xi
\end{equation}

Thus,
\begin{eqnarray}
\label{positive}
| e^{bx} f(x) |
&\leq&
\|\widehat{f}(\cdot-ib) \|_{L^1}
\nonumber \\
\end{eqnarray}
If $x<0$, a similar estimate leads to
\begin{equation}
\label{negative}
    | e^{-bx} g(x) |
    \leq
    \left\|\widehat{f}(\cdot + ib)\right\|_{L^1}
\end{equation}
From (\ref{positive}) and (\ref{negative})
\begin{equation}
    | e^{b|x|} f(x) |
    \leq
    \max_{y=b,-b}\left\|\widehat{f}(\cdot + iy)\right\|_{L^1}
\end{equation}
Thus, the desired inequality follows when $C_b := \max_{y=b,-b}\left\|\widehat{f}(\cdot + iy)\right\|_{L^1}$, which is bounded by condition (ii).
\end{proof}

\noindent\textbf{Acknowledgments:} This work was supported by the World-leading INnovative Graduate Study Program for Frontiers of Mathematical Sciences and Physics, The University of Tokyo.

\bibliographystyle{unsrt}
\bibliography{sample}

\end{document}